\documentclass{amsart}

\usepackage{amsmath,amssymb,amsthm,graphicx,pstricks,amscd,amsfonts,latexsym}
\usepackage{xcolor}
\usepackage[all]{xy}
\usepackage[normalem]{ulem}
\usepackage{lscape}
\usepackage{hyperref}

\usepackage{graphicx}
\usepackage{caption} 
\usepackage{float}
\usepackage[utf8]{inputenc}
\usepackage{cleveref}

\usepackage{tikz,tikz-cd}
\usetikzlibrary{snakes}
\usetikzlibrary{shapes.geometric,positioning}
\usetikzlibrary{arrows,decorations.pathmorphing,decorations.pathreplacing}
\usepackage{quiver}

\usepackage{ifthen} 
\newcounter{pos} 
\tikzset{									
	initcounter/.code={\setcounter{pos}{0}},
	style between/.style n args={3}{
		postaction={
			initcounter,
			decorate,
			decoration={
				show path construction,
				curveto code={
					\addtocounter{pos}{1}
					\pgfmathtruncatemacro{\min}{#1 - 1}
					\ifthenelse{\thepos < #2 \AND \thepos > \min}{
						\draw[#3]
						(\tikzinputsegmentfirst)
						..
						controls (\tikzinputsegmentsupporta) and (\tikzinputsegmentsupportb)
						..
						(\tikzinputsegmentlast);
					}{}
				}
			}
		},
	},
}

\theoremstyle{plain} 
\newtheorem{theorem}{Theorem}[section]
\newtheorem{lemma}[theorem]{Lemma}
\newtheorem{proposition}[theorem]{Proposition}
\newtheorem{corollary}[theorem]{Corollary}

\theoremstyle{definition} 
\newtheorem{example}[theorem]{Example}

\newtheorem{definition}[theorem]{Definition}
\newtheorem{remark}[theorem]{Remark}

\newtheorem{notation}[theorem]{Notation}

\newcommand{\End}{\operatorname{End}}

\newcommand{\Hom}{\operatorname{Hom}}

\renewcommand{\H}{\operatorname{H}}

\newcommand{\per}{\operatorname{per}}
\newcommand{\add}{\operatorname{add}}

\newcommand{\Obj}{\operatorname{Obj}}
\newcommand{\Tw}{\operatorname{Tw}}

\newcommand{\cA}{\mathcal{A}}
\newcommand{\cB}{\mathcal{B}}
\newcommand{\cC}{\mathcal{C}}
\newcommand{\cD}{\mathcal{D}}

\newcommand{\cF}{\mathcal{F}}

\newcommand{\bD}{\mathbb{D}}
\newcommand{\bR}{\mathbb{R}}
\newcommand{\bM}{\mathbb{M}}

\newcommand{\bP}{\mathbb{P}}

\newcommand{\bS}{\mathbb{S}}
\newcommand{\bX}{\mathbb{X}}
\newcommand{\bZ}{\mathbb{Z}}

\newcommand{\surf}{\mathcal{S}}
\newcommand{\grad}{\mathfrak{f}}

\definecolor{dark-green}{RGB}{14,150,2}

\usepackage{todonotes}

\begin{document}
\title{Skew-group~$A_{\infty}$-categories as Fukaya categories of orbifolds}

\author{Claire Amiot}
\address[Claire Amiot]{Institut Fourier, 100 rue des maths, 38402 Saint Martin d'H\`eres, Université Grenoble Alpes, Institut Universitaire de France (IUF)}
\email{\href{mailto:claire.amiot@univ-grenoble-alpes.fr}{claire.amiot@univ-grenoble-alpes.fr}}

\author{Pierre-Guy Plamondon}
\address[Pierre-Guy Plamondon]
{Laboratoire de Mathématiques de Versailles, UVSQ, CNRS, Université Paris-Saclay, Institut universitaire de France (IUF)}
\email{\href{mailto:pierre-guy.plamondon@uvsq.fr}{pierre-guy.plamondon@uvsq.fr}}

\thanks{}

\date{\today}


\begin{abstract}
We study the partially wrapped Fukaya category of a surface with boundary with an action of a group of order two.  We let the notion of a skew-group~$A_\infty$-category play the role of the partially wrapped Fukaya category of an orbifold surface.  We classify indecomposable objects in terms of graded curves with signs, or taggings, at orbifold points.  We compute morphisms between a class of objects, and we use this to describe tilting objects and find algebras derived-equivalent to skew-gentle algebras.
\end{abstract}

\maketitle

\tableofcontents

\section{Introduction}

Our main interest in this paper is the partially wrapped (or topological) Fukaya category of a surface with boundary \cite{HaidenKatzarkovKontsevich} with an action of a group of order two.   We use the notion of a skew-group~$A_\infty$-category \cite{OpperZvonareva}, which is an enhancement of the notion of a skew-group algebra as studied in \cite{ReitenRiedtmann}, as a way to construct the partially wrapped Fukaya category of an orbifold surface.  We use this to describe algebras derived-equivalent to skew-gentle algebras.

The partially wrapped Fukaya category of a graded surface with marked boundary is an~$A_\infty$-category defined in \cite{HaidenKatzarkovKontsevich}, where it was seen that dissections of the surface gave rise to formal generators.  The endomorphism algebras of these generators, in cohomology, are isomorphic to graded gentle algebras, and any such algebra arises from a dissection \cite{OpperPlamondonSchroll,BaurCoelho,PaluPilaudPlamondon2}.  This point of view has been fruitful in providing new results on derived equivalences of gentle algebras \cite{LekiliPolishchuk,AmiotPlamondonSchroll,Opper,ChangJinSchroll,JinSchrollWang}.  Related geometric models have also been used to study gentle algebras \cite{DavidRoeslerSchiffler,CanakciPauksztelloSchroll,LiuZhang,BaurSchroll,BaurLaking,ChangSchroll}, Brauer graph algebras \cite{OpperZvonareva} and skew-gentle algebras \cite{AmiotBrustle,LabardiniSchrollValdivieso,HeZhouZhu,Amiot23}.

Skew-gentle algebras were introduced by Geiss and de la Peña in \cite{GeissdelaPena} and are a special case of clannish algebras \cite{CrawleyBoevey89}; however, when working over a field of characteristic different from~$2$, skew-gentle algebras are conveniently described (up to Morita-equivalence) as skew-group algebras of gentle algebras with an action of a group of order two on their Gabriel quiver.  The group action lifts to an action on the graded surface with dissection giving rise to the gentle algebra, and thus to the partially wrapped Fukaya category.

For an~$A_\infty$-category~$\cA$ with a strict action of a finite group~$G$, we define the skew-group~$A_\infty$-category~$\cA*G$ as in \cite{OpperZvonareva}, see also \cite{ReitenRiedtmann} for usual categories (Definition~\ref{defi::skew-group-category}).

The philosophy put forward in this paper is that, if~$\cA$ comes from a~$G$-invariant arc system on the graded surface~$\surf$, then the skew-group~$A_\infty$-category~$\cA*G$, or rather the split-closure of the category~$\Tw(\cA*G)$ of twisted complexes, plays the role of the partially wrapped Fukaya category of the graded orbifold surface~$\surf/G$. 

This necessitates an understanding of the split-closure~$\Tw(\cA*G)_+$ (Section~\ref{sect::split-closure-skew-group}).  In the case where~$\cA$ is a finite-dimensional algebra~$\Lambda$, we obtain in this way an equivalence of triangulated categories between~$\big( (\per\Lambda)*G \big)_+$ and~$\per(\Lambda * G)$ (Corollary~\ref{coro::split-closure-triangulated}).  

We apply this to the partially wrapped Fukaya category of a graded surface~$\surf$ with marked boundary endowed with an action of a group~$G$ of order two.  In this case,~$\Tw(\cA*G)_+$ is equivalent (in cohomology) to the perfect derived category of a graded skew-gentle algebra.  The natural functor
\[
 \Tw\cA \to \Tw(\cA*G)_+
\]
sends indecomposable objects either to indecomposable objects or to direct sums of two distinct indecomposable objects; this allows us to give a description of all indecomposable objects in this category in terms of curves on~$\surf$ together with \emph{taggings} (Section~\ref{Section4}).   In Section~\ref{Section5}, we turn this into a description in terms of \emph{tagged graded arcs} on the orbifold surface~$\surf/G$.  By the above functor, objects corresponding to arcs are sent to objects corresponding to tagged arcs, but interestingly, objects corresponding to a closed curve with paramater are sent either to closed curves or to tagged arcs joining two orbifold points, depending on the parameter. This was already noted in \cite{Amiot23}, and is similar to a phenomenon observed in~\cite{AmiotPlamondon} in the context of cluster categories from triangulated surfaces.

One interesting and useful facet of the construction of the Fukaya category developped in \cite{HaidenKatzarkovKontsevich} is the fact that given two full systems of graded arcs on the surface $\surf$, there is a canonical Morita equivalence linking the corresponding Fukaya categories. This permits in particular when studying an object to choose a suitable full system of arcs in which the description of this object as twisted complexes becomes simple. One central point studied in Section \ref{Section4} is to understand how the tagging behaves when changing the system. This careful checking is of importance when computing the morphisms spaces between tagged arcs interescting in an orbifold point.

This allows us to describe tilting objects in the derived category of a graded skew-gentle algebra in terms of dissections of the orbifold surface~$\surf/G$ (Thoerem~\ref{theo::local-conditions}).  This gives a class of algebras (Definition~\ref{defi::algebra-of-dissection}) which are derived-equivalent to skew-gentle algebras.  In the case of a disk with one orbifold point, this yields in Section~\ref{Section6} a complete list of the algebras piecewise hereditary of Dynkin type~$\mathbf{D}_n$ (see also~\cite{Keller91}).  

As stated above, the philosophy of this paper is that the skew-group~$A_\infty$-category~$\cA*G$, or rather the category~$\Tw(\cA*G)_+$, plays the role of the partially wrapped Fukaya category of the graded orbifold surface~$\surf/G$.  Recently, another approach has appeared in~\cite{ChoKim}, in which the authors define an~$A_\infty$-structure directly on a tagged arc system on~$\surf/G$, and relate it to the~$G$-invariant portion of~$\Tw\cA$.  Our approach differs from theirs in that we define the skew-group~$A_\infty$-algebra~$\cA*G$, which should be seen as an~$A_\infty$-enhancement of the algebra~$\Lambda*G$, with~$\Lambda$ a graded gentle algebra.  It should be noted that~$\Lambda*G$ is not \emph{basic}, and so is not given by a quiver with relations; however, it has the advantage of being canonical -- no choice of idempotent is involved in its description.  It would be of great interest to compare more carefully the two points of view.

Related work of Barmeier, Schroll and Wang on Fukaya categories of orbifold surfaces can be found in~\cite{BWZ}.

The paper is organized as follows.  In Section~\ref{Section2}, we recall the main defnition of~$A_\infty$ structures and their split closures, and introduce skew-group~$A_\infty$-categories.  In Section~\ref{Section3}, we recall the main results on topological Fukaya categories of graded surfaces, and in particular the description of their indecomposable objects in terms of graded curves.  In Section~\ref{Section4}, we describe some indecomposable objects in~$\Tw(\cA*G)_+$ and the morphisms between them.  This allows us to obtain tilting objects in Section~\ref{Section5} in terms of dissection of the orbifold~$\surf/G$.  Finally, in Section~\ref{Section6}, we apply this to provide a complete list of piecewise hereditary algebras of Dynkin type~$\mathbf{D}$.

\section{Skew-group $A_{\infty}$-categories}\label{Section2}

\subsection{$A_{\infty}$-categories}

The aim of this subsection is mainly to fix the notations we will be using throughout the paper.  Our main references are \cite[Chapter 1]{Seidel} and \cite[Section 3]{HaidenKatzarkovKontsevich}.

We fix an algebraically closed field~$K$.

\begin{definition}
 An \emph{$A_{\infty}$-category}~$\cA$ is given by the following data.
 \begin{itemize}
  \item A set of objects~$\Obj(\cA)$.
  
  \item For each pair of objects~$X$ and~$Y$ of~$\cA$, a~$\bZ$-graded vector space of morphisms $\Hom_{\cA}(X,Y) = \bigoplus_{m\in \bZ} \Hom_{\cA}^m(X,Y)$.  The degree of a homogeneous morphism~$f$ is denoted by~$|f|$ and its \emph{reduced degree} is~$\| f \| := |f| -1$.
  
  \item For each integer~$n\geq 1$ and each~$(n+1)$-tuple of objects~$X_0,X_1,\ldots, X_n$, a morphism of degree~$2-n$
  \[
   \mu_{\cA}^n : \Hom_{\cA}(X_{n-1},X_n) \otimes \cdots \otimes \Hom_{\cA}(X_{0},X_1) \longrightarrow \Hom_{\cA}(X_{0},X_n)
  \]
  such that the \emph{Stasheff relations} are satisfied for all~$n$-tuples of homogeneous morphisms~$(a_n, \ldots, a_1)$:
  \[
    \sum_{\stackrel{k\geq 0,j\geq 1}{k+j\leq n}} (-1)^{\sum_{i=1}^k \| a_i \|}  \mu_{\cA}^{n-j+1}\big(a_n, \ldots, a_{k+j+1},\mu_{\cA}^{j}(a_{k+j},\ldots,a_{k+1}),a_k,\ldots,a_1\big).
  \]
  The~$\mu_{\cA}^n$ are called the \emph{higher multiplications} of~$\cA$.
 \end{itemize}
 
 An $A_{\infty}$-category~$\cA$ is \emph{strictly unital} if for every object~$X$ of~$\cA$, there is a morphism~$1_X \in \Hom_{\cA}^0(X,X)$ such that
 \begin{itemize}
  \item $\mu_{\cA}^1(1_X) = 0$;
  \item $\mu_{\cA}^2(a,1_X) = f$ and~$\mu_{\cA}^2(1_Y,a) = (-1)^{|a|}f$ for all homogeneous morphism~$a$ in~$\Hom_{\cA}(X,Y)$;
  \item $\mu_{\cA}^n (\ldots, 1_X, \ldots) = 0$ if~$n\geq 3$.
 \end{itemize}
\end{definition}

We will often write~$\mu^n$ instead of~$\mu_{\cA}^n$.  We will also sometimes denote~$\mu^n$ simply by~$\mu$, since the~$n$ in~$\mu^n$ is implied when writing~$\mu(a_n,\ldots,a_1)$.

\subsection{Twisted complexes}

Our main interest in~$A_{\infty}$-categories lies in their use as enhancement of~$K$-linear or triangulated categories.  

\begin{definition}
 Let~$A$ be a strictly unital~$A_{\infty}$-category.  Its \emph{cohomological category}~$H^*\cA$ is the (graded) category whose objects are those of~$\cA$ and whose morphism graded spaces are defined as
 \[
  \Hom_{H^*\cA}(X,Y) = H^*\Hom_{\cA}(X,Y),
 \]
 where~$\Hom_{\cA}(X,Y)$ is viewed as a complex with differential~$\mu_{\cA}^1$, and where composition of homogeneous morphisms is defined by
 \[
  [a_2][a_1] := (-1)^{|a_1|}\mu_{\cA}^2(a_2,a_1).
 \]
 We denote by~$H^0\cA$ the subcategory obtained by keeping only morphisms of degree~$0$.
\end{definition}

The main way triangulated categories will arise in this paper is by taking twisted complexes over an~$A_{\infty}$ category.  We define them in several steps.

\begin{definition}
 Let~$\cA$ be an~$A_{\infty}$-category.  The~$A_{\infty}$-category~$\add\bZ\cA$ is defined as follows.
 \begin{itemize}
  \item Its objects are formal direct sums of the form~$\bigoplus_{X\in\Obj(\cA)} V_X\otimes X$, where~$V_X$ is a finite-dimensional graded vector space which is zero for all but finitely many~$X$.
  
  \item Morphism spaces are defined by
    \[
     \Hom_{\add\bZ\cA}(V_X\otimes X, V_Y\otimes Y) := \Hom(V_X,V_Y)\otimes_K \Hom_{\cA}(X,Y),
    \]
    extended to direct sums by additivity.
    
  \item The higher multiplications are defined on homogeneous morphisms by
  \[
   \mu_{\add\bZ\cA}^n(\phi_n\otimes a_n, \ldots, \phi_1\otimes a_1) = (-1)^{\sum_{i<j} |\phi_i| \| a_j \| }  \phi_n\cdots\phi_1 \otimes \mu_{\cA}^n(a_n, \ldots, a_1).
  \]
 \end{itemize}
\end{definition}

\begin{notation}
\begin{enumerate}
\item If~$\bigoplus_{X\in\Obj(\cA)} V_X\otimes X$ and~$\bigoplus_{X\in\Obj(\cA)} W_X\otimes X$ are two objects of~$\add\bZ\cA$, we will denote their \emph{direct sum} by
\[
\left( \bigoplus_{X\in\Obj(\cA)} V_X\otimes X \right) \oplus \left( \bigoplus_{X\in\Obj(\cA)} W_X\otimes X \right) := \bigoplus_{X\in\Obj(\cA)} (V_X\oplus W_X)\otimes X.
\]

\item 
When writing an object of~$\add\bZ\cA$, we usually omit the summands for which~$V_X$ is the zero space.

\item
 If we view the field~$K$ as a graded vector space over itself, then for any integer~$m$, the identity map from~$K$ to itself induces a morphism of degree~$-m$ from~$K$ to~$K[m]$.  We will denote this morphism by~$s^m\in \Hom^{-m}(K,K[m])$.  By a slight abuse of notation, we will also denote by~$s^m$ the induced morphism in~$\Hom^{-m}(K[d], K[d+m])$ for any integer~$d$.
 
\item
 If~$X$ is an object of~$\cA$, then~$X[m]$ stands for the object~$K[m]\otimes X$ of~$\add\bZ\cA$.  Thus, a morphism from~$X[m]$ to~$Y[n]$ in~$\add\bZ\cA$ is a scalar multiple of a morphism of the form~$s^{n-m}\otimes a$, with~$a\in\Hom_{\cA}(X,Y)$.
\end{enumerate}
\end{notation}

\begin{definition}
 Let~$\cA$ be an~$A_{\infty}$ category.  A \emph{twisted complex} over~$\cA$ is a pair~$(W,\delta)$, where~$W$ is an object of~$\add\bZ\cA$ and~$\delta\in\Hom_{\add\bZ\cA}^1(W,W)$ such that
 \begin{itemize}
  \item there is a direct sum decomposition~$W = \bigoplus_{i=1}^m W_i$ such that~$\delta$ is strictly upper triangular;
  \item $\sum_{k\geq 1} \mu_{\add\bZ\cA}^k (\delta, \delta, \ldots, \delta) = 0$.
 \end{itemize}

 The~$A_{\infty}$-category~$\Tw\cA$ has for objects all twisted complexes over~$\cA$ with morphism spaces defined by
 \[
  \Hom_{\Tw\cA} \big( (W_1,\delta_1), (W_2,\delta_2) \big) := \Hom_{\add\bZ\cA}(W_1,W_2)
 \]
 and higher multiplications defined on~\[a_n \otimes \cdots \otimes a_1\in\Hom_{\Tw\cA}\big( (W_{n-1},\delta_{n-1}),(W_n,\delta_n) \big) \otimes \cdots \otimes \Hom_{\Tw\cA}\big( (W_{0},\delta_{0}),(W_1,\delta_1) \big)\] by
 \begin{eqnarray*}
  \mu_{\Tw\cA}^n(a_n,\ldots,a_1) &:=& \\
   && \hspace{-2.5cm}\sum \mu_{\add\bZ\cA}(\delta_n,\ldots,\delta_n, a_n, \delta_{n-1},\ldots,\delta_{n-1}, a_{n-1}, \delta_{n-2},\ldots,\delta_{n-2},a_{n-2} \ldots, \delta_{1},\ldots,\delta_{1}, a_1, \delta_{0},\ldots,\delta_{0}),
 \end{eqnarray*}
 where the sum is taken over all possible ways of inserting a finite number of copies of the~$\delta_i$ between the~$a_j$. 
\end{definition}

\begin{notation}
 \begin{enumerate}
  \item We denote the \emph{direct sum} of two twisted complexes as
  \[
   (W_1,\delta_1)\oplus (W_2,\delta_2) = (W_1\oplus W_2, \delta_1\oplus \delta_2).
  \]
  
  \item We represent a twisted complex~$(W,\delta)$ as follows.  If~$W = \bigoplus_{i=1}^n X_i[d_i]$ is a decomposition of~$W$ for which~$\delta$ is strictly upper triangular, then we depict~$(W,\delta)$ as
\[\begin{tikzcd}
	{X_n[d_n]} & \cdots & {X_3[d_3]} & {X_2[d_2]} & {X_1[d_1]}
	\arrow[from=1-1, to=1-2]
	\arrow["{\delta_{34}}", from=1-2, to=1-3]
	\arrow["{\delta_{23}}", from=1-3, to=1-4]
	\arrow["{\delta_{12}}", from=1-4, to=1-5]
	\arrow["{\delta_{3n}}"', curve={height=6pt}, from=1-1, to=1-3]
	\arrow["{\delta_{2n}}"', curve={height=18pt}, from=1-1, to=1-4]
	\arrow["{\delta_{1n}}"', curve={height=30pt}, from=1-1, to=1-5]
	\arrow["{\delta_{13}}"', curve={height=6pt}, from=1-3, to=1-5]
\end{tikzcd}\]
where we will omit all arrows for which the corresponding~$\delta_{ij}$ is zero.

 \end{enumerate}

\end{notation}

The notion of mapping cones of morphisms of twisted complexes allows for a definition of a triangulated structure on~$H^0\Tw\cA$.  We say that a strictly unital~$A_{\infty}$-category~$A$ is a \emph{triangulated~$A_{\infty}$-category} when these mapping cones confer to~$H^0\cA$ a structure of a triangulated category.  We will not need the technicalities, so we will simply recall the main results and refer the reader to~\cite[Sections 3h to 3p]{Seidel}.

\begin{theorem}[e.g. Lemma 3.28 of \cite{Seidel}]
 If~$\cA$ is a strictly unital~$A_{\infty}$-category, then~$\Tw\cA$ is triangulated.
\end{theorem}

\subsection{$A_{\infty}$-functors}
The only~$A_{\infty}$-functors we will need in this paper are \emph{strict~$A_{\infty}$-functors}.
\begin{definition}
 Let~$\cA$ and~$\cB$ be two~$A_{\infty}$-categories.  A \emph{strict~$A_{\infty}$-functor}~$F:\cA\to \cB$ is given by
 \begin{itemize}
  \item a map~$F:\Obj\cA \to \Obj\cB$;
  \item for each pair of objects~$X,Y$ of~$\cA$, a morphism of~$K$-vector spaces
  \[
   F^1: \Hom_{\cA}(X,Y) \to \Hom_{\cB}(FX,FY)
  \]
  such that, for all integers~$n\geq 1$ and all morphisms~$a_1,\ldots,a_n$ in~$\cA$,
  \[
   \mu^n_{\cB}(Fa_n,\ldots,Fa_1) = F\mu^n_{\cA}(a_n,\ldots,a_1).
  \]
 \end{itemize}
 If~$\cA$ and~$\cB$ are strictly unital, then~$F$ is \emph{strictly unital} if for each object~$X$ of~$\cA$, we have that~$F(1_X) = 1_{FX}$.
\end{definition}
For a strict~$A_{\infty}$-functor~$F$, we will write~$F$ instead of~$F^1$, as this will not lead to any confusion.

A strict~$A_{\infty}$-functor~$F:\cA\to \cB$ induces a strict~$A_{\infty}$-functor~$\Tw\cA \to \Tw\cB$, which we will denote by~$\Tw F$.  It also induces functors in cohomology~$H^0F:H^0\cA\to H^0\cB$ and~$H^0\Tw F:H^0\Tw\cA\to H^0\Tw\cB$, the former being~$K$-linear and the latter being triangulated.

\begin{definition}
 
 A strictly unital strict~$A_{\infty}$-functor~$F:\cA\to \cB$ is 
 \begin{itemize}
  \item a \emph{quasi-equivalence} if the induced functor~$H^0$ is an equivalence;
  \item a \emph{Morita equivalence} if the induced~$A_{\infty}$-functor~$\Tw F:\Tw\cA\to \Tw\cB$ is a quasi-equivalence.
 \end{itemize} 
\end{definition}

A useful fact is that a strictly unital~$A_{\infty}$-category~$\cA$ is triangulated if and only if the embedding~$I_{\cA}:\cA\to \Tw\cA$ is a quasi-equivalence~\cite[Lemma 3.33]{Seidel}.

\subsection{Skew-group $A_{\infty}$-categories}

 Let~$G$ be a finite group whose order is not divisible by the characteristic of the field~$K$.  Let~$\cA$ be a strictly unital~$A_{\infty}$-category.  We are interested in actions of~$G$ on~$\cA$ by strictly unital strict~$A_{\infty}$-functors.

\begin{definition}
 A \emph{strict~$G$-action on~$\cA$} is a group morphism~$\rho$ from~$G$ to the group of strictly unital strict invertible~$A_{\infty}$-functors.
\end{definition}
By a slight abuse of notation, we will usually write~$g$ instead of~$\rho(g)$ when dealing with group actions.  For instance, if~$X$ is an object of~$\cA$, then we write~$gX$ instead of~$\rho(g)(X)$. 

The main object of this section is that of a \emph{skew-group~$A_{\infty}$-category}.  It appears in \cite{OpperZvonareva}, and is reminiscent of the definition of a skew-group category from~\cite{ReitenRiedtmann}.

\begin{definition}\label{defi::skew-group-category}
 Given a strict~$G$-action on~$\cA$, its \emph{skew-group~$A_{\infty}$-category}~$\cA * G$ is defined as follows.
 \begin{itemize}
  \item Its objects are symbols of the form~$\{X\}$ for~$X\in \Obj\cA$.
  \item Morphism spaces have the form~$$\Hom_{\cA*G}(\{X\},\{Y\}) = \bigoplus_{g\in G}\Hom_{\cA}(gX,Y).$$
  If~$a\in \Hom_{\cA}(gX,Y)$, then we denote the corresponding morphism from~$\{X\}$ to~$\{Y\}$ by~$a\otimes g$.
  
  \item For any integer~$n$, the higher multiplication~$\mu^n_{\cA*G}$ is defined by
  \[
   \mu^n_{\cA*G}(a_n\otimes g_n, \ldots, a_1\otimes g_1) = \mu^n_{\cA}(a_n, g_n\cdot a_{n-1}, g_ng_{n-1}\cdot a_{n-2}, \ldots, g_n\cdots g_3g_2\cdot a_1)\otimes g_n\cdots g_1
  \]
  and extended to all morphisms by multilinearity.  
 \end{itemize}
\end{definition}

\begin{proposition}
 If~$\cA$ is a strictly unital with a strict~$G$-action, then the skew-group category~$\cA*G$ is a strictly unital~$A_{\infty}$-category.
\end{proposition}

The construction of a skew-group~$A_{\infty}$-category extends to~$A_{\infty}$-functors.

\begin{proposition}
 Let~$\cA$ and~$\cB$ be strictly unital~$A_\infty$-categories with a strict~$G$-action, and let~$\Phi:\cA\to \cB$ be a strict~$A_{\infty}$-functor which commutes with the action of~$G$. Then the following data defines a strict~$A_{\infty}$-functor~$\Phi*G:\cA*G\to \cB*G$:
 \begin{itemize}
  \item For any object~$\{X\}$ of~$\cA*G$, let~$(\Phi*G)(\{X\}):= \{\Phi X\}$.
  \item For any morphism~$a\otimes g$ of~$\cA*G$, let~$(\Phi*G)(a\otimes g):= \Phi a\otimes g$.
 \end{itemize}
\end{proposition}

\begin{proposition}\label{prop::comm-square}
 The following is a commutative diagram of strict~$A_{\infty}$-functors which are all Morita equivalences.
\[\begin{tikzcd}
	\cA*G & (\Tw\cA) * G \\
	\Tw(\cA*G) & \Tw\big((\Tw\cA)*G\big)
	\arrow["I_{\cA}*G", hook, from=1-1, to=1-2]
	\arrow["I_{\cA*G}"', hook, from=1-1, to=2-1]
	\arrow["\Tw\big(I_{\cA}*G\big)"', hook, from=2-1, to=2-2]
	\arrow["I_{(\Tw\cA)*G}", hook, from=1-2, to=2-2]
\end{tikzcd}\]
In particular, the bottom arrow is a quasi-equivalence.
\end{proposition}
\begin{proof}
 Commutativity of the diagram follows from the fact that for any strict~$A_\infty$-functor~$\Phi:\cB\to \cC$, the diagram
 \[\begin{tikzcd}
	\cB & \cC \\
	\Tw\cB & \Tw\cC
	\arrow["\Phi",  from=1-1, to=1-2]
	\arrow["I_{\cB}"', hook, from=1-1, to=2-1]
	\arrow["\Tw(\Phi)"',  from=2-1, to=2-2]
	\arrow["I_{\cC}", hook, from=1-2, to=2-2]
\end{tikzcd}\]
is commutative. 
 Moreover,~$I_{\cB}$ is a Morita equivalence for any~$\cB$, so the two vertical arrows are Morita equivalences.
 
 Let us prove that the bottom arrow is a Morita equivalence.  Since the two bottom~$A_{\infty}$-categories are triangulated, the bottom functor is a Morita equivalence if and only if it is a quasi-equivalence.  Moreover,~
 \[
  H^0\Tw\big(I_{\cA}*G\big):H^0\Tw(\cA*G)\to H^0\Tw\big((\Tw\cA)*G\big)
 \]
 is a triangulated functor; we need to prove that it is an equivalence.  It is sufficient to prove that its image contains a generator of the target category.  Such a generator is given by (the~$H^0$ of) the image of~$I_{(\Tw\cA)*G}$.  
 
 An object of~$(\Tw\cA)*G$ is a direct sum of objects of the form~$\{(X,\delta)\}$, with~$(X,\delta)$ in~$\Tw\cA$.  The image of this object by~$I_{(\Tw\cA)*G}$ is~$\big(\{(X,\delta)\},0\big)$.  
 
 Now, the image of the object~$(\{X\},\delta\otimes 1)$ by the bottom functor is the object~$\big(\{(X,0)\}, \delta\otimes 1\big)$, where~$\delta$ in the right hand side is viewed as an endomorphism of~$(X,0)$.  If we prove that there is an isomorphism in cohomology between~$(\{(X,0)\},\delta\otimes 1)$ and~$\big(\{(X,\delta)\},0\big)$, then we are done proving that the bottom arrow is a quasi-equivalence.  The morphism~$1_X\otimes 1_G$ provides this isomorphism.
 
 Lastly, the top arrow is a Morita equivalence because the three other arrows are.
\end{proof}

\begin{remark}
 In~\cref{prop::comm-square}, it is unclear whether the right arrow should be a quasi-equivalence.  In the next section, we will see that it will become one if we pass to split-closures.
\end{remark}

\subsection{Split-closure of~$A_{\infty}$-categories}
In this section, we rely heavily on \cite[(4c)]{Seidel}.  Recall that a~$K$-linear category is \emph{split-closed} if all its idempotents split.  Any~$K$-linear category~$\cC$ admits a \emph{split-closure}, that is, a~$K$-linear fully faithful functor~$F:\cC\to \cD$ with~$\cD$ split-closed and any object of~$\cD$ is a direct summand of an object in the image of~$F$.  

We will need an~$A_{\infty}$ version of split-closure.

\begin{definition}[(4c) of \cite{Seidel}]
 A strictly unital~$A_{\infty}$-category~$\cA$ is \emph{split-closed} if the~$K$-linear category~$H^0\cA$ is split-closed.  A~\emph{split-closure} of~$\cA$ is an~$A_{\infty}$-functor~$F:\cA\to \cB$ such that~$H^0F$ is a split-closure of~$H^0\cA$.
\end{definition}

We will be using a strong result of existence and uniqueness of split-closures.

\begin{lemma}\label{lemm::existence-split-closure}
 Let~$\cA$ be a strictly unital~$A_{\infty}$-category.
 \begin{enumerate}
  \item \cite[Lemma 4.7]{Seidel} There exists a split closure~$F:\cA\to \cB$ of~$\cA$.  Moreover, if~$F':\cA\to \cB'$ is another split closure, then there exists a quasi-equivalence~$H:\cB\to \cB'$ such that~$F\circ H$ is isomorphic to~$F'$ in cohomology.
  \item \cite[Remark 4.10]{Seidel} There exists a split closure~$F:\cA\to \cA_+$ such that~$F$ is a strict~$A_{\infty}$-functor and~$F^1:\Hom_{\cA}(X,Y)\to \Hom_{\cA_+}(FX,FY)$ is an isomorphism for all~$X,Y$.  In other words,~$\cA$ is identified with a full subcategory of~$\cA_+$.
 \end{enumerate}
\end{lemma}

\begin{definition}
 Let~$\cA$ be a strictly unital~$A_{\infty}$-category.  We will call the split-closure~$\cA\to \cA_+$ of \cref{lemm::existence-split-closure}(2) the \emph{faithful split closure of~$\cA$}.
\end{definition}
In this paper, whenever we will need a split-closure of~$\cA$, we will choose the faithful split-closure~$\cA_+$.

\begin{notation}
 For an object~$X$ of~$\cA$, any idempotent~$e$ of~$X$ (in cohomology) defines a direct factor of~$X$.  We will denote by~$(X,e)$ an object of~$\cA_+$ corresponding to this direct factor.
\end{notation}

\begin{lemma}\label{lemm::+induit}
 If~$\Phi:\cA\to \cB$ is a strict~$A_{\infty}$-functor between two strictly unital categories, then there is a strict~$A_{\infty}$-functor~$\Phi_+:\cA_+\to \cB_+$ such that the following diagram commutes.
 \[
  \begin{tikzcd}
   \cA \ar{r}{\Phi}\ar{d} & \cB\ar{d} \\
   \cA_+ \ar{r}{\Phi_+} & \cB_+ \\
  \end{tikzcd}
 \]

\end{lemma}
\begin{proof}
 In \cite[Remark 4.10]{Seidel},~$\cA_+$ is constructed as a full subcategory of~$\Tw^-\cA$, the~$A_{\infty}$-category of twisted complexes ``unbounded on the left''.  It suffices to check that~$\Phi$ induces a strict~$A_{\infty}$-functor~$\Tw^-\Phi:\Tw^-\cA\to \Tw^-\cB$ which sends objects of~$\cA_+$ into~$\cB_+$.  We then take~$\Phi_+$ to be the restriction of this functor to domain~$\cA_+$ and codomain~$\cB_+$.
\end{proof}

\begin{lemma}\label{lemm::split-twist-is-twist-split}
 We have that~$\Tw(\cA)_+$ and~$\Tw(\cA_+)_+$ are quasi-equivalent.
\end{lemma}
\begin{proof}
 The inclusion functor~$\cA\to \cA_+$ induces a functor~$\Tw(\cA) \to \Tw(\cA_+)$, and thus by Lemma~\ref{lemm::+induit} we have a functor~$\Tw(\cA)_+ \to \Tw(\cA_+)_+$, which is fully faithful in cohomology.  To prove that it is an equivalence, it suffices to prove that all objects on the right-hand side are direct factors of objects in the image of the functor.  In fact it suffices to prove that this is true for all objects of~$\Tw(\cA_+)$.  Let~$( W, \delta)\in\Tw(\cA_+)$, with~$W\in \add \bZ\cA_+$ and~$\delta\in \Hom_{\cA_+}^1(W,W)$.  There exists a~$W'\in\add\bZ\cA_+$ such that~$W\oplus W' \in \add\bZ\cA$.  Then~$(W,\delta)$ is a direct factor of the twisted complex~$(W\oplus W', \delta\oplus 0)$, which is in~$\Tw(\cA)$.  This finishes the proof.

\end{proof}

We note that~$\Tw(\cA_+)$ might not be split-closed; situations where it known to be split-closed include the case where~$\cA$ is a non-positively graded dg algebra with vanishing differential, see~\cite[Lemm. 4.19 and Thm. 4.28]{GHW}.

The following two results are an adaptation of \cite[Lemma 2.3.1]{LeMeur} to~$A_\infty$-categories.  Recall that we assumed that the order of~$G$ is not divisible by the characteristic of~$K$.

\begin{lemma}\label{lemm::split-natural-transformations}
 Let~$\cA$ be a strictly unital~$A_{\infty}$-category with a strict~$G$-action.  Let~$F:\cA\to \cA*G$ be the strict~$A_{\infty}$-functor defined on objects by~$FX = \{X\}$ and on morphisms by~$F(f) = f\otimes 1$.  Extend it naturally to a functor~$F:\add \cA \to \add (\cA*G)$, where~$\add(-)$ denotes the~$A_{\infty}$-category obtained by taking formal direct sums.
 
 There is a strict~$A_{\infty}$-functor~$H:\add(\cA*G) \to \add\cA$ defined on objects by~$H(\{X\}) = \bigoplus_{g\in G} gX$ and on morphisms by sending~$\sum_{g\in G}a_g\otimes g : \{X\} \to \{Y\}$ to the morphism given in matrix form by $\big( ha_{h^{-1}g} \big)_{h,g\in G}$.
 
 Moreover, there are split natural transformations
 \begin{itemize}
  \item $\varepsilon: HF \to 1_{\add\cA}$ defined by projecting~$HF(X) = \bigoplus_{g\in G} gX$ onto the copy of~$X$ corresponding to~$g=1$;
  
  \item $\eta: 1_{\add(\cA*G)}\to FH$ defined by sending~$\{X\}\in\cA*G$ into~$FH(\{X\}) = \bigoplus_{g\in G} gX$ by the column vector of morphisms~$1_{gX}\otimes g$.
 \end{itemize}
\end{lemma}

\begin{lemma}\label{lemm:split2}
 The functors and natural transformations of Lemma~\ref{lemm::split-natural-transformations} induce functors~$F:\Tw\cA \to \Tw(\cA*G)$ and~$H:\Tw(\cA*G)\to \Tw\cA$ and split natural transormations~$\varepsilon: HF \to 1_{\add\cA}$ and~$\eta: 1_{\add(\cA*G)}\to FH$. 
\end{lemma}

\subsection{Split-closure of skew-group~$A_{\infty}$-categories}\label{sect::split-closure-skew-group}
We now come to the first result of this paper.

\begin{theorem}\label{theo::comm-square}
 Let~$\cA$ be a strictly unital~$A_{\infty}$-category with a strict~$G$-action.  Then the following is a commutative diagram.
\[\begin{tikzcd}[column sep = huge]
	(\cA*G)_+ & \big((\Tw\cA)*G\big)_+ \\
	\big(\Tw(\cA*G)\big)_+ & \Big(\Tw\big((\Tw\cA)*G\big)\Big)_+ \\
	\big(\Tw(\cA*G)_+\big)_+ & \Big(\Tw\big((\Tw\cA)*G\big)_+\Big)_+
	\arrow["(I_{\cA}*G)_+", hook, from=1-1, to=1-2]
	\arrow["(I_{\cA*G})_+"', hook, from=1-1, to=2-1]
	\arrow["(\Tw (I_{\cA}*G))_+"', hook, from=2-1, to=2-2]
	\arrow["(I_{(\Tw\cA)*G})_+", hook, from=1-2, to=2-2]
	\arrow[hook, from=2-1, to=3-1]
	\arrow[hook, from=2-2, to=3-2]
	\arrow[hook, from=3-1, to=3-2]
\end{tikzcd}\]
Moreover, the top-most and top-left arrows are Morita-equivalences and all other arrows are quasi-equivalences.
\end{theorem}
\begin{proof}
 By Proposition~\ref{prop::comm-square}, the upper square commutes, all its arrows are Morita-equivalences and the middle one is a quasi-equivalence.   By Lemma~\ref{lemm::+induit}, the lower square commutes and all its arrows are quasi-equivalences.

 All that remains is to prove that the top-right arrow is a quasi-equivalence, that is, an equivalence once we apply~$H^0$.  Consider the following.

 \[
  \begin{tikzcd}
   H^0((\Tw \cA)*G) \ar{dr}{L} & \\
   H^0(\Tw \cA) \ar{u}{M}\ar{r}{F} & H^0(\Tw(\cA*G))\ar{l}{H}.
  \end{tikzcd}
 \]
where $L$ is the composition of $H^0(I_{(\Tw\cA)*G})$ with the inverse of $H^0(\Tw (I_\cA*G))$, and $F$ an $H$ are as in Lemma~\ref{lemm::split-natural-transformations}. We would like to prove that $L_+$ is an equivalence, and since it is fully faithful it is enough to prove that it is dense. For that, it is enough to prove that any object of $H^0(\Tw(\cA*G))$ is isomorphic to a direct summand of an object in the image of $L$.

We have ~$F=LM$, so~$FH=LMH$, and by Lemma~\ref{lemm:split2}, for any object~$X$ of~$\Tw(\cA*G)$,~$H^0X$ is a direct factor of~$LMH(H^0X)$.  In particular, it is a direct factor of an object in the image of~$L$.  But then the same is true if we take split-closures; this shows that any object of~$\Big(\Tw\big((\Tw\cA)*G\big)\Big)_+$ is a direct factor (in cohomology) of an object in the image of the top-right arrow in the above square.  Since all categories in the upper square are split-closed, this shows that the right-most map is a quasi-equivalence.
\end{proof}

As a consequence of this result, we reobtain the following which can be seen as a corollary of \cite[Thm 1.5]{BalmerSchlichting}.

\begin{corollary}\label{coro::split-closure-triangulated}
 Let~$\Lambda$ be a finite-dimensional~$K$-algebra with an action of~$G$ by automorphisms, and let~$\Lambda*G$ be the skew-group algebra (see \cite{ReitenRiedtmann}).  Then the category~$\big((\per\Lambda)*G\big)_+$ is triangulated and is equivalent to~$\per(\Lambda*G)$.
\end{corollary}
\begin{proof}
 Apply Theorem~\ref{theo::comm-square} to~$\Lambda$ viewed as an~$A_\infty$-category~$\cA$.  Then~$H^0$ of the upper-right corner is~$\big((\per\Lambda)*G\big)_+$ and~$H^0$ of the lower-left corner is~$\per(\Lambda*G)$.
\end{proof}

\section{Topological Fukaya category of a surface}\label{Section3}

In this section we recall the definition of the toplogical Fukaya category of a graded surface given in \cite[section 3]{HaidenKatzarkovKontsevich} and we recall the main results concerning these categories. We end the section by giving some explicit computations of change of basis that will be used in the section \ref{Section4}.

\subsection{Graded surfaces and graded arcs}

 We define the $A_\infty$-category associated with a graded marked surface, and recall the main properties needed for our setup.

\begin{definition}
A \emph{graded marked surface} is a triple $(\surf,\bM,\eta)$ where
\begin{itemize}
\item the \emph{surface} $\surf$ is a smooth compact connected oriented surface with non empty boundary;
\item the set of \emph{marked segments} $\bM$ is a non empty finite subset of pairwise disjoint closed segments of the boundary of $\surf$;
\item the \emph{graduation} $\eta:\surf\to \bP (T\surf)$ is a section of the projectivized tangent bundle.
\end{itemize}
Note that we consider $\surf$ to be a smooth closed connected oriented surface from which we remove open discs, and as such, the tangent bundle is also defined on the boundary of $\surf$. Note also that we allow boundary component without marked segments and boundary components which are completely marked.
\end{definition}

\begin{definition}
An \emph{arc} on $(\surf,\bM,\eta)$ is a non contractible smooth immersion $\gamma:[0,1]\to \surf$ with endpoints in $\bM$, intersecting transversally the boundary, and such that $\gamma(0,1)\subset \surf\setminus\partial\surf$. Isotopies of arcs are regular isotopies with endpoints in $\bM$. An arc~$\gamma$ is called \emph{simple} if the map~$\gamma$ is injective. A \emph{boundary arc} is a simple arc which is isotopic to a boundary segment, that is, to a segment $s:[0,1]\to \partial \surf$ with endpoints in $\bM$, and such that $s(0,1)\cap \bM=\emptyset$.

Arcs are considered up to orientation.

\end{definition}

\begin{definition}\label{def-graded-arc}
 
A \emph{graded arc} $(\gamma,\grad)$ is the equivalence class of an arc $\gamma$ together with a homotopy $\grad$ of from $\gamma^*(\eta)$ to $\dot{\gamma}$. That is $\grad$ is a continuous map $\grad:[0,1]^2\to \bP(T\surf)$ such that for all $s,t\in [0,1]$, $\grad(0,t)=\eta\circ\gamma (t)$, $\grad(1,t)=\dot{\gamma}(t)$ and $\pi(\grad(s,t))=\gamma(t)$ where $\pi:\bP(T\surf)\to \surf$ is the canonical projection. Two graded curves $(\gamma_1,\grad_1)$ and $(\gamma_2,\grad_2)$ are equivalent if $\gamma_1=\gamma_2$, and if $\grad_1$ and $\grad_2$ are homotopic (among homotopies from $\gamma^*(\eta)$ to $\dot{\gamma}$). 

\end{definition}

Note that a grading $\eta$ on $\surf$ allows to construct an orientation preserving trivialisation $\Psi_{\eta}$ of the circle bundle $\bP(T\surf)$, that is a homeomorphism of bundles $\Psi_{\eta}:\bP(T\surf)\to \surf\times \bR/\bZ$ such that $\Psi_{\eta}(x)=(x,0)$ for all $x\in \surf$. This trivialisation is unique up to isotopy.  Now given a graded arc $(\gamma,\grad)$, the homotopy $\Psi_{\eta}\circ \grad$ can be uniquely lifted  to a continuous map $\tilde{\grad}:[0,1]^2\to \surf\times \bR$

\[\xymatrix{ & &\surf\times \bR\ar[d]\\ [0,1]^2\ar[r]_{\grad}\ar[urr]^{\tilde{\grad}} & \bP (T\surf) \ar[r]^{\sim}_{\Psi_{\eta}} & \surf\times \bR\times \bZ}\]
 of the form $\tilde{\grad}(t,s)=(\gamma(t),\theta_s(t))$ satisfying $\theta_0(t)=0$.

Then one easily checks that two gradings $\grad$ and $\grad'$ are homotopic if and only if $\theta_1(t)=\theta'_1(t)$ for all $t\in [0,1]$, if and only if $\theta_1(0)=\theta'_1(0)$. 
Therefore the equivalence class of a grading on $\gamma$ is determined by the integral part of the real number $\theta_1(0)$.

\begin{definition}
Let $(\gamma,\grad)$ and $(\gamma',\grad')$ be two graded curves intersecting transversally in a point $p=\gamma(t_0)=\gamma'(t_0')$. Denote by $\kappa:[0,1]\to \bP(T_p\surf)$ the smallest  path from $\dot{\gamma}(t_0)$ to $\dot{\gamma'}(t_0')$ which is positive with respect to the orientation of $\surf$ (that induces an orientation of $\bP(T_p\surf)$). The map $\grad(-,t_0):[0,1]\to \bP(T_p(\surf))$ is a path from $\eta(p)$ to $\dot{\gamma}(t_0)$ while the map $\grad'(-,t_0')^{-1}:[0,1]\to \bP(T_p(\surf))$  is a path from $\dot{\gamma'}(t_0')$ to $\eta(p)$. The \emph{index} of the intersection of the graded curves at $p$ is defined as 
$${\rm ind}_p((\gamma,\grad),(\gamma',\grad')):=\grad(-,t_0).\kappa.\grad'(-,t_0')^{-1}\in \pi_1(\bP(T_p\surf),\eta(p))$$ 
viewed as an integer with the natural isomorphism  $\pi_1(\bP(T_p\surf),\eta(p))$ where the canonical generator is given by the orientation of $\surf$.
\end{definition}
 \begin{figure}[!h]
 \[
  \scalebox{1}{
  \begin{tikzpicture}[scale=2,>=stealth]
\draw[blue] (-1,1)--(1,-1);
\draw[cyan] (-1,-1)--(1,1);
\draw[red] (0,1)--(0,-1);

\draw[purple, ->] (0.25,-0.25) arc (-45:45:0.35);

\node[purple, inner sep=0pt, fill=white] at (0.35,0) {$\kappa$};

\draw[cyan,->] (0,-0.6) arc (-90:45:0.6);
\node[cyan, inner sep=0pt, fill=white] at (0.6,0) {$\grad'_{t_0'}$};
\node[cyan, inner sep=0pt, fill=white] at (-0.8,-0.8) {$\gamma'$};

\draw[blue,->] (0,0.5) arc (90:315:0.5);
\node[blue, inner sep=0pt, fill=white] at (-0.5,0) {$\grad_{t_0}$};
 
 \node[blue, inner sep=0pt, fill=white] at (-0.8,0.8) {$\gamma$}; 
 
 \node[red, inner sep=0pt, fill=white] at (0,0.8) {$\eta(p)$};
  
  \end{tikzpicture}}\] 
 \caption{}
 \end{figure}
 
\begin{remark}
Using the discussion above, the index can be alternatively defined as follows : 
$${\rm ind}_p((\gamma,\grad),(\gamma',\grad')):= \lceil \theta_1(t_0)-\theta'_1(t'_0)\rceil.$$
Indeed, under the identification $\bP(T_p(\surf))\simeq \bR/\bZ$ given by the grading $\eta$, the path $\kappa$ can be uniquely lifted to a path $\widetilde{\kappa}$ in $\bR$ with $\widetilde{\kappa}(0)=\theta_1(t_0)$ and $\widetilde{\kappa}(1)$ being the unique real number such that $\widetilde{\kappa}(1)-\theta_1(t_0)\in [0,1)$, and $\widetilde{\kappa}(1)-\theta'_1(t'_0)\in \bZ$. Lifting the path $\grad'(-,t_0')^{-1}.\kappa.\grad(-,t_0)$ in $\bR$ one gets a path from $0$ to $\widetilde{\kappa}(1)-\theta'_1(t'_0)$. Therefore the index is the integer $\widetilde{\kappa}(1)-\theta'_1(t'_0)$. 

\end{remark}

From the remark above, one immediately obtains the following

\begin{proposition}\cite[(2.5)]{HaidenKatzarkovKontsevich}\label{prop::IndexFormula}
Let $(\gamma_1,\grad_1)$ and $(\gamma_2,\grad_2)$ be two graded curves intersecting transversally in a point $p=\gamma_1(t_1)=\gamma_2(t_2)$ in the interior of $\surf$. Then one has 
$${\rm ind}_p((\gamma_1,\grad_1),(\gamma_2,\grad_2))+{\rm ind}_p((\gamma_2,\grad_2),(\gamma_1,\grad_1))=1.$$
\end{proposition}

\begin{definition}
Let $(\gamma,\grad)$ be a graded arc (resp. graded loop). We define the graded arc (resp. graded loop) $(\gamma,\grad)[1]:=(\gamma,\grad[1])$ as follows $$\grad[1](s,t):=(\gamma(t),s(\theta_1(t)-1)),$$ where the map $\theta_1:[0,1]\to \bR$ is defined as in the previous remark.

One easily checks that ${\rm ind}_p((\gamma_1,\grad_1),(\gamma_2,\grad_2)[1])={\rm ind}_p((\gamma_1,\grad_1),(\gamma_2,\grad_2))+1.$

\end{definition}

\subsection{$A_\infty$-categories associated with a full system of graded arcs}

\begin{definition}
A \emph{system of graded arcs} $\bS=((\gamma_1,\grad_1),\ldots, (\gamma_s,\grad_s))$ is a collection of graded simple arcs such that the arcs $
\gamma_i$ are pairwise disjoint and non isotopic. 

A system of arcs  $\bS$ is called \emph{full} if it contains all boundary arcs and cuts out the surface into discs or annuli with one non marked boundary component.
\end{definition}

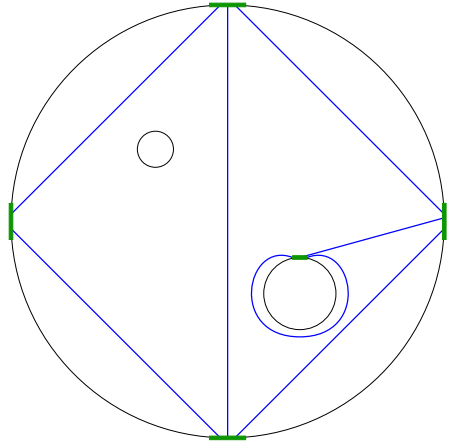
\begin{figure}[!h]
 
  \scalebox{0.6}{
  \begin{tikzpicture}[scale=0.8,>=stealth]
  
  \draw (0,0) circle (6);

  \draw (2,-2) circle (1);
  
  \draw (-2,2) circle (0.5);
  
  \draw[blue, thick] (0,6)--(0,-6);
  \draw[blue, thick] (-0.2,6)--(-6,0.2);
  \draw[blue, thick] (-6,-0.2)--(-0.2,-6);
  \draw[blue, thick] (0.2,-6)--(6,-0.2);
  \draw[blue, thick] (6,0.2)--(0.2,6);
  \draw[blue, thick] (6,0.1)--(2,-1);
  
  \draw[blue, thick] (2.2,-1).. controls (3.5,-0.5) and (4,-3.2)..(2,-3.2);
    \draw[blue, thick] (1.8,-1).. controls (0.5,-0.5) and (0,-3.2)..(2,-3.2);
  
    \draw[dark-green, fill=dark-green] (5.95,-0.5) rectangle(6.05,0.5);
  \draw[dark-green, fill=dark-green] (-0.5,5.95) rectangle(0.5,6.05);
    \draw[dark-green, fill=dark-green] (-5.95,-0.5) rectangle(-6.05,0.5);
  \draw[dark-green, fill=dark-green] (-0.5,-5.95) rectangle(0.5,-6.05);
  \draw[dark-green, fill=dark-green] (2.2,-1.05) rectangle (1.8, -0.95);

  \end{tikzpicture}}
  \caption{A full system of arcs on a sphere with three boundary components}
  \end{figure}

\begin{definition}
An \emph{$\bM$-segment} $a$ on $(\surf,\bM,\eta)$ is an embedding  of $[0,1]\to \bM$ oriented with the reverse orientation of the boundary of $\surf$. Given two $\bM$-segments $a$ and $b$ such that $a(1)=b(0)$, we define $ba=a.b$ to be the unique $\bM$-segment such that $a.b(0)=a(0)$ and $a.b(1)=b(1)$.
\end{definition}

Let $(\gamma_i,\grad_i)$ and $(\gamma_j,\grad_j)$ be two graded arcs ending in a same marked segment of $\bM$. Denote by $m_i$ and $m_j$ these endpoints.  Let $a$ be the $\bM$-segment with endpoints $m_i$ and $m_j$. Without loss of generality, we assume that $a(0)=m_i$ and $a(1)=m_j$. We say that $a$ is an $\bM$-segment from $(\gamma_i,\grad_i)$ to $(\gamma_j,\grad_j)$. The degree of $a$ is defined to be :
$$|a|={\rm ind}_{m_i}((\gamma_i,\grad_i),(a,\grad))-{\rm ind}_{m_j}((\gamma_j,\grad_j),(a,\grad)),$$ where $\grad$ is a grading on $a$. By Proposition \ref{prop::IndexFormula}, this integer does not depend on the choice of $\grad$.

Note that if $(\gamma_i,\grad_i)$ have endpoints lying in the same boundary segment, then there exists an $\bM$-segment from $(\gamma_i,\grad_i)$ to itself. There also might exist two different $\bM$-segments from $(\gamma_i,\grad_i)$ to $(\gamma_j,\grad_j)$, or one from $(\gamma_i,\grad_i)$ to $(\gamma_j,\grad_j)$, and one from $(\gamma_j,\grad_j)$ to $(\gamma_i,\grad_i)$ depending on the relative position of $\gamma_i$ and $\gamma_j$.

Note moreover that if $a$ is a $\bM$-segment from $(\gamma_i,\grad_i)$ to $(\gamma_j,\grad_j)$, it can also be viewed as an $\bM$-segment from $(\gamma_i,\grad_i)$ to $(\gamma_j,\grad_j)[1]$. Then its degree is $|a|-1$.

We are now ready to define the \emph{Fukaya category} $\cA:=\cA_{\bS}$ associated to a full system of arcs $\bS=(X_1=(\gamma_1,\grad_1),\ldots, X_s=(\gamma_s,\grad_s))$:

\begin{itemize}
\item the objects of $\cA$ are the graded curves $((\gamma_1,\grad_1),\ldots, (\gamma_s,\grad_s))$;

\item given two objects $X_i$ and $X_j$ of $\bS$, a basis of $\Hom_{\cA}(X_i,X_j)$ is given by the $\bM$-segments from $X_i$ to $X_j$ (to which we add the identity if $i=j$). If $a$ is an $\bM$-segment, the degree of the morphism $a$ is defined to be $|a|$.

\item $\mu^1_{\cA}$ is defined to be zero;

\item if $a\in \Hom (X_i,X_j)$ and $b\in \Hom_{\cA}(X_j,X_\ell)$ are $\bM$-segments, then the composition is defined by 
$$\mu^2_{\cA}(b,a):=\left\{\begin{array}{ll} (-1)^{|a|}ba  & \textrm{if $a.b$ exists,}\\ 0 & \textrm{else.}\end{array}\right.  $$
Moreover we set $\mu^2_{\cA}(1_{X_j},a)=(-1)^{|a|}a$, and $\mu^2_{\cA}(b, 1_{X_j})=b$ for any homogenous $a\in \Hom_{\cA}(X_i,X_j)$ and $b\in\Hom_{\cA}(X_j,X_{\ell})$.

\item for $n\geq 3$, the higher multiplication $\mu^n_{\cA}$ is defined as follows: for $j=1,\ldots, n-1$, assume that $a_{j}\in \Hom(X_{i_j},X_{i_{j+1}})$  and $a_n\in \Hom(X_{i_n},X_{i_1})$ are pairwise distincts $\bM$-segments such that $(\gamma_{i_1},a_1,\gamma_{i_2},a_2,\ldots,\gamma_{i_n},a_n)$ enclose a disc, then define    
$$\mu^n_{\cA}(a_n,\ldots,a_1b)=(-1)^{|b|}b$$
if $b$ is an $\bM$-segment such that $a_1b$ is defined, and 
$$\mu^n_{\cA}(ba_n,\ldots,a_1)= b$$
if $b$ is an $\bM$-segment such that $ba_n$ is defined.

Moreover, for any $X_i\in \cA$, we set $\mu^n_{\cA}(\ldots,1_{X_i},\cdots)=0$, for $n\geq 3$.

Other higher multiplications between $\bM$-segments vanish, and these operations are extended by linearity.

\end{itemize}

\begin{theorem}\cite[Prop 3.1]{HaidenKatzarkovKontsevich}
The category $\cA_{\bS}$ defined above is a strictly unital $A_\infty$-category. 
\end{theorem}

Note that in \cite{HaidenKatzarkovKontsevich}, the authors assume that there are no unmarked boundary component. The case of unmarked boundary components is treated in \cite[Definition~3.24]{OpperZvonareva}.

One of the fundamental result in \cite[Section 3]{HaidenKatzarkovKontsevich} is the following :

\begin{theorem}\cite[Lemma 3.2]{HaidenKatzarkovKontsevich}\label{thm HKK Morita equivalence}
If $\bS$ and $\bS'$ are full arc systems such that $\bS\subset \bS'$, then the inclusion $\bS\subset \bS'$ induces a Morita equivalence of $A_\infty$-categories $$\cA_{\bS}\longrightarrow \cA_{\bS'}.$$

\end{theorem}

As a consequence, given two full arc systems $\bS$ and $\bS'$ on $\surf$, since one can pass from $\bS$ to $\bS'$ by adding and deleting simple graded arcs, there is a canonical equivalence $\H^0(\Tw \cA_{\bS})\simeq \H^0(\Tw \cA_{\bS'})$. We denote by $\Psi_{\bS\to \bS'}$ this equivalence. Then if $\bS$, $\bS'$ and $\bS''$ are full systems of arcs we have $\Psi_{\bS'\to \bS''}\circ\Psi_{\bS\to \bS'}=\Psi_{\bS\to \bS''}$.  
In fact, for any full arc systems~$\bS$ and~$\bS'$ not necessarily related by inclusion, it is proved in~\cite{HaidenKatzarkovKontsevich} that $\H^0(\Tw \cA_{\bS})$ and~$\H^0(\Tw \cA_{\bS'})$ are \emph{canonically} equivalent.  Thus, by an abuse of notation, we will denote by~$\cF(\surf)$ the category~$\H^0(\Tw(\cA_\bS))$ for any full arc system~$\bS$.

This results allows to associate a well-defined object $X$ in $\cF(\surf)$ to any graded simple arc $(\gamma,\grad)$.

\begin{remark}

Assume that $\bS$ is a full system of arcs, and $X:=(\gamma,\grad)$ be a graded arc such that $\bS\cup X$ is a full system of arcs. Denote by $X'$ the graded arc $(\gamma,\grad [1])$. Then the inclusions 
$$\xymatrix{\cA_{\bS\cup X} & \cA_{\bS}\ar@{^(->}[r]\ar@{_(->}[l] & \cA_{\bS\cup X'}}$$
induce inclusions
$$\xymatrix{\bZ\cA_{\bS\cup X} & \bZ\cA_{\bS}\ar@{^(->}[r]\ar@{_(->}[l] & \bZ\cA_{\bS\cup X'}}.$$
If $f$ is an $\bM$-segment in $\bS$, then it is also an $\bM$-segment in $\bS\cup X$ and in $\bS\cup X'$. The morphism $s^\ell\otimes f$ is sent to $s^\ell\otimes f$ on both inclusions.

On the other hand, one can consider  the full $A_\infty$-subcategory $\cA'$ of $\bZ\cA_{\bS\cup X}$ containing the objects of $\bS$ and the object $X[1]$. Then one can construct a strict isomorphism between $\cA'$ and $\cA_{\bS\cup X'}$. This isomorphism induces a strict isomorphism
$$\bZ \cA_{\bS\cup X'}\longrightarrow \bZ\cA'\simeq \bZ \cA.$$
However one can check that this isomorphism sends $s^\ell\otimes ab$ to $(-1)^\ell s^\ell\otimes ab$ if $ab$ is a $\bM$-segment passing through $X$, thus this isomorphism is not compatible with the inclusions above. In general the equivalence between $H^0(\Tw \cA_{\bS\cup X})$ and $H^0(\Tw\cA_{\bS\cup X'})$ of Theorem \ref{thm HKK Morita equivalence} is not induced from an $A_\infty$-isomorphism between $\bZ\cA_{\bS\cup X}$ and $\bZ\cA_{\bS\cup X'}$.

\end{remark}

\subsection{Indecomposable objects in $\cF(\surf)$} 

All the indecomposable objects in $\cF(\surf)$ can be described in term of graded curves on $(\surf,\bM,\eta)$. 

We first need to extend the definition of graded curves to the case of loops.

\begin{definition}
Let $(\surf,\bM,\eta)$ be a graded marked surface. A \emph{graded loop} on $\surf$ is the equivalence class of $(\gamma,\grad)$ where
\begin{itemize}
\item  $\gamma:\bS^1=\bR/\bZ\to \surf\setminus \partial \surf$ is an immersion, which does not bound an immersed teardrop, and which is primitive as an element in $\pi_1(\surf)$.
\item a grading $\grad:[0,1]\times \bR/\bZ\to \bP(T\surf)$ as in Definition \ref{def-graded-arc}.
\end{itemize}
Two graded loops $(\gamma,\grad)$ and $(\gamma',\grad')$ are equivalent if $\gamma =\gamma'$ and $\grad$ and $\grad'$ are homotopic.
\end{definition}

\begin{remark}

If $\gamma:\bR/\bZ\to \surf$ is a closed loop, then using the trivialisation $\Psi_{\eta}$ of the circle bundle $\bP(T\surf)$, one obtains a map $\Psi_\eta\circ\dot{\gamma}: [0,1]\to \surf\times \bR/\bZ$ given by  $t\mapsto (\gamma(t),\theta (t))$ where $\theta(0)=\theta(1)$. By lifting $\theta$ to a continuous map $\tilde{\theta}:[0,1]\to\bR$, we obtain an integer $\tilde{\theta}(1)-\tilde{\theta}(0)$ which is the winding number $w_\eta(\gamma)$ of the curve $\gamma$ with respect to the line field $\eta$.

Now let $(\gamma,\grad)$ be a graded loop, by setting $\Psi_{\eta}(\grad(s,t))=(\gamma(t),\theta_s(t))\in \surf\times \bR/\bZ$, we have $\theta_s(0)=\theta_s(1)$ for any $s\in [0,1]$. By lifting $\grad$ to $\tilde{\grad}:[0,1]\times \bR\to \surf\times \bR$, 

\[\xymatrix{[0,1]\times \bR \ar[d]\ar[rr]^{\tilde{\grad}} & & \surf\times \bR\ar[d]\\ [0,1]\times \bR/\bZ\ar[r]_{\grad} & \bP(T\surf)\ar[r]^\sim_{\Psi_{\eta}} & \surf\times \bR/\bZ}\]
one obtains a continous map $\tilde{\grad}(s,t)=(\gamma(t),\tilde{\theta}_s(t))$.  Therefore the map $s\mapsto \tilde{\theta}_s(1)-\tilde{\theta}_s(0)$ is continuous with codomain $\bZ$, hence it is constant. Since $\tilde{\theta}_0(1)=\tilde{\theta}_0(0)=0$, we obtain $w_\eta(\gamma)=0$. Therefore a graded loop has always winding number $0$. Conversely, if $\gamma$ is a closed loop with winding number $0$, then one can construct a grading on $\gamma$.    

\end{remark}

To $(\gamma,\grad)$ a graded loop, and $V$ a finite dimensional $k[x]$-module, one can associate an object in $\Tw(\cA_{\bS})$ for a well-chosen full system of arcs, a precise construction will be given in the next subsection.

\begin{theorem}(\cite[Thm 4.3]{HaidenKatzarkovKontsevich} (see also \cite[Thm 2.12]{OpperPlamondonSchroll}, and \cite[Definition 3.24]{OpperZvonareva} for surfaces with unmarked boundary components)\label{thm HKK indec obj})
Let $(\surf,\bM,\eta)$ be a graded marked surface and let $\bS$ a full system of graded arcs. 
The indecomposable objects of the category $\H^0\Tw(\cA_{\bS})$ are in bijection with the following sets:

\begin{itemize}
\item the set of graded arcs, up to orientation and regular homotopy;
\item the set of primitive graded loops with indecomposable $k[x]$-module, up to orientation, regular free homotopy and isomorphism on $k[x]$-module.
\end{itemize}
This bijection is independent of the choice of $\bS$ via the canonical Morita equivalences and commute with $[1]$. In other words, if $(\gamma,\grad)$ is a graded arc (resp. $(\gamma,\grad,V)$ a graded loop with $k[x]$-module) and if $X$ is the corresponding object in $\cF(\surf)$, then the object corresponding to $(\gamma,\grad[1])$ (resp. $(\gamma,\grad[1],V)$) is isomorphic to $X[1]$.
\end{theorem}

\subsection{Explicit construction of the bijection for certain curves}

In this subsection, we give some concrete computation of certain objects through the map $\Psi_{\bS\to \bS'}$ where $\bS$ and $\bS'$ are full systems of graded arcs. These concrete computations will be useful in Section \ref{Section4}. 

\subsubsection{Case of strings}

\begin{lemma}\label{CalculString1}
Let $\bS$ be a full system of arcs containing a triangle $X_0$, $X_1$ and $X_2$ and denote by $a_0$, $a_1$ and $a_2$ the corresponding $\bM$-segments.
\begin{figure}[!h]
 \[
  \scalebox{0.8}{
  \begin{tikzpicture}[>=stealth]
  \draw[dark-green,thick,->]  (1,0) arc (0:90:1);
  \draw[dark-green,thick,->] (0,3) arc (-90:0:1);
  \draw[dark-green,thick,->] (4,4) arc (-180:-90:1);
  \node at (1,3) {$a_1$};
  \node at (4.5,3.5) {$a_2$};
  \node at (0.5,0.5) {$a_0$};
  \draw[blue] (0.5,0.85)--node[midway,left]{$X_1$}(0.5,3.15);
  \draw[blue] (0.85,3.5)--node[midway,above]{$X_2$}(4.15,3.5);
  \draw[blue] (0.7,0.7)--node[midway,right,below]{$X_0$}(4.3,3.3);
  \end{tikzpicture}}\]
\caption{}
\end{figure}

 Then the object $X_0[0]$ is isomorphic in $H^0(\Tw\cA_{\bS})$ to the twisted complex
$$X:\xymatrix@C=2cm{X_1[|a_0|]\ar[r]^{s^{\|a_1\|}\otimes a_1} & X_2[-|a_2|]}.$$
Inverse isomorphisms are given by $s^{|a_0|}\otimes a_0$ and $(-1)^{\|a_1\|\|a_2\|}s^{|a_2|}\otimes a_2$.
\end{lemma}

\begin{proof}
The space $\Hom_{\Tw\cA}(X_0[0],X)$ is generated by 
$s^{|a_0|}\otimes a_0$ which has degree zero. So we have 
$$H^0(\Hom_{\Tw\cA}(X_0,X))={\rm vect} (s^{|a_0|}\otimes a_0).$$ 
The space $\Hom_{\Tw\cA}(X,X_0[0])$ is generated by 
$s^{|a_2|}\otimes a_2$ which has degree zero, so we have  $$H^0(\Hom_{\Tw\cA}(X,X_0))={\rm vect} (s^{|a_2|}\otimes a_2).$$

Now we compute 
\begin{eqnarray*} \mu^2_{\Tw\cA}(s^{|a_2|}\otimes a_2,s^{|a_0|}\otimes a_0) & = & \mu^3_{\mathbb Z \cA} (s^{|a_2|}\otimes a_2,s^{\|a_1\|}\otimes a_1,s^{|a_0|}\otimes a_0)\\ & = & (-1)^\dagger s^0\otimes \mu^3_{\cA}(a_2,a_1,a_0) \\ & =  & (-1)^\dagger s^0\otimes 1_{X_0},
\end{eqnarray*}
where the sign $\dagger$ is given by 
$\dagger =\|a_2\|(\|a_1\|+|a_0|)+\|a_1\| |a_0|.$ Using the fact that $|a_0|+|a_1|+|a_2|=1$ we obtain
$$\mu^2_{\Tw\cA}(s^{|a_2|}\otimes a_2,s^{|a_0|}\otimes a_0)=(-1)^{\|a_1\|\|a_2\|}s^0\otimes 1_{X_0}.$$

We also compute 
\begin{eqnarray*} \mu^2_{\Tw\cA}(s^{|a_0|}\otimes a_0,s^{|a_2|}\otimes a_2) & = & \mu^3_{\mathbb Z \cA} (s^{\|a_1\|}\otimes a_1, s^{|a_0|}\otimes a_0,s^{|a_2|}\otimes a_2)\\ & & \hspace{1cm} + \mu^3_{\mathbb Z \cA} ( s^{|a_0|}\otimes a_0,s^{|a_2|}\otimes a_2,s^{\|a_1\|}\otimes a_1)\\ & = & (-1)^\dagger s^0\otimes \mu^3_{\cA}(a_1,a_0,a_2) +(-1)^{*}s^0\otimes \mu_\cA^3(a_0,a_2,a_1) \\ & =  & (-1)^\dagger s^0\otimes 1_{X_2}+(-1)^*s^0\otimes 1_{X_1}
\end{eqnarray*}
where the signs are given by 
$$\dagger =\|a_1\| (|a_0|+|a_2|)+\|a_0\| |a_2| \quad \textrm {and}\quad *=\|a_0\|(|a_2|+\|a_1\|)+\|a_2\|\|a_1\|.$$
Then using the fact that $|a_0|+|a_1|+|a_2|=1$ we obtain
$\mu^2_{\Tw\cA}(s^{|a_0|}\otimes a_0,s^{|a_2|}\otimes a_2)=(-1)^{\|a_1\|\|a_2\|}s^0\otimes 1_{X}$.

\end{proof}

\subsubsection{Case of bands}

Let $(\gamma,\grad)$ be a graded simple loop, and $\lambda\in K^*$. Assume that $\bS$ is a full system of arcs containing graded simple arcs $X,Y,Z_1$ and $Z_2$ as in the following picture. 

\begin{figure}[!h]
 \[
  \scalebox{0.8}{
  \begin{tikzpicture}[scale=1.2,>=stealth]
  \draw[dark-green,thick,->]  (1,0) arc (0:90:1);
  \draw[dark-green,thick,->] (0,3) arc (-90:0:1);
  \draw[dark-green,thick,->] (4,4) arc (-180:-90:1);
  \draw[dark-green,thick,->] (5,1) arc (90:180:1);
  \node at (0.5,2.8) {$c_2$};
  \node at (1.2,3.2) {$a_1$};
  \node at (4.5,1.2) {$c_1$};
  \node at (3.8,0.7) {$a_2$};
  \node at (4.5,3.5) {$b_1$};
  \node at (0.5,0.5) {$b_2$};
  \draw[blue] (0,1)--node[midway,left]{$X$}(0,3);
  \draw[blue] (0.85,3.5)--node[midway,above]{$Z_1$}(4.15,3.5);
  \draw[blue] (0.7,3.3)--node[midway,right,above]{$Y$}(4.3,0.7);
  \draw[blue] (0.85,0.5)--node[midway,below]{$Z_2$}(4.15,0.5);
  \draw[blue] (5,1)--node[midway,right]{$X$}(5,3);
  \draw[thick,->,purple] (0,1.5)--node[midway,left, inner sep=2pt,fill=white]{$(\gamma,\grad)$}(5,1.5);
  \draw[purple,->] (0.5,1.5) arc (0:90:0.5);
  \node[purple, inner sep=1pt, fill=white] at (0.3,1.7) {$p$};
  \node[purple, inner sep=1pt, fill=white] at (3.1,1.7) {$q$};
  \draw[purple,->] (3.5,1.5) arc (0:120:0.5);
  
  \end{tikzpicture}}\]
  \caption{}
\end{figure}

We denote by $p$ (resp. $q$) the index between $(\gamma,\grad)$ and $X$ (resp. $Y$) and define the following object in $\Tw (\cA_{\bS})$
\begin{equation}\label{defBand}
\xymatrix@C=2cm{B_{(\gamma,\grad,\lambda,\bS)}: X[p]\ar[rr]^-{s^{\|c\|}\otimes (c_1+\lambda (-1)^{\|c\|} c_2)} & & Y[q-1].}
\end{equation}

\begin{lemma}\label{calculBand1}
 The twisted complex $B:=B_{(\gamma,\grad,\lambda,\bS)}$ is isomorphic in $\H^0(\Tw(\cA_{\bS}))$ to the following twisted complexes : 
\begin{itemize}
\item[$B_1:=$] $ \xymatrix@C=2cm{X[p]\ar[rr]^-{s^{\|c\|}\otimes(\lambda^{-1}c_1+ (-1)^{\|c\|} c_2)} & & Y[q-1].}$
\item[$B_2:=$] $ \xymatrix@C=2.5cm{X[p]\ar@/_/[rr]_{-\lambda^{-1}s^{-1}\otimes 1_X}\ar[r]^{s^{\|a_1c_2\|}\otimes a_1c_2} & Z_1[p-|b_1|]\ar[r]^{s^{\|b_1\|}\otimes b_1} & X[p-1]}$
 \item[$B_3:=$] $ \xymatrix@C=2.5cm{Y[q]\ar@/_/[rr]_{-\lambda^{-1}s^{-1}\otimes 1_Y}\ar[r]^{s^{\|a_1\|}\otimes a_1} & Z_1[p-|b_1|]\ar[r]^{s^{\|c_2b_1\|}\otimes c_2b_1} & Y[q-1]}$
  \item[$B_4:=$] $ \xymatrix@C=2.5cm{X[p]\ar@/_/[rr]_{(-\lambda)s^{-1}\otimes 1_X}\ar[r]^{s^{\|a_2c_1\|}\otimes a_2c_1} & Z_2[p-|b_2|]\ar[r]^{s^{\|b_2\|}\otimes b_2} & X[p-1]}$
   \item[$B_5:=$] $ \xymatrix@C=2.5cm{Y[q]\ar@/_/[rr]_{(-\lambda)s^{-1}\otimes 1_Y}\ar[r]^{s^{\|a_2\|}\otimes a_2} & Z_1[p-|b_1|]\ar[r]^{s^{\|c_1b_2\|}\otimes c_1b_2} & Y[q-1]}$
\end{itemize} 
 
\end{lemma}

\begin{proof}
First note that we have the following equalities on the degrees:
$$|c_i|+p+(1-q)=1 \quad \textrm{and} \quad |a_i|+|b_i|+|c_i|=1 \quad \textrm{ for }i=1,2.$$

We prove the isomorphism between $B$ and $B_2$. The other cases are completely similar.

\medskip

First note that the degree $0$ subspace of the graded space $\Hom_{\Tw\cA}(B,B_2)$ is generated by $s^0\otimes 1_X:X[p]\to X[p]$ and $s^{|a_1|}\otimes a_1:Y[q-1]\to Z_1[q-1+|a_1|]=Z_1[p-|b_1|]$ shown in the following diagram.

$$\xymatrix@C=3cm{X[p]\ar[d]_{s^0\otimes 1_X}\ar[r]^-{s^{\|c\|}\otimes (c_1+\lambda (-1)^{\|c\|} c_2)} &  Y[q-1]\ar[d]^{s^{|a_1|}\otimes a_1} & \\X[p]\ar@/_/[rr]_{-\lambda^{-1}s^{-1}\otimes 1_X}\ar[r]^{s^{\|a_1c_2\|}\otimes a_1c_2} & Z_1[p-|b_1|]\ar[r]^{s^{\|b_1\|}\otimes b_1} & X[p-1] .}$$

\noindent
Then we compute

\begin{eqnarray*} \mu^1_{\Tw \cA}(s^0\otimes 1_X) & = & \mu^2_{\bZ\cA}(-\lambda^{-1} s^{-1}\otimes 1_X,s^0\otimes 1_X)+\mu^2_{\bZ\cA}( s^{\|a_1c_2\|}\otimes a_1c_2,s^0\otimes 1_X)\\ & = & -\lambda^{-1}s^{-1}\otimes 1_X+ s^{\|a_1c_2\|}\otimes a_1c_2.\end{eqnarray*}

\begin{eqnarray*}
\mu^1_{\Tw\cA}(s^{|a_1|}\otimes a_1) & = & (-1)^{\|c\|}\lambda \mu^2_{\bZ\cA}(s^{|a_1|}\otimes a_1,s^{\|c\|}\otimes c_2)+\mu^3_{\bZ\cA}(s^{\|b_1\|}\otimes b_1,s^{|a_1|}\otimes a_1,s^{\|c\|}\otimes c_1) \\  & = & (-1)^{\|c\|}(-1)^{\|a_1\|.\|c\|}s^{|a_1|+\|c\|}\otimes \mu^2_{\cA}(a_1,c_2) + (-1)^\dagger s^{-1}\otimes \mu^3_\cA(b_1,a_1,c_1) \\ & = & -\lambda (-1)^{\|a_1\|.\|c\|}s^{|a_1|+\|c\|}\otimes a_1c_2 +(-1)^\dagger s^{-1}\otimes 1_X.
\end{eqnarray*}
The sign $\dagger$ is given by 
$$\dagger= \|b_1\| (|a_1|+\| c\|)+\|a_1\|.\|c\|=\|a_1\|.\|c\|.$$
Therefore we obtain that $f:=\lambda s^0\otimes 1_X +(-1)^{\|a_1\|.\|c\|}s^{|a_1|}\otimes a_1$ is in the kernel of $\mu^1_{\Tw \cA}$. Moreover, there is no degree $-1$ morphism from $B$ to $B_2$, therefore we have 
$$H^0 (\Hom_{\Tw\cA_\bS}(B,B_2))={\rm vect} (f).$$

\medskip

Now the degree $0$ subspace of the graded space $\Hom_{\Tw\cA}(B_2,B)$ is generated by $s^0\otimes 1_X:X[p]\to X[p]$, $s^{|b_1|}\otimes b_1:Z_1[p-|b_1|]\to X[p]$, and $s^{|c|}\otimes c_i:X[p-1]\to Y[q-1]$ for $i=1,2$ as shown in the following diagram

$$\xymatrix@C=3cm{X[p]\ar@/^2pc/[rr]^{-\lambda^{-1}s^{-1}\otimes 1_X}\ar[d]_{s^0\otimes 1_X}\ar[r]^{s^{\|a_1c_2\|}\otimes a_1c_2} & Z_1[p-|b_1|]\ar[dl]^{s^{|b_1|}\otimes b_1}\ar[r]^{s^{\|b_1\|}\otimes b_1} & X[p-1]\ar[dl]^{s^{|c|}\otimes c_i} \\ X[p]\ar[r]_-{s^{\|c\|}\otimes (c_1+\lambda (-1)^{\|c\|} c_2)} &  Y[q-1] & .}$$

Then we compute :

\begin{eqnarray*}
\mu^1(s^0\otimes 1_X) & = & \mu^2(s^{\|c\|}\otimes (c_1+(-1)^{\|c\|}\lambda c_2),s^0\otimes 1_X) \\ & = &  s^{\|c\|}\otimes (c_1+(-1)^{\|c\|}\lambda c_2).
\end{eqnarray*}

\begin{eqnarray*}
\mu^1(s^{|c|}\otimes c_1) & = & \mu^3(s^{|c|}\otimes c_1,s^{\|b_1\|}\otimes b_1,s^{\|a_1c_2\|}\otimes a_1c_2) -\lambda^{-1}\mu^2(s^{|c|}\otimes c_1,s^{-1}\otimes 1_X)\\ & = & (-1)^{\|c\|}s^{\|c\|}\otimes \mu_3(c_1,b_1,a_1c_2) -\lambda^{-1}(-1)^{\|c\|}s^{\|c\|}\otimes c_1 \\ & = & -s^{\|c\|}\otimes c_2 -\lambda^{-1} (-1)^{\|c\|}s^{\|c\|}\otimes c_1 
\end{eqnarray*}

Therefore $g:=s^0\otimes 1_X+\lambda (-1)^{\|c\|c}s^{|c|}\otimes c_1$ is in the kernel of $\mu^1_{\Tw\cA}$. Since there are no degree $-1$ morphism from $B_2$ to $B$, we have

$$H^0(\Hom_{\Tw\cA}(B_2,B))={\rm vect}(g).$$

Finally we have

\begin{eqnarray*}
\mu^2_{\Tw\cA}(g,f) & = & \lambda \mu^2_{\bZ\cA}(s^0\otimes 1_X,s^0\otimes 1_X) +(-1)^{\|c\| |a_1|}\lambda \mu^3_{\bZ\cA}(s^{|c|}\otimes c_1,s^{\|b_1\|}\otimes b_1,s^{|a_1|}\otimes a_1) \\
 & = & \lambda (s^0\otimes 1_X +s^0\otimes 1_Y)= \lambda 1_{B},
\end{eqnarray*}

and

\begin{eqnarray*}
\mu^2(f,g) & = & \lambda \mu^2(s^0\otimes 1_X,s^0\otimes 1_X)+ (-1)^{\|a_1\||c|}\lambda \left(\mu^3 (s^{|a_1|}\otimes a_1,s^{|c|}\otimes c_1,s^{\| b_1\|}\otimes b_1)+ \right.\\ & & \left.\qquad \mu^3(s^{|c|}\otimes c_1,s^{\| b_1\|}\otimes b_1,s^{|a_1|}\otimes a_1)\right)\\ &= & \lambda (s^0\otimes 1_X+s^0\otimes 1_{Z_1}+s^0\otimes 1_Y)= \lambda 1_{B_2}.
\end{eqnarray*}

So $\lambda^{-1}f$ and $g$ are inverse isomorphisms from $B$ to $B_2$ in $H^0(\Tw\cA_{\bS})$. 

\end{proof}

This lemma implies that we have $\Psi_{\bS\to \bS'}(B_{(\gamma,\grad,\lambda,\bS)})\simeq B_{(\gamma,\grad,\lambda,\bS')}$ for any full system of graded arcs~$\bS$ and~$\bS'$ containing $Z_1$ and $Z_2$. In particular, we have this isomorphism for $\bS':=(\bS\setminus X)\cup X'$ where $X'$ is the graded arc corresponding to $X[1]$. Indeed in order to pass from $\bS$ to $\bS'$ we use the following two inclusions of full arc systems
$$\xymatrix{ \bS & \bS'':= \bS\setminus X\ar@{_(->}[l]\ar@{^(->}[r] & \bS'}$$
By definition we have 
$$\xymatrix@C=2cm{B_{(\gamma,\grad,\lambda,\bS')}:= X'[p-1]\ar[rr]^-{s^{\|c'\|}\otimes (c'_1+\lambda (-1)^{\|c'\|} c'_2)} & & Y[q-1],}$$
 where $\|c'\|=\|c\|-1$. By the previous lemma it is isomorphic in $\Tw(\cA_{\bS'})$ to the twisted complex  
$$ \xymatrix@C=2.5cm{Y[q]\ar@/_/[rr]_{-\lambda^{-1}s^{-1}\otimes 1_Y}\ar[r]^{\lambda s^{\|a_1\|}\otimes a_1} & Z_1[p-1-|b'_1|]\ar[r]^{s^{\|c'_2b'_1\|}\otimes c'_2b'_1} & Y[q-1]}.$$ 
Since we have the equality  $c'_1b'_2=c_1b_2$ as $\bM$-segments, and since $|b'_1|=|b_1|+1$ we obtain
that $\Psi_{\bS'\to \bS''}(B_{(\gamma,\grad,\lambda,\bS')})$ is isomorphic to
$$ \xymatrix@C=2.5cm{Y[q]\ar@/_/[rr]_{-\lambda^{-1}s^{-1}\otimes 1_Y}\ar[r]^{ s^{\|a_1\|}\otimes a_1} & Z_1[p-|b_1|]\ar[r]^{s^{\|c_2b_1\|}\otimes c_2b_1} & Y[q-1]}$$ 
viewed as an object in $\Tw(\cA_{\bS''})$.  Therefore, we have 
$$\Psi_{\bS\to\bS''}(B_{(\gamma,\grad,\lambda,\bS)})\simeq \Psi_{\bS'\to\bS''}(B_{(\gamma,\grad,\lambda,\bS')}) \quad \textrm{in }\Tw(\cA_{\bS''}).$$

\section{Topological Fukaya category of a surface with $\bZ/2\bZ$-action}\label{Section4}

\subsection{$G$-invariant graded marked surfaces}

In the rest of the paper,  the group $G$ will be assumed to be $G=\bZ/2\bZ$, written multiplicatively with generator~$g$.  Moreover, we assume that~$\surf$ has no completely marked boundary components (since we are mainly interested in finite-dimensional algebra).

\begin{definition}\label{def:Gsurf}
A \emph{$G$-graded marked surface} $(\surf,\bM,\eta,g)$ is a graded marked surface $(\surf,\bM,\eta)$ together with an orientation-preserving automorphism $g$  of $\surf$ of order $2$ such that $g(\bM)=\bM$ and $g^*(\eta)=\eta$. We assume that $g$ has only finitely many fixed points.
\end{definition}

The group $G=\langle g\rangle$ acts on the set of graded arcs and curves. More precisely, if $(\gamma,\grad)$ is a graded arc on $(\surf,\bM,\eta,g)$, then so is $(g_*(\gamma),g^*(\grad))$. A \emph{$G$-invariant graded arc} is an arc $(\gamma,\grad)$ such that $(\gamma^{-1},\grad^{-1})=(g_*(\gamma),g^*(\grad))$. Note that there are no arcs satisfying $g_*(\gamma)=\gamma$, otherwise~$g$ would act as the identity. 

A \emph{full $G$-invariant arc system} $\bS$ is a full arc system $\bS$ which is globally invariant by the action of $G$. It consists of pairs of graded arcs $\{(\gamma,\grad),(g_*(\gamma),g^*(\grad)\}$ and of $G$-invariant graded arcs.
The action of $G$ on $\bS$ induces a strict action of $G$ on the~$A_\infty$-category $\cA:=\cA_\bS$. Our aim is now to understand the indecomposable objects in the category $H^0(\Tw ((\cA* G)_+))$. By Theorem  \ref{theo::comm-square} they are in natural bijection with the indecomposable objects in the category $H^0((\Tw \cA) *G)_+$ which is the split-closure of the category $H^0((\Tw \cA) *G)$.

Therefore, from now on, in order to make the notations less heavy, we will write~$H^0(\Tw \cA *G)_+$ instead of  $H^0((\Tw \cA) *G)_+$ or $H^0(\Tw ((\cA* G)_+))$ when discussing indecomposable objects.

The results of Section \ref{Section2} can be summarized as follows for $X$ an object in $\Tw \cA$. 

\begin{itemize}
\item If $gX$ is not isomorphic to $X$ in $H^0(\Tw \cA)$, then the object $\{X\}:=(\{X\},1_X\otimes 1_G)$ is indecomposable in $H^0(\Tw \cA *G)_+$. 

\item If $gX$ is isomorphic to $X$ in $H^0(\Tw \cA)$, and if there is a morphism $\varphi\in\Hom_{\Tw \cA}(g X, X)$ satisying $\mu^2_{\Tw\cA}(\varphi,g \varphi)=1_X$ then $e_X^{\pm \varphi}:=\frac{1}{2}(1_X\otimes 1_G\pm \varphi\otimes g)$ is an idempotent of the algebra $\End_{H^0(\Tw\cA *G)}(\{X\})$. We denote by $X^{\pm\varphi}$ the object $(\{X\},e_X^{\pm \varphi})$. It is an indecomposable object in $H^0(\Tw \cA *G)_+$, and we have $\{X\}\simeq X^{\varphi}\oplus X^{-\varphi}$. The objects $X^{\pm\varphi}$ will be called the tagged objects.

\end{itemize}

Moreover, by Theorem \ref{thm HKK Morita equivalence} if $\bS'$ is another $G$-invariant full system of graded arcs, then we have a canonical equivalence 
$$(\Psi_{\bS\to \bS'}*G)_+: H^0(\Tw \cA*G)_+\longrightarrow H^0(\Tw\cA'*G)_+.$$ where $\cA':=\cA_{\bS'}$. 

Our aim in the rest of the section is to describe some objects of the form $X^{\pm\varphi}$ and their image through the functors $(\Psi_{\bS\to \bS'}*G)_+$.

\subsection{$G$-invariant strings}

Let $X=(\gamma,\grad)$ be a $G$-invariant graded simple arc, and let $\bS$ be a $G$-invariant full system of graded arcs containing $X$. Such a system always exists. Then we are in the situation where $gX=X$. We denote by $X^\pm$ the object $X^{\pm 1_X}$. The next proposition tells that the assignment of the sign is independent of the choice of $\bS$. 

\begin{proposition} If $\bS$ and $\bS'$ are two $G$-invariant full systems of graded arcs containing $X$, then the equivalence 
$$ (\Psi_{\bS\to \bS'}*G)_+ : H^0(\Tw \cA*G)_+\longrightarrow H^0(\Tw \cA'*G)_+$$ sends the object $X^+$ to the object $X^+$ (resp. $X^-$ to $X^-$).
\end{proposition}

\begin{proof} One can pass from $\bS$ to $\bS'$ by adding and deleting arcs so that in each step we have a $G$-invariant full system containing $X$.  Then since $\Psi_{\bS\to \bS'}(1_X)=1_X$, we have $(\Psi_{\bS\to \bS'}*G) (e_X^+)=(\Psi_{\bS\to \bS'}*G) \big(\frac{1}{2} (1_X\otimes 1_G + 1_X\otimes g)\big) = \frac{1}{2}(\Psi_{\bS\to \bS'}(1_X)\otimes 1_G + \Psi_{\bS\to \bS'}(1_X)\otimes g) = e_X^+$ and we get the result.
\end{proof}

\begin{proposition}\label{prop:endoX+}
Let $X=(\gamma,\grad)$ be a $G$-invariant graded simple arc, and $\mathbb S$ be a full system of graded $G$-invariant graded arcs. Let $X^+$ and $X^-$ be the indecomposable objects in $H^0(\Tw\cA*G)_+$ as defined above. Then in the category $H^0(\Tw\cA*G)_+$ we have 
$$ \dim_K \Hom(X^\epsilon,X^{\epsilon'}[\ell]) = \left\{\begin{array}{cc} 1 & \textrm{ if } \ell=0 \textrm{ and } \epsilon=\epsilon' \\ 0 & \textrm{ else}.\end{array} \right.$$

\end{proposition}

\begin{proof}
The object $X$ in $H^0\Tw\cA$ is a brick, it has no higher self extension, and a one dimensional Hom-space. Hence in the category $H^0(\Tw\cA*G)_+$ we have 
$$\dim_K \End (X^+\oplus X^-)= 2\ \textrm{ and } \Hom(X^+\oplus X^-,(X^+\oplus X^-)[\ell])= 0 \textrm{ for } \ell\neq 0.$$
Since $1_X\otimes \frac{1_G\pm g}{2}$ is the identity from $X^\pm$ to itself, we get the result.
\end{proof}

For $\bS$ a $G$-invariant full system of graded arcs, we can then easily compute the morphism spaces between the objects in $(\cA*G)_+$.

\begin{proposition}Let $\bS$ be a full system of graded arcs which is $G$-invariant, and $\cA$ the corresponding $A_\infty$-category. Let $X$ and $Y$ be arcs in $\bS$.

\begin{itemize}
\item
if $X$ and $Y$ are not $G$-invariant, then $\dim_K \Hom_{H^0(\Tw\cA*G)_+} (\{X\},\{Y\}[\ell])=\dim_K \Hom_{H^0(\Tw\cA)}(X\oplus gX,Y[\ell]).$
\item if $X$ is $G$-invariant and $Y$ is not, then $\dim_K \Hom_{H^0(\Tw\cA*G)_+} (X^\pm,\{Y\}[\ell])=\dim_K \Hom_{H^0(\Tw\cA)}(X,Y[\ell]).$
\item if $Y$ is $G$-invariant and $X$ is not, then $\dim_K \Hom_{H^0(\Tw\cA*G)_+} (\{X\},Y^\pm[\ell])=\dim_K \Hom_{H^0(\Tw\cA)}(X,Y[\ell]).$
\item if $X$ and $Y$ are both $G$-invariant then $\dim_K \Hom_{H^0(\Tw\cA*G)_+} (X^\pm,Y^\pm[\ell])=\dfrac{1}{2}\dim_K \Hom_{H^0(\Tw\cA)}(X,Y[\ell]).$

\end{itemize}
\end{proposition}
 
\begin{proof}
By definition we always have \begin{eqnarray*}\dim_K \Hom_{H^0(\Tw\cA*G)_+} (\{X\},\{Y\}[\ell]) & = & \dim_K \Hom_{H^0(\Tw\cA*G)} (\{X\},\{Y\}[\ell])\\ & = & \dim_K \Hom_{H^0(\Tw\cA)}(X\oplus gX,Y[\ell]).\\ \end{eqnarray*}
Now for the second item, if $X$ is $G$-invariant and $Y$ is not, then the space   $$\Hom_{H^0(\Tw\cA*G)_+} (X^\pm,\{Y\}[\ell])$$ is generated by elements of the form $s^0\otimes a\otimes \frac{1\pm g}{2}=\mu^2(s^0\otimes a\otimes 1_G, e_X^{\pm})$ where $a$ is an $\bM$-segment from $X$ to $Y$.
The third item is dual.  

Finally, assume that $X$ and $Y$ are both $G$-invariant. Then note that the dimension of $\Hom_{H^0(\Tw\cA)}(X,Y[\ell])$ is even since for any morphism $s^0\otimes a$ from $X$ to $Y$, $s^0\otimes ga$ is also a morphism from $X$ to $Y$. Then for $\epsilon,\epsilon'\in \{+,-\}$, the space $\Hom (X^\epsilon,Y^{\epsilon '}[\ell])$ is generated by elements of the form 
$$\mu^2 (e_X^{\epsilon},\mu^2(s^0\otimes a\otimes 1_G,e_Y^{\epsilon'})),$$ where $a$ is an $\bM$-segment from $X$ to $Y$. 
We refer to \cite[Prop. 4.3]{AmiotBrustle} for similar computations in the algebra setting.
\end{proof}

\subsection{Morphisms between $G$-invariant strings intersecting in their middle}

Now we would like to understand the morphisms between the objects $X^\pm$ and $Y^\pm$ when the corresponding curves intersect at the orbifold point. Since the curves do intersect in the interior of $\surf$, one has first to describe the objects $X^\pm$ in the category $H^0(\Tw\cA_{\bS}*G)_+$ where $\bS$ does not contain the arc $X$.

Assume that we have a $G$-invariant full system of arcs $\bS$ containing the following graded arcs~$X,X_1$ and~$X_2$, and denote by $\cA:=\cA_{\bS}$ the corresponding $A_\infty$-category. 

\begin{figure}[!h]
 \[
  \scalebox{0.8}{
  \begin{tikzpicture}[scale=1.2,>=stealth]
  \draw[dark-green,thick,->]  (1,0) arc (0:90:1);
  \draw[dark-green,thick,->] (0,3) arc (-90:0:1);
  \draw[dark-green,thick,->] (4,4) arc (-180:-90:1);
  \draw[dark-green,thick,->] (5,1) arc (90:180:1);
  \node at (1,3) {$a_1$};
  \node at (4,1) {$ga_1$};
  \node at (3.75,3.25) {$a_2$};
  \node at (4.25,2.75) {$ga_0$};
  \node at (1.25,0.75) {$ga_2$};
  \node at (0.75,1.25) {$a_0$};
  \draw[blue] (0.5,0.85)--node[midway,left]{$X_1$}(0.5,3.15);
  \draw[blue] (0.85,3.5)--node[midway,above]{$X_2$}(4.15,3.5);
   \draw[blue] (0.7,0.7)--node[midway,right,below]{$X$}(4.3,3.3);
  
  \draw[blue] (0.85,0.5)--node[midway,below]{$gX_2$}(4.15,0.5);
  \draw[blue] (4.5,0.85)--node[midway,right]{$gX_1$}(4.5,3.15);
 
 \node[red] at (2.5,2) {$\times$};
  
  \end{tikzpicture}}\]
  \caption{}
\end{figure}

We have $|a_0|+|a_1|+|a_2|=1$. Moreover since the system is $G$-invariant, for each $i=0,1,2$, we have $|a_i|=|ga_i|$.

We define the following object in $\Tw \cA$
$$ \xymatrix{\tilde{X}: & X_1[|a_0|]\ar[rr]^{s^{\|a_1\|}\otimes a_1} & &X_2[-|a_2|]}.$$ Denote by $\varphi$ the morphism in $\Hom_{\Tw \cA}( g\tilde{X}, \tilde{X}) $ defined as follows

$$ \xymatrix{g\tilde{X}\ar[d]^\varphi: & gX_1[|a_0|]\ar[rr]^{s^{\|a_1\|}\otimes ga_1} & &gX_2[-|a_2|]\ar[dll]^-{(-1)^{|a_0|}s^{-\|a_1\|}\otimes a_0ga_2}\\
\tilde{X}: &  X_1[|a_0|]\ar[rr]_{s^{\|a_1\|}\otimes a_1} & &X_2[-|a_2|]}.$$

Then using the fact that 

$$\mu_\cA^4(a_1,a_0ga_2,ga_1,ga_0a_2)=1_{X_2} \textrm{ and } \mu_\cA^4(a_0ga_2,ga_1,ga_0a_2,a_1)=1_{X_1},$$ we compute 
\begin{eqnarray*}
\mu^2_{\Tw\cA}(\varphi, g\varphi)  & = & \mu^4_{\bZ\cA}(s^{\|a_1\|}\otimes a_1,s^{-\|a_1\|}\otimes a_0 ga_2, s^{\|a_1\|}\otimes ga_1,s^{-\|a_1\|}\otimes ga_0 a_2) \\ & & + \mu^4_{\bZ\cA}(s^{-\|a_1\|}\otimes a_0 ga_2, s^{\|a_1\|}\otimes ga_1,s^{-\|a_1\|}\otimes ga_0 a_2, s^{\|a_1\|}\otimes a_1)\\ & = & s^0\otimes 1_{X_1}+ s^0\otimes 1_{X_2} = 1_{\tilde{X}}.
\end{eqnarray*}
In a similar fashion we have $\mu^2_{\Tw \cA}(g\varphi,\varphi)=1_{g\tilde{X}}.$
Therefore, we have two indecomposable objects in $H^0(\Tw\cA*G)_+$ that we can denote by $\tilde{X}^{\varphi}$ and $\tilde{X}^{-\varphi}$. 

\begin{proposition}\label{prop:X epsilon independant}
With the setup above, we have an isomorphism  between $X^+$ and $\tilde{X}^{\varphi}$ (resp. between $X^-$ and $\tilde{X}^{-\varphi}$) in the category $H^0(\Tw\cA *G)_+$. 
\end{proposition}

\begin{proof} By Lemma \ref{CalculString1}, we know that $s^0\otimes a_0\otimes 1_G$ gives an isomorphism in $\cF:=H^0(\Tw\cA*G)_+$ between $X^+\oplus X^-$ and $\tilde{X}^{\varphi}\oplus\tilde{X}^{-\varphi}$. Now we have an isomorphism 
$$\Hom_{\cF} (X^\pm,\tilde{X}^{\varphi})=e_{\tilde{X}}^\varphi \Hom_{H^0(\Tw\cA *G)}(X^+\oplus X^-,\tilde{X}^\varphi\oplus \tilde{X}^{-\varphi}) e_X^{\pm},$$ in the category $\cF=H^0(\Tw\cA *G)_+$ where $e_X^\pm=\frac{1}{2}(1_X\otimes 1_G \pm 1_X\otimes g)$ and $e_{\tilde{X}}^\varphi= \frac{1}{2} (1_{\tilde{X}}\otimes 1_G +\varphi\otimes g)$. Therefore it is enough to compute 
$$\mu^2 (e_{\tilde{X}}^\varphi,\mu^2(s^{|a_0|}\otimes a_0\otimes 1_G,e_X^\pm)),$$
in the $A_\infty$-category $(\Tw \cA)*G$ in order to know whether $\tilde{X}^{\varphi}$ is isomorphic to $X^+$ or to $X^-$. 

First we have 
\begin{eqnarray*} \mu^2 (s^{|a_0|}\otimes a_0\otimes 1_G, e_X^\pm) & = & \mu^2_{\Tw\cA}(s^{|a_0|}\otimes a_0, 1_X)\otimes \frac{1_G\pm g}{2}\\ & = & s^{|a_0|}\otimes a_0\otimes \frac{1_G\pm g}{2}.
\end{eqnarray*}

Then we have 
\begin{eqnarray*}
\mu^2 ( e_{\tilde{X}}^{\varphi},s^{|a_0|}\otimes a_0\otimes \frac{1_G\pm g}{2}) & = & \frac{1}{2} \left( \mu^2 (1_{\tilde{X}}\otimes 1_G + (-1)^{|a_0|}s^{-\|a_1\|}\otimes a_0ga_2 \otimes g, s^{|a_0|}\otimes a_0\otimes \frac{1_G\pm g}{2})\right) \\ & = & 
\frac{1}{2} \mu^2_{\Tw \cA} (1_{\tilde{X}},s^{|a_0|}\otimes a_0)\otimes \frac{1_G\pm g}{2} \\ &  & + \frac{1}{2}(-1)^{|a_0|} \mu^2_{\Tw\cA} (s^{-\|a_1\|}\otimes a_0 g a_2,s^{|a_0|}\otimes ga_0)\otimes \frac{g\pm 1_G}{2}.
\end{eqnarray*}

Finally we have 
\begin{eqnarray*} \mu^2_{\Tw\cA} (s^{-\|a_1\|}\otimes a_0 g a_2,s^{|a_0|}\otimes ga_0) & = & \mu^3_{\bZ\cA} (s^{-\|a_1\|}\otimes a_0ga_2, s^{\|a_1\|}\otimes ga_1,s^{|a_0|}\otimes ga_0) \\ & = & (-1)^{\dagger} s^{|a_0|}\otimes \mu^3_{\cA}(a_0ga_2,ga_1,ga_0) \\ & = & (-1)^{\dagger} s^{|a_0|}\otimes a_0.
\end{eqnarray*}
The sign $(\dagger)$ is given (mod $2$) by 
\begin{eqnarray*}\dagger & = &  \|a_0ga_2\| (\|a_1\|+|a_0|)+\|a_1\| |a_0| \\ & = & (\|a_1\|+1) (\|a_1\|+|a_0|)+\|a_1\| |a_0|  = |a_0|.
\end{eqnarray*}
Therefore we obtain  
\begin{eqnarray*}
\mu^2 ( e_{\tilde{X}}^{\varphi},s^{|a_0|}\otimes a_0\otimes \frac{1_G\pm g}{2}) & = &  (1\pm 1) s^{|a_0|}\otimes a_0\otimes \frac{1_G\pm g}{4}.
\end{eqnarray*}
This implies that $\Hom_{\cF}(X^+,\tilde{X}^{\varphi})\neq 0$, and that $\Hom_{\cF}(X^-,\tilde{X}^{\varphi})=0$ and we get the result.
\end{proof}

This proposition permits to compute the morphism space between the objects $X^\pm$ and $Y^\pm$ with $Y$ being a 
$G$-invariant arc intersecting $X$ in the fixed point. 

 \begin{figure}[!h]
  \scalebox{0.8}{
  \begin{tikzpicture}[scale=1.2,>=stealth]
  \draw[dark-green,thick,->]  (1,0) arc (0:90:1);
  \draw[dark-green,thick,->] (0,3) arc (-90:0:1);
  \draw[dark-green,thick,->] (4,4) arc (-180:-90:1);
  \draw[dark-green,thick,->] (5,1) arc (90:180:1);
  \node at (0.75,2.75) {$a'_1$};
  \node at (1.25,3.25) {$a''_1$};
  \node at (4.25,1.25) {$ga'_1$};
  \node at (3.75,0.75) {$ga''_1$};
  
    \node at (3.75,3.25) {$a_2$};
  \node at (4.25,2.75) {$ga_0$};
  \node at (1.25,0.75) {$ga_2$};
  \node at (0.75,1.25) {$a_0$};

  \draw[blue] (0.5,0.85)--node[midway,left]{$X_1$}(0.5,3.15);
  \draw[blue] (0.85,3.5)--node[midway,above]{$X_2$}(4.15,3.5);
  \draw[blue] (0.7,3.3)--(4.3,0.7);
    \draw[blue] (0.7,0.7)--(4.3,3.3);
  \draw[blue] (0.85,0.5)--node[midway,below]{$gX_2$}(4.15,0.5);
  \draw[blue] (4.5,0.85)--node[midway,right]{$gX_1$}(4.5,3.15);
   
   \node[blue, inner sep=1pt, fill=white] at (1.5,2.5) {$Y$};
   
   \node[blue, inner sep=1pt, fill=white] at (3.5,2.5) {$X$};
   
   \draw[red,->]  (3.2,2.5) arc (35:145:0.8);
   \node[red, inner sep=1pt, fill=white] at (2.5,2.8) {$p$};

 \node[red, inner sep=1pt, fill=white] at (2.5,2) {$\times$};

  \end{tikzpicture}}
\caption{Intersection of $G$-invariant graded arcs}\label{figureXandY}  
  \end{figure}
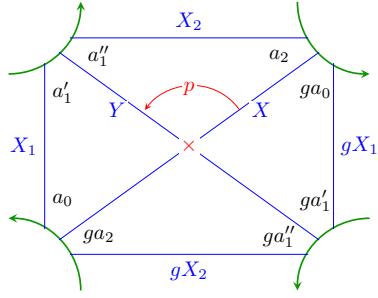

\begin{proposition}\label{prop: morphism X epsilon Y epsilon}
Let $X$ and $Y$ be $G$-invariant graded arcs intersecting in a fixed point as in Figure \ref{figureXandY}. Denote by $p$ the index from $X$ to $Y$. Consider now $\bS$ a $G$-invariant full system of arcs $\bS$ containing $Y$, $X_1$, $X_2$, $gX_1$ and $gX_2$. Let $\tilde{X}^{\varphi}$ be the object in $H^0(\Tw\cA_{\bS}*G)_+$ defined in Proposition \ref{prop:X epsilon independant}. Then in the category $H^0(\Tw\cA_{\bS}*G)_+$ we have the following equalities

\[\dim_K \Hom (\tilde{X}^{\varphi},Y^\pm[\ell]) = \left\{\begin{array}{ll} 1 & \textrm{if }\pm=(-1)^{p} \textrm{ and }\ell=p\\ 0 & \textrm{else.}\end{array}\right.\]
\[\dim_K \Hom (Y^\pm,\tilde{X}^{\varphi}[\ell]) = \left\{\begin{array}{ll} 1 & \textrm{if }\pm=(-1)^{1-p} \textrm{ and }\ell=1-p\\ 0 & \textrm{else.}\end{array}\right.\]
\end{proposition}

Hence in the setup of Figure \ref{figureXandY} we have the following morphisms linking $X^\pm$ and $Y^\pm$
$$\xymatrix@C=2cm{\ar @{} [dr] |{p\textrm{ even }}  X^{+}\ar[r]^{{\rm deg} p} & Y^+\ar[d]^{{\rm deg} (1-p)}\\ Y^-\ar[u]^{{\rm deg} (1-p)} & X^-\ar[l]_{{\rm deg} p}}
\xymatrix@C=2cm{\ar @{} [dr] |{p\textrm{ odd }}  X^+\ar[r]^{{\rm deg} p} & Y^-\ar[d]^{{\rm deg} (1-p)}\\ Y^+\ar[u]^{{\rm deg} (1-p)} & X^-\ar[l]_{{\rm deg} p}}
$$ 

\begin{proof}
The system $\bS$ looks as follows:

\begin{figure}[!h]
 \[
  \scalebox{0.8}{
  \begin{tikzpicture}[scale=1.2,>=stealth]
  \draw[dark-green,thick,->]  (1,0) arc (0:90:1);
  \draw[dark-green,thick,->] (0,3) arc (-90:0:1);
  \draw[dark-green,thick,->] (4,4) arc (-180:-90:1);
  \draw[dark-green,thick,->] (5,1) arc (90:180:1);
  \node at (0.75,2.75) {$a'_1$};
  \node at (1.25,3.25) {$a''_1$};
  \node at (4.25,1.25) {$ga'_1$};
  \node at (3.75,0.75) {$ga''_1$};
  \node at (4,3) {$ga_0a_2$};
  \node at (1,1) {$a_0ga_2$};
  \draw[blue] (0.5,0.85)--node[midway,left]{$X_1$}(0.5,3.15);
  \draw[blue] (0.85,3.5)--node[midway,above]{$X_2$}(4.15,3.5);
  \draw[blue] (0.7,3.3)--node[midway,right,above]{$Y$}(4.3,0.7);
  \draw[blue] (0.85,0.5)--node[midway,below]{$gX_2$}(4.15,0.5);
  \draw[blue] (4.5,0.85)--node[midway,right]{$gX_1$}(4.5,3.15);
 \node[red] at (2.5,2) {$\times$};

  \end{tikzpicture}}\]
  \caption{}
\end{figure}
where $p=|a_0|+|a'_1|$, and $a_1=a''_1a'_1$. 
Then we have $H^\ell(\Hom_{\Tw\cA}(\tilde{X},Y))=0$ for $\ell\neq p$ and 
$$H^{p}(\Hom_{\Tw\cA}(\tilde{X},Y))={\rm vect} (s^{-|a_0|}\otimes a'_1).$$
Therefore it is enough to compute $\mu^2(e_Y^\pm, \mu^2(s^{-|a_0|}\otimes a'_1\otimes 1_G,e^\varphi)).$

We have 
\begin{eqnarray*} \mu^2(s^{-|a_0|}\otimes a'_1\otimes 1_G,e^\varphi) & = & \frac{1}{2} \left(\mu^2(s^{-|a_0|}\otimes a'_1,1_{X_1})\otimes 1_G +(-1)^{|a_0|}\mu^{2}(s^{-|a_0|}\otimes a'_1,s^{\|a_1\|}a_0ga_2)\otimes g\right)\\
 & = & \frac{1}{2}\left( s^{-|a_0|}\otimes a'_1\otimes 1_G +(-1)^{|a_0|}\mu^3(s^{-|a_0|}\otimes a'_1,s^{\|a_1\|}\otimes a_0ga_2,s^{\|a_1\|}\otimes ga''_1ga'_1)\otimes g\right)\\
  & = & \frac{1}{2}\left(s^{-|a_0|}\otimes a'_1\otimes 1_G +(-1)^{|a_0|}s^{-|a_0|}\otimes \mu^3(a'_1,a_0ga_2,ga''_1ga'_1)\otimes g\right)\\ 
  & = & 
 \frac{1}{2}\left(s^{-|a_0|}\otimes a'_1\otimes 1_G +(-1)^{|a_0|+|a'_1|}s^{-|a_0|}\otimes ga'_1\otimes g\right)\\
  & = & \frac{1}{2}\left(  (s^{-|a_0|}\otimes a'_1\otimes 1_G +(-1)^{p}s^{-|a_0|}ga'_1\otimes g\right).
  \end{eqnarray*}
  Then we have 
  $$\mu^2(e_Y^\pm, \mu^2(s^{-|a_0|}\otimes a'_1\otimes 1_G,e^\varphi))= (-1)^p\frac{1}{4}(1\pm(-1)^p)s^{-|a_0|}\otimes \left(a'_1\otimes 1_G+(-1)^p ga'_1\otimes g\right).$$
Therefore this does not vanish if and only if $\pm=(-1)^p$.

In a similar fashion we have $H^\ell(\Hom_{\Tw\cA}(Y,\tilde{X}))=0$ for $\ell\neq 1-p$ and 
$$H^{1-p}(\Hom_{\Tw\cA}(Y,\tilde{X}))={\rm vect }(s^{-|a_2|}\otimes a''_1).$$
Hence one needs to compute 
$\mu^2(e^\varphi,\mu^2(s{-|a_2|}\otimes a''_1\otimes 1_G,e_Y^\pm))$. First one has
$$\mu^2(s^{-|a_2|}\otimes a''_1\otimes 1_G,e_Y^\pm)=s^{-|a_2|}\otimes a''_1\otimes \frac{1\pm g}{2}.$$
Finally setting $A:=  \mu^2(e^\varphi,s^{-|a_2|}\otimes a''_1\otimes \frac{1\pm g}{2}) $ we have
\begin{eqnarray*}A & = & \left(\mu^2(s^0\otimes 1_{X_2},s^{-|a_2|}\otimes a''_1)\pm (-1)^{|a_0|}\mu^2(s^{-\|a_\|}\otimes a_0ga_2,s^{-|a_2|}\otimes ga''_1)\right) \otimes \frac{1\pm g}{4}\\
 &= & \left( (-1)^{1-p} s^{|a_2|}\otimes a''_1\pm (-1)^{|a_0|}\mu^3(s^{\|a_1\|}\otimes a''_1a'_,s^{-\|a_1\|}\otimes a_0ga_2,s^{-|a_2|}\otimes ga''_1)\right) \otimes \frac{1\pm g}{4}\\ & = & ((-1)^{1-p}\pm 1)s^{-|a_2|}\otimes a''_1\otimes \frac{1\pm g}{4}.\end{eqnarray*}
So this does not vanish if and only if $\pm(-1)^{1-p}=1$ which finishes the proof. 
  
\end{proof}

\begin{remark}\label{rema::shift} Let $(\gamma,\grad)$ is a $G$-invariant simple graded arc, and let $\bS$ be a $G$-invariant full system of arcs containing $X:=(\gamma,\grad)$.  We can consider the arc $X':=(\gamma,\grad[1])$ and the full system $\bS':=(\bS\setminus X)\cup X'$. Then we have a canonical triangle equivalence between the categories $H^0(\Tw\cA_{\bS}*G)_+$ and $H^0(\Tw\cA_{\bS'}*G)_+$. Then Proposition \ref{prop: morphism X epsilon Y epsilon} implies that through this equivalence the object $X'^{\pm}$ of the category $\in H^0(\Tw\cA_{\bS'}*G)_+$ is sent to the object $X^{\mp}[1]$ of the category $ H^0(\Tw\cA_{\bS}*G)_+$. 
\end{remark}

\subsection{$G$-invariant (simple) bands with $[\gamma]=[g\gamma]$}

Let $(\gamma,\grad,V)$ be a primitive graded loop with a $K[x]$-module $V$. The action of $g$ on the corresponding object $B_{(\gamma,\grad,V)}$ satisfies $g(B_{(\gamma,\grad,V)})=B_{(g_*\gamma,g^* \grad,V)}$. Therefore we have that $gB\simeq B$ if one of the following occurs (see \cite{AmiotBrustle}):
\begin{enumerate}
\item either $\gamma$ is homotopic to $g_*\gamma$, or
\item $g_*\gamma$ is homotopic to $\gamma^{-1}$, and the scalar attached to $V$ is $\pm 1$.
\end{enumerate}

The aim of this subsection is to find a suitable $G$-invariant full system of arcs in which we can explicitly describe the band object $B_{(\gamma,\grad,V)}$ together with some explicit morphism $\varphi$ from $gB$ to $B$ with $\mu^2(\varphi,g \varphi)=1_B$. We restrict to the case of a simple graded loop with one dimensional $K[x]$-module $V$.

We treat seperately these two cases which are really different. We start with the case where $[\gamma]=g[\gamma]$.

\begin{proposition}\label{constructionsystem1}
Let $(\gamma, \grad)$ be a simple graded  loop on $\surf$ such that $g_*(\gamma)$ is homotopic to $\gamma$, and such that ${\rm Im}\gamma={\rm Im} g_*(\gamma)$. Then it is possible to find a $G$-invariant system of arcs containing arcs $(X_1,X_2,X_3, X_4,gX_1,gX_2,gX_3,gX_4)$ as in Figure~\ref{figure::8-arcs}.

\begin{figure}[!h]
 \[
  \scalebox{0.8}{
  \begin{tikzpicture}[scale=1.2,>=stealth]
  \draw[dark-green,thick,->]  (1,0) arc (0:65:1);
  \draw[dark-green,thick,->] (0.5,3.15) arc (-65:0:1);
  \draw[dark-green,thick,->] (4,4) arc (-180:0:1);
  \draw[dark-green,thick,->] (6,0) arc (0:180:1);
   \draw[dark-green,thick,->] (9.5,0.85) arc (115:180:1);
    \draw[dark-green,thick,->] (9,4) arc (-180:-115:1);

  \draw[blue] (0.5,0.85)--node[midway,left]{$X_3$}(0.5,3.15);
  \draw[blue] (0.85,3.5)--node[midway,above]{$X_1$}(4.15,3.5);
  \draw[blue] (0.7,3.3)--node[midway,above]{$X_4$}(4.3,0.7);
  \draw[blue] (0.85,0.5)--node[midway,below]{$X_2$}(4.15,0.5);
  \draw[blue] (5,1)--node[midway,right,left]{$gX_3$}(5,3);
  \draw[blue] (5.85,0.5)-- node[midway,below]{$gX_2$}(9.15,0.5);
   \draw[blue] (5.85,3.5)-- node[midway,above]{$gX_1$}(9.15,3.5);
    \draw[blue] (5.7,3.3)--node[midway,above]{$gX_4$}(9.3,0.7);
    \draw[blue] (9.5,0.85)--node[midway,right]{$X_3$}(9.5,3.15);

  \draw[thick,->,purple] (0,1.5)--(11,1.5);
  \node[purple,inner sep=2pt,fill=white] at (10,1.5){$(\gamma,\grad)$};

  \end{tikzpicture}}\]
  \caption{}\label{figure::8-arcs}
\end{figure}

\end{proposition}

\begin{proof} Since the images of $\gamma$ and $g_*\gamma$ coincide, the homotopy from $\gamma$ to $g_*\gamma$ can be assumed to be the translation by $\frac{1}{2}$, and the surface $\surf$ with $G$-action is locally a cylinder as follows on a $G$-invariant neighborhood of $\gamma$.

\begin{figure}[!h]
 \[
  \scalebox{0.8}{
  \begin{tikzpicture}[scale=1.2,>=stealth]

\draw (-1,0).. controls (-1,0.7) and (1,0.7)..(1,0);
\shadedraw[top color= blue!30] (1,0)--(1,-2.5).. controls (1,-3.2) and (-1,-3.2).. (-1,-2.5)--(-1,0).. controls (-1,-0.7) and (1,-0.7) ..(1,0);

\draw[purple, thick] (-1,-1.5).. controls (-1,-2.2) and (1,-2.2).. (1,-1.5);
\draw[blue, very thick] (0,1.5)--(0,-0.5);
\draw[blue, very thick] (0,-3)--(0,-4.5);
\draw[->] (-0.25,1) arc (-180:0:0.25);

  \end{tikzpicture}}\]
 \caption{}
\end{figure}

Let $C$ be a $G$-invariant cylinder containing $\gamma$, and let $\gamma_1$ and $\gamma_2$ be its boundary components. Since the winding number of $\gamma$ is zero, $\gamma_1$ and $\gamma_2$ are not the boundary of a closed subsurface of $\surf$. For $i=1,2$, there exists some simple curves $\alpha_i:[0,1]\to \surf\setminus C$ with $\alpha_i(0)\in \bM$ and $\alpha_i(1)=\gamma_i(0)$. We moreover can assume that $\alpha_i$ and $g\alpha_i$ do not intersect, by resolving  potential crossings. Then we set $$\tilde{\alpha}_i:=\alpha_i.(\gamma_i)_{|_{[0,1/2]}}.(g\alpha_i)^{-1}.$$ Since $g(\gamma_i)_{|_{[0,1/2]}}$ is homotopic to $(\gamma_i)_{|_{[1/2,1]}}$, we obtain that $\alpha_i.g\alpha_i$ is homotopic to $\gamma_i$. For $i=1,2$ we choose a grading on $\alpha_i$ and denote by $X_i$ the corresponding graded arc. It is then possible to complete the system $(X_1,X_2,gX_1,gX_2)$ into a full $G$-invariant system of graded arcs as desired. 

\end{proof}

Now the aim is to define the band object corresponding to the closed curve $(\gamma,\grad)$ as defined in \eqref{defBand} as a $G$-invariant twisted complex in the $G$-invariant system $\bS$. Fix some notation for the $\bM$-segments as follows. 

\begin{figure}[!h]
 \[
  \scalebox{0.8}{
  \begin{tikzpicture}[scale=1.2,>=stealth]
  
   \draw[dark-green,thick,->]  (1,0) arc (0:65:1);
  \draw[dark-green,thick,->] (0.5,3.15) arc (-65:0:1);
 
  \draw[dark-green,thick,->] (4,4) arc (-180:0:1);
  \draw[dark-green,thick,->] (6,0) arc (0:180:1);
  
   \draw[dark-green,thick,->] (9.5,0.85) arc (115:180:1);
    \draw[dark-green,thick,->] (9,4) arc (-180:-115:1);
   
  \node at (1,1) {$c$};
  \node at (0.75,3) {$a$};
  \node at (1.25,3.25) {$b'$};
  \node at (3.75,0.75) {$b$};
  \node at (4.25,1.25) {$a'$};
  \node at (4,3) {$c'$};
   \node at (6,1) {$gc$};
  \node at (5.75,3) {$ga$};
  \node at (6.25,3.25) {$gb'$};
  \node at (8.75,0.75) {$gb$};
  \node at (9.25,1.25) {$ga'$};
  \node at (9,3) {$gc'$};
  
  \draw[blue] (0.5,0.85)--node[midway,left]{$X_3$}(0.5,3.15);
  \draw[blue] (0.85,3.5)--node[midway,above]{$X_1$}(4.15,3.5);
  \draw[blue] (0.7,3.3)--node[midway,above]{$X_4$}(4.3,0.7);
  \draw[blue] (0.85,0.5)--node[midway,below]{$X_2$}(4.15,0.5);
  \draw[blue] (5,1)--node[midway,right,left]{$gX_3$}(5,3);
  \draw[blue] (5.85,0.5)-- node[midway,below]{$gX_2$}(9.15,0.5);
   \draw[blue] (5.85,3.5)-- node[midway,above]{$gX_1$}(9.15,3.5);
    \draw[blue] (5.7,3.3)--node[midway,above]{$gX_4$}(9.3,0.7);
    \draw[blue] (9.5,0.85)--node[midway,right]{$X_3$}(9.5,3.15);

  \draw[thick,->,purple] (0,1.5)--(11,1.5);
  \node[purple,inner sep=2pt,fill=white] at (10,1.5){$(\gamma,\grad)$};

  \end{tikzpicture}}\]
  \caption{}
\end{figure}

Then we have the relations 
$$|a|+|b|+|c|=1, \quad |a'|+|b'|+|c'|=1 \quad \textrm{and }|a|=|a'|,$$
the last relations coming from the fact that the winding number of $\gamma$ is $0$. 
For simplicity, we choose the arc $X_4$ in such a way that $|a|=|a'|=0$.  

From now on and in order to lighten the writing we omit the $s$-shift in the computations when the shift degree is clear. In particular, in all the twisted complexes, the differential are of degree $1$.  

\begin{proposition} Let $(\gamma, \grad)$ and $\bS$ be as above, and $\lambda$ be a scalar in $k^*$. Then the object in $\Tw\cA_{\bS}$ defined as follows 

$$B:=\xymatrix{X_3[p]\ar[r]^{\lambda a}\ar[dr]_{ga'} & X_4[p-1] \\ gX_3[p]\ar[r]^{\lambda ga}\ar[ur]^{a'}  & gX_4[p-1]}$$
is isomorphic in $H^0(\Tw \cA_{\bS})$ to the band object associated with the graded curve $(\gamma,\grad)$  and the parameter $\lambda^2$ as defined in \eqref{defBand}.

\end{proposition}

\begin{proof}  Let $\bS'$ be the full system of arc obtained from $\bS$ by flipping $gX_3$, and denote by $Z$ the flip of $gX_3$. 

\begin{figure}[!h]
 \[
  \scalebox{0.8}{
  \begin{tikzpicture}[scale=1.2,>=stealth]
  
   \draw[dark-green,thick,->]  (1,0) arc (0:65:1);
  \draw[dark-green,thick,->] (0.5,3.15) arc (-65:0:1);
 
  \draw[dark-green,thick,->] (4,4) arc (-180:0:1);
  \draw[dark-green,thick,->] (6,0) arc (0:180:1);
  
  \draw[dark-green,thick,->] (9.5,0.85) arc (115:180:1);
    \draw[dark-green,thick,->] (9,4) arc (-180:-115:1);

  \node at (0.75,3) {$a$};
  \node at (1.25,3) {$\beta$};
  \node at (2,3.25) {$\beta'$};
  \node at (5,2.75) {$gac'$};
  \node at (8.5,1) {$\delta$};
   \node at (9.25,1.25) {$ga'$};

  \draw[blue] (0.5,0.85)--node[midway,left]{$X_3$}(0.5,3.15);
  \draw[blue] (0.85,3.5)--node[midway,above]{$X_1$}(4.15,3.5);
  \draw[blue] (0.7,3.3)--node[midway,above]{$X_4$}(4.3,0.7);
  \draw[blue] (0.85,0.5)--node[midway,below]{$X_2$}(4.15,0.5);
  \draw[blue] (0.85,3.4)--node[midway,right,above]{$Z$}(9.15,0.6);
  \draw[blue] (5.85,0.5)-- node[midway,below]{$gX_2$}(9.15,0.5);
   \draw[blue] (5.85,3.5)-- node[midway,above]{$gX_1$}(9.15,3.5);
    \draw[blue] (5.7,3.3)--node[midway,above]{$gX_4$}(9.3,0.7);
    \draw[blue] (9.5,0.85)--node[midway,right]{$X_3$}(9.5,3.15);

  \draw[thick,->,purple] (0,1.5)--(11,1.5);
  \node[purple,inner sep=2pt,fill=white] at (10,1.5){$(\gamma,\grad)$};

  \end{tikzpicture}}\]
  \caption{}
\end{figure}

Here we may assume that $|\beta|=0$. 
Since $|a|=0$, the band object corresponding to $(\gamma,\grad,\lambda^2)$ in the category $H^0(\Tw \cA_{\bS'})$ is given by the complex 
$$B':=\xymatrix{X_3[p] \ar[rr]^{-\lambda^2 \beta a + \delta ga'} & & Z[p-1].}$$
 
Let us consider the following twisted complex 
$$B'':=\xymatrix@C=2cm{X_3\ar[r]_{\lambda b'a}\ar@/^1pc/[rr]^{ga'} & X_1 \ar[r]_{-\lambda ga c'} & gX_4[p-1]}. $$
This complex exists in both the categories $\Tw \cA_{\bS}$ and $\Tw \cA_{\bS'}$, so it is enough to construct an isomorphism between $B$ and $B''$ in $H^0(\Tw\cA_{\bS})$ and an isomorphism in $H^0(\Tw\cA_{\bS'})$ between $B''$ and $B'$. 

This is shown in the next two lemmas.

\end{proof}

\begin{lemma}
In the category $H^0(\Tw\cA_{\bS'})$, the (degree zero) morphism $1_{X_3}+ \delta \in \Hom_{\Tw \cA_{\bS'}}(B'',B')$ is an isomorphism whose inverse is given by (the degree zero) morphism $1_{X_3}-(-1)^{|b'|}\lambda^{-1} \beta'.$
\end{lemma}

\begin{proof}

By the relation $|\beta'|+|c'|+|\delta|=1$ we immediately get that $|\delta|=0$.
We have the following diagram.
$$\xymatrix@C=2cm{X_3\ar[r]_{\lambda \beta'\beta a}\ar@/^1pc/[rr]^{ga'}\ar[d]^{1_{X_3}} & X_1 \ar[r]_{-\lambda ga c'} & gX_4[p-1]\ar[d]^{\delta}\\ X_3[p] \ar[rr]^{-\lambda^2 \beta a + \delta ga'}  & & Z[p-1]} $$
where the horizontal maps have degree $1$, and the vertical maps have degree $0$, that is the vertical map $\delta$ has to be understood as $s^{|\delta|}\otimes \delta$, and the horizontal map and $ga c'$ has to be understood as $s^{\|ga c'\|}\otimes ga c'$. This is of importance in the computation of signs.

Then we compute 
$$\mu^1_{\Tw \cA'} (1_{X_3}) = -\lambda^2 \mu^2(\beta a, 1_{X_3})+\mu^2(\delta ga', 1_{X_3})=-\lambda^2 \beta a+ \delta ga',$$
and 
$$\mu^1_{\Tw\cA'}(\delta)= \mu^2(\delta,ga')-\Lambda^2\mu^3(\delta,gac',\beta'\beta a)= -\delta ga'+\lambda^2 \beta a.$$
Hence the map $1_{X_3}+\delta$ is in the kernel of $\mu^1$.

In the same way let us consider the diagram
$$\xymatrix@C=2cm{X_3[p] \ar[rr]^{-\lambda^2 \beta a + \delta ga'}\ar[d]^{1_{X_3}}  & & Z[p-1]\ar[dl]_{\beta'}\\ X_3\ar[r]^{\lambda \beta'\beta a}\ar@/_1pc/[rr]_{ga'} & X_1 \ar[r]^{-\lambda ga c'} & gX_4[p-1],} $$ 
then we can check the equalities
$$\mu^1(1_{X_3})= ga' +\lambda \beta'\beta a,$$
and 
$$\mu^1(\beta')=-\lambda^2\mu^2(\beta',\beta a)-\lambda \mu^3(ga c',\beta',\delta ga')= (-1)^{|b'|}(\lambda^2 \beta'\beta a+\lambda ga').$$
Hence we obtain that $1_{X_3}-(-1)^{|b'|}\lambda^{-1}\beta'$ is in the kernel of $\mu^1$.

Then one checks that 
$$\mu^2(\beta',\delta)=-\lambda \mu^3 (gac',\beta',\delta)-\lambda \mu^3(\beta',\delta,ga c')=-\lambda (-1)^{|b'|}(1_{gX_4}+1_{X_1}).$$
This implies that 
$$\mu^2(1_{X_3}-(-1)^{|b'|}\lambda^{-1}\beta',1_{X_3}+\delta)={\rm Id}_{B''}.$$

Likewise one has the equalitiy
$$\mu^2(\delta,\beta')=-\lambda\mu^3(\delta,ga c',\beta')=-(-1)^{|b'|}\lambda 1_Z,$$
which gives the relation
$$\mu^2(,1_{X_3}+\delta, 1_{X_3}-(-1)^{|b'|}\lambda^{-1}\beta')={\rm Id}_{B'}.$$

\end{proof}

\begin{lemma}In the category $H^0(\Tw \cA_\bS)$, the morphsim $1_{X_3}-c'+1_{gX_4}$ from $B''$ to $B$ is an isomorphism whose inverse is given by $1_{X_3+(-1)^{|b'|}b'+1_{gX_4}}.$
\end{lemma}

\begin{proof}
The proof is completely analogous to the previous one. We consider the following diagram
$$\xymatrix@C=2cm{ X_3\ar[r]^{\lambda \beta'\beta a}\ar@/^1pc/[rr]^{ga'}\ar@(dl,ul)[dd]_{1_{X_3}} & X_1 \ar[dl]^{c'}\ar[r]^{-\lambda ga c'} & gX_4[p-1]\ar[dl]^{1_{gX_4}} \\ gX_3[p]\ar[r]^{\lambda ga}\ar[dr]_{a'} & gX_4[p-1] \\ X_3[p]\ar[r]^{\lambda a}\ar[ur]^{ga'}  & X_4[p-1]}$$
where we check that 
$$\mu^1(1_{X_3}-c'+1_{gX_4})=0.$$

Likewise we have the diagram 
$$\xymatrix@C=2cm{  gX_3[p]\ar[r]^{\lambda ga}\ar[dr]_{a'} & gX_4[p-1]\ar[ddr]^{1_{gX_4}} \\ X_3[p]\ar[d]_{1_{X_3}}\ar[r]^{\lambda a}\ar[ur]^{ga'}  & X_4[p-1]\ar[d]^{b'}\\ X_3\ar[r]^{\lambda \beta'\beta a}\ar@/_1pc/[rr]_{ga'} & X_1 \ar[r]^{-\lambda ga c'} & gX_4[p-1] }$$
where we check that 
$$\mu^1(1_{X_3}+(-1)^{|b'|}b'+1_{bX_4})=0.$$

Finally combining the two diagrams, we check the equalities
$$\mu^2(1_{X_3}+(-1)^{|b'|}b'+1_{bX_4},1_{X_3}-c'+1_{gX_4})={\rm Id}_{B''},$$ 
and 
$$\mu^2(1_{X_3}-c'+1_{gX_4}, 1_{X_3}+(-1)^{|b'|}b'+1_{bX_4})={\rm Id}_{B}.$$

\end{proof}

The object $B$ defined above is clearly $G$-invariant. Hence we can define the objects $B^+$ (respectively $B^-$) in the category $H^0(\Tw \cA_{\bS} *G)_+$ as the object  $(\{B\}, e_{B}^{1_B})$ (respectively $(\{B\}, e_B^{-1_B})$).  

\begin{proposition}
Let $\bS$ and $B^{\pm}\in ((\Tw\cA)*G)_+ $ be as above. Then in the category $H^0(\Tw \cA_{\bS}*G)_+$ we have 
$$ \Hom (B^+, B^+[\ell])= \left\{ \begin{array}{cl} K & \textrm{ if }\ell=0,1\\ 0 & \textrm{else}\end{array}\right. \quad \textrm{and}\quad \Hom (B^+, B^-[\ell])= 0 \quad \textrm{for any }\ell.$$

 \end{proposition}
 
 \begin{proof}
We immediately see that $H^*(\Hom_{\Tw \cA}(B,B))$ is concentrated in degrees $0$ and $1$. We have 
$$H^0\Hom_{\Tw \cA}(B,B)={\rm vect} (1_B).$$
and we have \begin{equation}\label{H1} H^1(\Hom_{\Tw \cA}(B,B)) =  {\rm vect} (a, a', ga, ga')/\left\langle \begin{array}{c}\lambda a+ga'= 0\\ \lambda ga+ a'= 0\\ \lambda a+ a'=0 \\ \lambda ga+ga'=0\end{array}\right\rangle. \end{equation} which is one-dimensional.
Likewise we obtain that the space
$\Hom (B^+\oplus B^-,(B^+\oplus B^-)[1])$ is $2$-dimensional. It is then sufficient to compute 
$\mu^2(e_B^{\pm},\mu^2(a\otimes 1_G,e_B^+))$ in $H^0(\Tw \cA *G)$.
Then we have 
$$\mu^2(a\otimes 1_G, 1_B\otimes \frac{1+g}{2})=a\otimes \frac{1+g}{2},$$
 and $$\mu^2(1_B\otimes \frac{1\pm g}{2},a\otimes \frac{1+g}{2})= ((-1) a\pm (-1) ga))\otimes \frac{1+g}{4}.$$
 Finally using the relations \eqref{H1} in $H^1(\Hom_{\Tw\cA}(B,B)$ we obtain that $a=ga$ in the homology and hence $\mu^2(e_B^+,\mu^2(a\otimes 1_G,e_B^+))\neq 0$ while $\mu^2(e_B^{-},\mu^2(a\otimes 1_G,e_B^+))=0$.

 \end{proof}

\subsection{$G$-invariant simple bands with $[g\gamma]=[\gamma^{-1}]$}

In this subsection, we focus on the case where the closed curve $\gamma$ satsifies $[g_*\gamma]=[\gamma^{-1}]$. First note that such a curve can be written $\gamma=\alpha.g(\alpha^{-1})$where $\alpha$ is a curve with orbifold endpoints. Likewise it always has winding number $0$ (see \cite[Prop 5.17 (4)]{AmiotPlamondon}). Such a closed curve gives rise to $G$-invariant objects if and only if the parameter $\lambda$ attached to it satisfies $\lambda^2=1$. Therefore we obtain four non isomorphic objects attached with the curve $\gamma$ after splitting the idempotents. These objects will be called double tagged objects. The first tagging corresponds to the sign of the parameter, while the other tagging corresponds to the splitting of the idempotents. As we will see in Definition \ref{def:Bepsilon}, it will be more natural to attach one sign to each orbifold point lying on the curve $\gamma$, so that their product is equal to the parameter $\lambda=\pm 1$.

The plan of the subsection is as follows. We first exhibit in Proposition \ref{prop iso band G-invariant} an isomorphism $\varphi$ from the object $gB$ to $B$ satisfying  $\mu^2(\varphi, g\varphi)={\rm Id}$. We then define the double tagged objects $B^{\pm_P,\pm_Q}$ in Definition \ref{def:Bepsilon}. We compute the endomorphism algebra of these objects, showing that they are bricks in Proposition \ref{prop: morphism B epsilon B epsilon'}.  Then we compute the morphisms between double tagged objects and tagged objects of the system when they intersect in an orbifold point Proposition \ref{prop: morphism B epsilon X epsilon}.  This computation permits to show that the double tagging behaves well through the canonical functors $(\Psi_{\bS\to \bS'}*G)_+$ (Propostion \ref{prop:B epsilon independant system}). Finally, we compute the morphisms between two double tagged objects intersecting at an orbifold point (Proposition \ref{prop: morphism B epsilon B epsilon'}).

\subsubsection{Definition of double tagged objects}

\begin{proposition}\label{constructionsystem2}
Let $(\gamma,\grad)$ be a simple graded  loop on $\surf$ such that $g_*(\gamma)$ is homotopic to $\gamma^{-1}$ and such that ther image coincide on $\surf$. Then it is possible to find a $G$-invariant system of arcs containing arcs $(X=gX,Y=gY, Z, gZ)$ as in Figure \ref{figure1}.

\end{proposition}

 \begin{figure}[!h]

  \scalebox{0.8}{
  \begin{tikzpicture}[scale=1.2,>=stealth]
  \draw[dark-green,thick,->]  (1,0) arc (0:65:1);
  \draw[dark-green,thick,->] (0.5,3.15) arc (-65:0:1);
  \draw[dark-green,thick,->] (4,4) arc (-180:-115:1);
  \draw[dark-green,thick,->] (4.5,0.85) arc (115:180:1);
  \node at (0.75,2.75) {$c$};
  \node at (1.25,3.25) {$ga$};
  \node at (4.25,1.25) {$gc$};
  \node at (3.75,0.75) {$a$};
  \node at (4,3) {$gb$};
  \node at (1,1) {$b$};
  \draw[blue] (0.5,0.85)--(0.5,3.15);
  \draw[blue] (0.85,3.5)--node[midway,above]{$Z$}(4.15,3.5);
  \draw[blue] (0.7,3.3)--(4.3,0.7);
  \draw[blue] (0.85,0.5)--node[midway,below]{$gZ$}(4.15,0.5);
  \draw[blue] (4.5,0.85)--(4.5,3.15);
 \node[red] at (2.5,2) {$\times$};
 \node[red] at (0.5,2) {$\times$};
 \node[red] at (4.5,2) {$\times$};
 \node[red] at (0.25,1.75) {$P$};
 \node[red] at (2.5,1.75) {$Q$};
 
 \node[blue, inner sep= 1pt, fill=white] at (3.25,1.5) {$Y$};
 \draw[thick,purple,->] (0,2)--(6,2);
 \node[blue] at (0.25,2.5) {$X$};
 \node[blue] at (4.75,2.5) {$X$};
\node[purple, inner sep=1pt, fill=white] at (5,2) {$(\gamma,\grad)$};

\draw[red,->] (1,2) arc (0:90:0.5);
\draw[red,->] (3,2) arc (0:150:0.5);
\node[red,inner sep=1pt, fill=white] at (0.8,2.3) {$p$};
\node[red,inner sep=1pt, fill=white] at (2.8,2.3) {$q$};
 
  \end{tikzpicture}}
  \caption{System for a closed curve $\gamma$ with $[g\gamma]=[\gamma^{-1}]$.}\label{figure1}
  \end{figure}
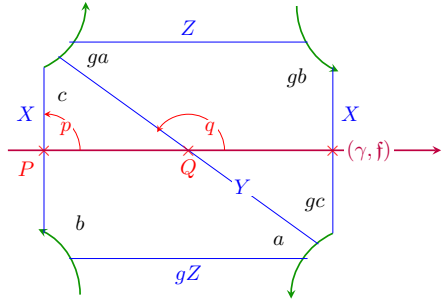

\begin{proof}
We first prove that on a $G$-invariant neighborhood of $\gamma$, the surface is a cylinder as follows. 

\begin{figure}[!h]
\[  \scalebox{0.8}{
  \begin{tikzpicture}[scale=1.2,>=stealth]

\shadedraw[top color=blue!30] (-1.5,1)..controls (-1,1) and (-1,-1).. (-1.5,-1)--(1.5,-1).. controls (2,-1) and (2,1)..(1.5,1)--(-1.5,1);
\draw (-1.5,1).. controls (-2,1) and (-2,-1).. (-1.5,-1);

\draw[thick, blue] (0,2.5)--(0,1);
\draw[thick, blue] (0,-1)--(0,-2);

\draw[->] (-0.25,2) arc (-180:0:0.25);

\node[red] at (0,1) {$\times$};
\node[red] at (0,-1) {$\times$};

\draw[purple,thick] (0,1)..controls (0.5,1) and (0.5,-1).. (0,-1);

\end{tikzpicture}}
\]
\caption{}
\end{figure}

\noindent
The proof is completely analogous to the one in Proposition \ref{constructionsystem1}. 

We can assume that contains exactly 2 fixed points of $g$, that will be denoted by $P$ and $Q$. 
The curve $\gamma$ can be written $\alpha.(g\alpha)$ where $\alpha :[0,1]\to \surf$ is an immersion with $\alpha(0)$ and $\alpha(1)$ being the two fixed points mentionned above. This fact was already proved in \cite[Prop 4.17]{AmiotPlamondon} using fundamental groups of orbifolds.

It is then enough to show the existence of an arc intersecting $\gamma$ in one of the fixed point. The existence of such an arc is clear since $\gamma$ is not the boundary of a closed subsurface of $\surf$ (since its winding number is zero). 
  
\end{proof}   

\begin{proposition}\label{prop iso band G-invariant} Let $(\gamma,\grad)$ and $\mathbb S$ be as in Figure \ref{figure1}, and $\lambda$ be a scalar in $K^*$. Let $B$ be the object $B:=B_{(\gamma,\grad,\lambda)}$ in $\Tw \cA_{\bS}$ defined as follows 
$$\xymatrix{X[p]\ar[rr]^{\lambda (-1)^{\|c\|}c+gc} && Y[q-1]}.$$ Then $B$ is isomorphic to $gB$ in $H^0\Tw \cA_{\bS}$ if and only if $\lambda=\pm 1$. In that case, an isomorphism $gB\to B$ is given by 
$\varphi:=(-1)^p\lambda 1_X+(-1)^{q-1} 1_Y$ and satisfies $\mu^2(g\varphi,\varphi)=1_B$.
\end{proposition}

\begin{proof}
We consider the following morphisms
\[\xymatrix{X[p]\ar[rr]^{\lambda (-1)^{\|c\|}gc+c}\ar[d]^{1_X} && Y[q-1]\ar[d]^{1_Y}\\ X[p]\ar[rr]^{\lambda (-1)^{\|c\|}c+gc} && Y[q-1] }\]
An immediate calculation gives  $$\mu^1(1_X)=\lambda (-1)^{\|c\|}c+gc \textrm{ and }\mu^1(1_Y)=-c-\lambda (-1)^{\|c\|}gc.$$
Hence $H^0(\Hom_{\Tw\cA}(gB,B))\neq 0$ if and only if $\lambda^2=1$ and in that case it is generated by $\varphi:=(-1)^p\lambda 1_X+(-1)^{q-1}1_Y$ since we have $p+|c|=q$. It is then straightforward to check that $\mu^2(\varphi,g\varphi)=1_B.$

\end{proof}

\begin{definition}\label{def:Bepsilon} Let $(\gamma,\grad)$ and $\bS$ be as in Figure \ref{figure1}. For $\epsilon_P$ and $\epsilon_Q$ in $\{\pm 1\}$ and $\lambda=\epsilon_P\epsilon_Q$ we define the object $B^{\epsilon_P,\epsilon_Q}_{\bS}$ in $H^0(\Tw\cA_{\bS}*G)_+$
$$B^{\epsilon_P,\epsilon_Q}_{\bS}:= \left(\{B_{(\gamma,\grad,\lambda)}\},1_X\otimes \frac{1+\epsilon_P(-1)^pg}{2}+1_Y\otimes \frac{1+\epsilon_Q(-1)^{q-1}g}{2}\right).$$

With the notation of the begining of the section we have $B^{\epsilon_P,\epsilon_Q}=B^\varphi$ with $\varphi:=\epsilon_P(-1)^p 1_X+\epsilon_Q(-1)^{q-1}1_Y$ which is an isomorphism from $gB_{\lambda}$ to $B_{\lambda}$ with $\lambda=\epsilon_P\epsilon_Q$. In particular we have $B^{\epsilon_P,\epsilon_Q}\oplus B^{-\epsilon_P,-\epsilon_Q}=B_{(\gamma,\grad,\lambda)}.$ These objects will be called \emph{double tagged objects}.
\end{definition}

\begin{remark} 
We will see in Proposition \ref{prop:B epsilon independant system} that if $\bS$ and $\bS'$ are two systems of arcs containing arcs $X$, $Y$ and $Z$ (resp. $X'$, $Y'$, $Z'$) as in Figure \ref{figure1} then 
$$\Psi_{\bS \to \bS'}(B^{\epsilon_P,\epsilon_Q}_{\bS})=B^{\epsilon_P,\epsilon_Q}_{\bS'}.$$
In order to do so, we first need to compute morphisms between $B^{\epsilon_P,\epsilon_Q}_{\bS}$ and the arcs $X^\pm$ and $Y^\pm$.
\end{remark}

\subsubsection{Endomorphism of the double tagged object}

\begin{proposition}\label{prop: morphism B epsilon B epsilon'}
Let $(\gamma,\grad)$ and $\bS$ be as above. Then in the category $H^0(\Tw\cA_{\bS}*G)_+$, we have 
\[\dim_k\Hom (B^{\epsilon_P,\epsilon_Q}_{\bS},B^{\epsilon'_P,\epsilon'_Q}_{\bS}[\ell])= \left\{ \begin{array}{ll} 1 & \textrm{if }\ell=0, \quad \epsilon_P=\epsilon'_P \textrm{ and } \epsilon_Q=\epsilon'_Q \\ 1 & \textrm{if }\ell=1, \quad \epsilon_P=-\epsilon'_P \textrm{ and } \epsilon_Q=-\epsilon'_Q \\ 0 & \textrm{else.}\end{array}\right.\]
\end{proposition}

\begin{proof}
Denote by $\lambda=\epsilon_P\epsilon_Q$ and $\lambda'=\epsilon_P'\epsilon'_Q$ and by $B:=B_{(\gamma,\grad,\lambda)}$ and $B':=B_{(\gamma,\grad,\lambda')}$. Then we have 
$$B^{\epsilon_P,\epsilon_Q}\oplus B^{-\epsilon_P,-\epsilon_Q}=B \textrm{ and }B^{\epsilon'_P,\epsilon'_Q}\oplus B^{-\epsilon'_P,-\epsilon'_Q}=B'.$$
Since the curve $\gamma$ is simple, we already know that 
$$\dim_k \Hom(B,B'[\ell]) = \left\{\begin{array}{ll} 1 & \textrm{if }\ell=0,1 \textrm{ and } \lambda=\lambda'\\ 0 & \textrm{else}\end{array}\right.$$

Now assume that $\lambda=\lambda'=\epsilon_P\epsilon_Q=\epsilon'_P\epsilon'_Q$, that is $B=B'$. Consider the following diagram 

\[\xymatrix@C=2.5cm{X[p]\ar@<0.5ex>[drr]^{c}\ar@<-0.5ex>[drr]_{gc}\ar[rr]^{\lambda (-1)^{\|c\|}c+gc}\ar[d]_{1_X} && Y[q-1]\ar[d]^{1_Y}\\ X[p]\ar[rr]_{\lambda (-1)^{\|c\|}c+gc} && Y[q-1] }\]

By the previous computation we have that $$\mu^1(1_X)=\epsilon_P\epsilon_Q (-1)^{\|c\|}c+gc$$ hence 
$$H^1(\Hom_{\Tw \cA}(B,B))={\rm vect}(s^{\| c\|}\otimes c)$$
modulo the relation $\epsilon_P\epsilon_Q (-1)^{\|c\|}c+gc$.

One needs then to compute $\mu^2(e_B,\mu^2(s^{\|c\|}\otimes c\otimes 1_G,e_{B'}))$ with $e_B:=1_X\otimes \frac{1_G+\epsilon_P(-1)^pg}{2}+1_Y\otimes \frac{1_G+\epsilon_Q(-1)^{q-1}g}{2}$ and $e_{B'}:=1_X\otimes \frac{1_G+\epsilon'_P(-1)^pg}{2}+1_Y\otimes \frac{1_G+\epsilon'_Y(-1)^{q-1}g}{2}$.
A direct computation yields 
$$\mu^2(e_B,\mu^2(s^{\|c\|}\otimes c\otimes 1_G,e_{B'}))=(-1) s^{\|c\|}\otimes \left(c+\epsilon_P\epsilon'_Q(-1)^{p+q-1} gc\right)\otimes \frac{1_G+\epsilon'_P(-1)^pg}{4}.$$
Therefore using the relation in $H^1$ and the fact that $p+|c|=q$ we obtain a non zero morphism if and only if $\epsilon_Q=-\epsilon'_Q.$

\end{proof}

\begin{remark}
In the case where the surface is a cylinder with an involution as in the proof of Proposition~\ref{constructionsystem1}, then the category $H^0(\Tw\cA*G)_+$ is equivalent to the derived category of the path algebra of type $\widetilde{\mathbf{D} _n}$ ({\it cf} Section~\ref{Section5}). One can check that the two objects $B^{\epsilon_P,\epsilon_Q}$ and $B^{-\epsilon_P,-\epsilon_Q}$ corresponds to objects which are on the mouth of tubes of rank $2$ in the Auslander-Reiten quiver (compare also with \cite[Subsection 6.1]{AmiotPlamondon}). They are indeed bricks, but there is a non trivial extension between them.  

\end{remark}

\subsubsection{Morphisms between a double tagged object and a tagged object}

\begin{proposition}\label{prop: morphism B epsilon X epsilon} Let $(\gamma,\grad)$ and $\bS$ as in Figure \ref{figure1}.
In the category $H^0(\Tw\cA_{\bS} *G)_+$ we have the following equalities
\[\dim_k \Hom (B^{\epsilon_P,\epsilon_Q}_{\bS},X^\pm[\ell]) = \left\{\begin{array}{ll} 1 & \textrm{if }\pm\epsilon_P=(-1)^p \textrm{ and }\ell=p\\ 0 & \textrm{else.}\end{array}\right.\]

\[\dim_k \Hom (X^\pm,B^{\epsilon_P,\epsilon_Q}_{\bS}[\ell]) = \left\{\begin{array}{ll} 1 & \textrm{if }\pm\epsilon_P=(-1)^{1-p} \textrm{ and }\ell=1-p\\ 0 & \textrm{else.}\end{array}\right.\]

\[\dim_k \Hom (B^{\epsilon_P,\epsilon_Q}_{\bS},Y^\pm[\ell]) = \left\{\begin{array}{ll} 1 & \textrm{if }\pm\epsilon_Q=(-1)^{q} \textrm{ and }\ell=q\\ 0 & \textrm{else.}\end{array}\right.\]

\[\dim_k \Hom (Y^\pm, B^{\epsilon_P,\epsilon_Q}_{\bS}[\ell]) = \left\{\begin{array}{ll} 1 & \textrm{if }\pm\epsilon_Q=(-1)^{1-q} \textrm{ and }\ell=1-q\\ 0 & \textrm{else.}\end{array}\right.\]

\end{proposition}
 
Hence we have the following morphisms linking $X^\pm$ and $B^{\epsilon_P,\epsilon_Q}_{\bS}$
$$\xymatrix@C=2cm{\ar @{} [dr] |{p\textrm{ even }}  B^{+_P,\pm_Q}\ar[r]^{{\rm deg} p} & X^+\ar[d]^{{\rm deg} (1-p)}\\ X^-\ar[u]^{{\rm deg} (1-p)} & B^{-_P,\pm_Q}\ar[l]_{{\rm deg} p}}
\xymatrix@C=2cm{\ar @{} [dr] |{p\textrm{ odd }}  B^{+_P,\pm_Q}\ar[r]^{{\rm deg} p} & X^-\ar[d]^{{\rm deg} (1-p)}\\ X^+\ar[u]^{{\rm deg} (1-p)} & B^{-_P,\pm_Q}\ar[l]_{{\rm deg} p}}
$$

\begin{proof}
We have $H^\ell(\Hom_{\Tw\cA}(B,X))=0$ for $\ell\neq p$ and 
$H^\ell(\Hom_{\Tw\cA}(B,X)={\rm vect}(s^{-p}\otimes 1_X)$. Therefore it is enough to compute 
$\mu^2(e_X^\pm, \mu^2(s^{-p}\otimes 1_X\otimes 1_G, e_B))$ with $e_B$ as above. A direct computation yields 
$$\mu^2(e_X^\pm, \mu^2(s^{-p}\otimes 1_X\otimes 1_G, e_B))=(-1)^p[1\pm\epsilon_P(-1)^p]s^{-p}\otimes 1_X\otimes \frac{1_G+(-1)^p\epsilon_Pg}{2},$$
which is not zero if and only if $\pm\epsilon_P=(-1)^p$. 

Therefore, there is a morphism from $B^{\epsilon_P,\epsilon_Q}$ to $X^+$ if and only if $\epsilon_P=(-1)^p$ and this morphism has degree $p$. 

\smallskip
In the category $\Tw\cA$ there are morphisms in degree $-p$ and $1-p$ as follows
$$\xymatrix{X[0] \ar[d]_{s^{p}\otimes 1_X}\ar@<0.5ex>[drr]^{s^{q-1}\otimes c}\ar@<-0.5ex>[drr]_{s^{q-1}\otimes gc} &&\\ X[p\ar[rr]_{(-1)^{\|c\|}\epsilon_P\epsilon_Q c+gc} && Y[q-1] }.$$
We compute
$$\mu^1(s^p\otimes 1_X)=(-1)^{\|c\|p}((-1)^{\|c\|}\epsilon_P\epsilon_Q s^{q-1}\otimes c+s^{q-1}\otimes gc).$$
Hence we obtain
$$H^{1-p}(\Hom_{\Tw\cA}(X,B))={\rm vect} (s^{q-1}\otimes c, s^{q-1}\otimes gc)/\langle (-1)^{\|c\|}\epsilon_P\epsilon_Q s^{q-1}\otimes c+s^{q-1}\otimes gc\rangle.$$

It is then enough to compute 
$\mu^2(e_B,\mu^2(s^{q-1}\otimes c\otimes 1_G,e_X^\pm)).$ A direct computation gives
$$\mu^2(e_B,\mu^2(s^{q-1}\otimes c\otimes 1_G,e_X^\pm))=(-1)^{p-1}s^{q-1}\otimes (c\pm(-1)^{q-1}\epsilon_Q gc)\otimes \frac{1\pm g}{2}.$$
Hence using the relation in $H^{1-p}(\Hom_{\Tw\cA}(X,B))$, one obtain that $\mu^2(e_B,\mu^2(s^{q-1}\otimes c\otimes 1_G,e_X^\pm))$ is non zero in cohomology if and only if $(-1)^{1-p}=\pm\epsilon_P.$

\medskip

In the same way we have 
$$H^{1-q}(\Hom_{\Tw\cA}(Y,B))={\rm vect}(s^{q-1}\otimes 1_Y),$$
and 
$$\mu^2(e_B,\mu^2(s^{q-1}\otimes 1_Y\otimes1_G,e_Y^\pm))=(-1)^{q-1}(1\pm\epsilon_Q(-1)^{q-1})s^{q-1}\otimes 1_Y\otimes \frac{1_G\pm g}{4},$$
which is not zero if and only if $\pm\epsilon_Q=(-1)^{1-q}$. 

We also have
$$H^{q}(\Hom_{\Tw\cA}(B,Y))={\rm vect}(s^{-p}\otimes c,s^{-p}\otimes gc)/\langle \epsilon_P\epsilon_Q(-1)^{\|c\|}c+gc\rangle,$$
and $$\mu^2(e_Y^\pm,\mu^2(s^{-p}\otimes c\otimes 1_G,e_B))=(-1)^{q}s^{-p}\otimes (c\pm(-1)^p\epsilon_P gc)\otimes \frac{1_G+\epsilon_P(-1)^p g}{4}.$$
So this does not vanish in cohomology if and only if $\pm\epsilon_Q=(-1)^{q}.$

\end{proof}

\begin{proposition}\label{prop:B epsilon independant system} 
Let $(\gamma,\grad)$ be a graded simple curve such that $g\gamma=\gamma^{-1}$, and let $\bS$ and $\bS'$ be two full system of graded arcs as in Figure \ref{figure1}. Then we have isomorphisms in $H^0(\Tw\cA_{\bS'}*G)_+$
$$\Psi_{\bS\to \bS'}(B_{\bS}^{\epsilon_P,\epsilon_Q})\simeq B_{\bS'}^{\epsilon_P,\epsilon_Q}.$$
\end{proposition}

\begin{proof}
We have $$B^{\epsilon_P,\epsilon_Q}_{\bS}\oplus B^{-\epsilon_P,-\epsilon_Q}_{\bS}=B_{(\gamma,\grad,\epsilon_P\epsilon_Q,\bS)},\textrm{ and } B^{\epsilon_P,\epsilon_Q}_{\bS'}\oplus B^{-\epsilon_P,-\epsilon_Q}_{\bS'}=B_{(\gamma,\grad,\epsilon_P\epsilon_Q,\bS')}.$$ Moreover by Theorem \ref{thm HKK indec obj} we know that there is an isomorphism in $H^0(\Tw\cA_{\bS'})$ between $\Psi_{\bS\to \bS'}(B_{(\gamma,\grad,-\epsilon_P\epsilon_Q,\bS)})$ and $B_{(\gamma,\grad,-\epsilon_P\epsilon_Q,\bS')}$. Therefore we have either 
$$\Psi_{\bS\to \bS'}(B_{\bS}^{\epsilon_P,\epsilon_Q})\simeq B_{\bS'}^{\epsilon_P,\epsilon_Q} \textrm{ or }\Psi_{\bS\to \bS'}(B_{\bS}^{\epsilon_P,\epsilon_Q})\simeq B_{\bS'}^{-\epsilon_P,-\epsilon_Q}.$$

The arcs $X$ and $X'$ intersects the curve $\gamma$ in $P$. Assume that  they intersect as follows : 

\begin{figure}[!h]
\[  \scalebox{0.8}{
  \begin{tikzpicture}[scale=1.2,>=stealth]

\draw[purple, thick] (-1.5,0)--(1.5,0);
\draw[blue, thick] (1,1.5)--(-1,-1.5);
\draw[blue,thick] (-1,1.5)--(1,-1.5);
\node[red] at (0,0) {$\times$};
\node[red] at (0,-0.4) {$P$};
\node[purple, inner sep=1pt, fill=white] at (1,0) {$\gamma$};
\node[blue, inner sep=1pt,fill=white] at (0.5,-0.75) {$X$};
\node[blue, inner sep=1pt,fill=white] at (-0.5,-0.75) {$X'$};
\draw[red,->] (0.8,0) arc (0:60:0.8);
\draw[red,->] (0.5,0) arc (0:120:0.5);
\node[red, inner sep=1pt, fill=white] at (0.7,0.5) {$p'$};
\node[red, inner sep=1pt, fill=white] at (0,0.5) {$p$};

\end{tikzpicture}}
\]
\caption{}
\end{figure}

Denote by $p$ the index from $(\gamma,\grad)$ to $X$, and by $p'$ the index from $(\gamma,\grad)$ to $X'$. Then there is non zero morphisms in $\cF(\surf)$
$$\Hom_{\cF}(B,X'[p']) \textrm{ and } \Hom_{\cF}(X'[p'],X[p])$$
whose composition gives rise to a non zero morphism in $\Hom_{\cF}(B,X[p]).$ 
Note that here, because of Theorem \ref{thm HKK indec obj}, we can avoid subscript notation with respect to the system $\bS$ or $\bS'$.

Denote by $\epsilon'_P$ and $\epsilon'_Q$ the signs such that $\Psi_{\bS\to \bS'}(B^{+,\epsilon_Q}_{\bS})\simeq B^{\epsilon'_P,\epsilon'_Q}_{\bS'}.$ Then we have $\epsilon_Q=\epsilon'_P\epsilon'_Q$.

We first assume that $p$ is even to avoid heavy sign notation.  Now by Proposition \ref{prop: morphism B epsilon X epsilon} there is a non zero morphism in $H^0(\Tw\cA_{\bS}*G)_+$ from $B_{\bS}^{+,\epsilon_Q}$ to $X^{+}_{\bS}[p]$. Therefore there is a non zero morphism from $\Psi_{\bS\to \bS'}(B_{\bS}^{+,\epsilon_Q})$ to $\Psi_{\bS\to \bS'}(X_{\bS}^{+}[p])$ in the category $H^0(\Tw\cA_{\bS'}*G)_+$. By Proposition \ref{prop:X epsilon independant} we have $\Psi_{\bS\to \bS'}(X^+_{\bS})=X^+_{\bS'}$, so there is a non zero morphism from $\Psi_{\bS\to \bS'}(B_{\bS}^{+,\epsilon_Q})\simeq B^{\epsilon'_P,\epsilon'_Q}_{\bS'}$ to $X^{+}_{\bS'}[p]$.  This morphism factors through $X'[p']=X'^+[p']\oplus X'^-[p']$. But by Proposition \ref{prop: morphism X epsilon Y epsilon} we have 
$$\Hom (X'^\epsilon[p'],X^+[p])\neq 0$$ if and only if $\epsilon=(-1)^{p-p'}$. And by Proposition \ref{prop: morphism B epsilon X epsilon} we have 
$$\Hom (B^{\epsilon_P,\epsilon_Q}_{\bS'},X^\epsilon[p'])\neq 0$$ if and only if $\epsilon=\epsilon_P(-1)^{p'}.$ Therefore we have 
$$\Hom (B_{\bS'}^{\epsilon_P,\epsilon_Q},X^+[p])\neq 0$$ if and only if $\epsilon_P(-1)^{p'}=(-1)^{p-p'}.$ Since $p$ is even this is equivalent to $\epsilon_P=1$, hence we have 
$$\Psi_{\bS\to \bS'}(B^{+,\epsilon_Q}_{\bS})\simeq B^{+,\epsilon'_Q}_{\bS'}.$$ Since $\epsilon'_Q\epsilon'_P=\epsilon_Q$ we conclude. 
The case where $p$ is odd is completely similar. The case where the curves $X$, $X'$ and $B$ intersect in the order $(B,X,X')$ is also similar.

\end{proof} 

\subsubsection{Morphisms between double tagged objects}

\begin{proposition}\label{prop: morphism B epsilon B' epsilon}
Let $\gamma_1$ and $\gamma_2$ be two graded simple loops on $\surf$ with $[g\gamma_i]=[\gamma^{-1}_i]$ and intersecting in an orbifold point $P$. Denote by $B_1^{\epsilon_P,\epsilon_Q}$ and $B_2^{\epsilon'_P,\epsilon'_R}$ the corresponding double tagged objects in $\mathcal{F}^G_+$ and by $p$ the index between $\gamma_1$ and $\gamma_2$ at $P$. Then we have 
\[\dim_k \Hom (B_1^{\epsilon_P,\epsilon_Q},B_2^{\epsilon'_P,\epsilon'_R} [\ell]) = \left\{\begin{array}{ll} 1 & \textrm{if }\epsilon_P\epsilon'_P=(-1)^p \textrm{ and }\ell=p\\ 0 & \textrm{else.}\end{array}\right.\]

\[\dim_k \Hom (B_2^{\epsilon'_P,\epsilon'_R},B_1^{\epsilon_P,\epsilon_Q}[\ell]) = \left\{\begin{array}{ll} 1 & \textrm{if }\epsilon'_P\epsilon_P=(-1)^{1-p} \textrm{ and }\ell=1-p\\ 0 & \textrm{else.}\end{array}\right.\]

\end{proposition}

\begin{proof}

The proof is similar to the previous one. We denote by $X$ a $G$-invariant arc passing through the point $P$ as in the following picture and such that the index from $\gamma_1$ to $X$ is $0$, and we work in a full $G$-invariant system of graded arcs containing $X$.

\begin{figure}[!h]
\[  \scalebox{0.8}{
  \begin{tikzpicture}[scale=1.2,>=stealth]

\draw[purple, thick] (-1.5,0)--(1.5,0);
\draw[blue, thick] (1,1.5)--(-1,-1.5);
\draw[cyan,thick] (-1,1.5)--(1,-1.5);
\node[red] at (0,0) {$\times$};
\node[red] at (0,-0.4) {$P$};
\node[purple, inner sep=1pt, fill=white] at (1,0) {$\gamma_1$};
\node[cyan, inner sep=1pt,fill=white] at (0.5,-0.75) {$\gamma_2$};
\node[blue, inner sep=1pt,fill=white] at (-0.5,-0.75) {$X$};
\draw[red,->] (0.8,0) arc (0:60:0.8);
\draw[red,->] (0.5,0) arc (0:120:0.5);
\node[red, inner sep=1pt, fill=white] at (0.7,0.5) {$0$};
\node[red, inner sep=1pt, fill=white] at (0,0.5) {$p$};

\end{tikzpicture}}
\]
\caption{}
\end{figure}

Assume that $p$ is even. We prove that there is no morphism from $B_1^{+_P,+_Q}$ to $B_2^{-_P,-_R}$ and that the space of morphism from $B_1^{+_P,+_Q}$ to $B_2^{+_P,+_R}$ is one dimensional. 

We know that the intersection induces a degree $p$ morphism from $(B_1)_1$ to $(B_2)_1$ in $H^0(\Tw\cA)$ factorising through $X$. Hence the space of morphism from $B_1^{+_P,+_Q}\oplus B_1^{-_P,-_Q}$ to $B_2^{+_P,+_R}\oplus B_2^{-_P,-_R}$ in $H^0(\Tw\cA*G)_+$ is $2$-dimensional and of degree $p$, and these morphisms factor through $X^+\oplus X^-$. The space of morphism from $B_1^{+_P,+_Q}$ to $B_2^{+_P,+_R}\oplus B_2^{-_P,-_R}$ is then one dimensional and factors through $X^+\oplus X^-$. By proposition \ref{prop: morphism B epsilon X epsilon}, there is no morphism from $B_1^{+_P,+_Q}$ to $X^-$, and no morphism from $X^+$ to $B_2^{-_P,-_Q}$, hence there is no morphism from $B_1^{+_P,+_Q}$ to $B_2^{-_P,-_R}$. Thus the space of morphism from $B_1^{+_P,+_Q}$ to $B_2^{+_Q,+_R}$ is one dimensional. The other cases can be treated in a similar fashion by considering the morphisms from $(B_1)_1$ and $(B_1)_{-1}$ to $(B_2)_1$ and $(B_2)_{-1}$ in $H^0(\Tw\cA)$. 

\end{proof}

\section{Tilting objects and skew-gentle algebras}\label{Section5}

In this section, $(\surf,\bM,\eta,g)$ is a $G$-graded marked surface as in Definition \ref{def:Gsurf}. Then, we can naturally consider the orbifold surface $\surf/G$, it is an orbifold surface with orbifold points correpponding to  the points fixed by the automorphism $g$. The aim is to describe some tilting objects in the category $H^0\Tw\big( (\cA*G)_+ \big)$ in terms of collections of (tagged) arcs on the orbifold surface $\surf/G$.

For the graded surface $\surf$, we denote by $\mathbf{g}$ its genus, by $\mathbf{b}$ its number of boundary components and by $\mathbf{m}$ the number of boundary segments. For the orbifold surface $\surf/G$ we denote by $\bar{\mathbf{g}}$ its genus, by $\bar{\mathbf{b}}$ its number of boundary components and by $\bar{\mathbf{m}}=\frac{\mathbf{m}}{2}$ its number of marked segments, and by $\mathbf{x}$ the number of orbifold points, that is the number of fixed points of $\surf$ by the action of $g$. 

\subsection{Formal generators and skew-gentle algebras}

\begin{definition} 
A \emph{$G$-graded dissection} $\mathbb D$ on $(\surf,\bM,\eta,g)$ is a system of graded arcs globally invariant by $g$ on $\surf$ cutting out the surface $\surf$ into discs with exactly one non marked boundary component in its boundary, or annuli with one non marked boundary component. 
\end{definition}

A $G$-graded dissection always exists and contains exactly $\mathbf{m}+\mathbf{b}+2\mathbf{g}-2$ arcs by Proposition 1.12 in \cite{AmiotPlamondonSchroll}. 

Our use of~$G$-graded dissections will be giving us formal generators in the category of twisted complexes over the skew-group~$A_\infty$-algebra using the following general facts.

\begin{definition}
 Let~$\cA$ and~$\cB$ be~$A_\infty$-categories with a strict action of a finite group~$G$.  A strictly unital strict~$A_\infty$-functor~$F:\cA\to \cB$ is~\emph{$G$-equivariant} if~$g\cdot FX = F (g\cdot X)$ for all objects~$X$ of~$\cA$ and~$g\cdot Fa = F(g\cdot a)$ for all morphisms~$a$ of~$\cA$.
\end{definition}

\begin{proposition}\label{prop::G-equivariant-functors}
 Let~$F:\cA\to \cB$ be a~$G$-equivariant strictly unital strict~$A_\infty$-functor.
 \begin{enumerate}
  \item There is an induced strictly unital strict~$A_\infty$-functor~$F*G:\cA*G\to \cB*G$ defined by~$F*G(\{X\}) = \{FX\}$ and~$F*G (a\otimes g) = Fa\otimes g$.
  \item It~$F$ is a quasi-equivalence, then so is~$F*G$.
 \end{enumerate}
\end{proposition}

\begin{corollary}\label{coro::formal-generator}
 Let~$\cA$ be an~$A_\infty$-category with a strict action of a finite group~$G$, and let~$T$ be a~$G$-invariant formal generator in~$\Tw\cA$.  Assume further that there exists a~$G$-equivariant quasi-equivalence~$H^*\End(T) \to \End(T)$.  Then~$\{T\}$ is a formal generator in~$\big((\Tw\cA)*G\big)_+$.  In particular, the triangulated categories~$H^0\big((\Tw\cA)*G\big)_+$ and~$\per(H^*\End(T))$ are equivalent.
\end{corollary}
\begin{proof}
 Apply Proposition~\ref{prop::G-equivariant-functors} to the quasi-equivalence~$H^*\End(T) \to \End(T)$, noting that, if considered as~$A_\infty$-categories with one object, we have that~$\End(\{T\})$ and~$\End(T)*G$ are the same.
\end{proof}

Let $\bS$ be a full $G$-invariant arc system on $(\surf,\bM,\eta,g)$, and denote by $\cA:=\cA_\bS$ the corresponding $A_\infty$-category.  Let $\mathbb D$ be a $G$-graded dissection of $(\surf,\bM,\eta,g)$. Then by \cite[Section 3.4]{HaidenKatzarkovKontsevich}, the object corresponding object $\bD$ is a formal generator of the category $\Tw\cA$. This implies that there is an equivalence of triangulated categories 
$$H^0(\Tw\cA) \simeq \per (A_{DG}),$$ where $A=\bigoplus_{\ell\in \bZ} \Hom_{H^0\Tw\cA}(\bD,\bD[\ell])$ is the graded endomorphism algebra of the object $\bD$ seen a differential graded algebra with zero differential.   
Since $\bD$ is $G$-invariant, the algebra $A$ naturally inherits an action of the group $G$. Then we have the following.

\begin{corollary}
In the setup above, the object $\{\bD\}$ in $H^0(\Tw\cA*G)_+$ is a formal generator and we have an equivalence of triangulated categories
$$ H^0(\Tw\cA*G)_+ \simeq \per ((A*G)_{DG}).$$
\end{corollary}
\begin{proof}
 Using the work of~\cite{HaidenKatzarkovKontsevich}, up to Morita-equivalence, we can assume that~$\bD$ is part of the $G$-invariant full arc system~$\bS$.  Then by definition,~$\End{\bD}$ has vanishing~$\mu^1$ and~$\mu^{\geq 3}$.  Thus we have a canonical quasi-equivalence~$H^*\End{\bD}\to \End{\bD}$ which is simply the identity map.  Since it is~$G$-equivariant, we can apply Corollary~\ref{coro::formal-generator}.
\end{proof}

Note that the endormophism algebra of the object $\{\bD\}$ in the category $H^0(\Tw\cA*G)_+$ is not basic, indeed each arc of $\bD$ which is not $G$-invariant give rise to two indecomposable summand in $\{\bD\}$ which are isomorphic. There is no canonical choice of a basic representant of the object $\{\bD\}$. 

Since the algebra $A$ is graded gentle in the sense of \cite{AssemSkowronski}, we immediately obtain the following.

\begin{corollary}
The category $H^0(\Tw\cA*G)_+$ is triangle equivalent to the perfect derived category of a graded skew-gentle algebra.
\end{corollary}

\subsection{Indecomposable objects and tagged curves on orbifold surfaces} 
\label{subs::tilting}

Let $(\surf,\bM,\eta,g)$ be a $G$-graded marked surface, and denote by $\pi:\surf\to \surf/G$  the natural projection. 

\begin{definition}\label{defi::tagged-arc}
 A~\emph{tagged arc} on~$\surf/G$ is an arc~$\gamma$ whose endpoints are either in~$\bM$ or in the set of orbifold points, and which does not cut out a disc containing nothing but an orbifold point.  The~\emph{tagging} is a choice of a sign~$\pm$ at each orbifold endpoint of~$\gamma$. 
 A~\emph{graded tagged arc} is a tagged arc~$\gamma$ with a choice of grading on the preimage of~$\gamma$ on~$\surf$ via~$\pi$.
\end{definition}

\begin{remark}
 If two graded tagged arcs~$\gamma$ and~$\delta$ intersect, we can define the index of the intersection as the index of the lift of this intersection point on~$\surf$.  Note that this applies also in the case where the intersection point is an orbifold point: in that case, the index in one direction is some integer~$p$, while the index in the other is~$1-p$. 
\end{remark}

\begin{figure}[!h]
 
  \scalebox{0.8}{
  \begin{tikzpicture}[scale=1,>=stealth]
  
  \draw[thick, cyan] (-2,0)--(2,0);
  \draw[thick, blue] (-1.5,-1.5)--(1.5,1.5);
  
  \draw[red,->] (0.7,0) arc (0:45:0.7);
  \node[red] at (0.9,0.3) {$p$};
  
  \draw[red,<-] (-0.5,0) arc (180:45:0.5);
  \node[red] at (0,0.7) {$1-p$};
  
   \draw[red,->] (-0.7,0) arc (180:225:0.7);
  \node[red] at (-0.9,-0.3) {$p$};
  
   \draw[red,<-] (0.5,0) arc (0:-135:0.5);
  \node[red] at (0,-0.7) {$1-p$};
  
  \node at (0,-1.5) {$\surf$};
  
  \node[red, inner sep=0pt, fill=white] at (0,0) {$\times$};
  
  \begin{scope}[xshift=4cm]
   \draw[thick, cyan] (0,0)--(2,0);
  \draw[thick, blue] (0,0)--(0,1.5);
  \draw[red,->] (0.7,0) arc (0:90:0.7);
  \draw[red,->] (0,0.5) arc (90:360:0.5);
  \node[red] at (0.9,0.4) {$p$};
  \node[red, inner sep=0pt, fill=white] at (-0.6,-0.3) {$1-p$};
  
   \node[red, inner sep=0pt, fill=white] at (0,0) {$\times$};
    \node at (0,-1.5) {$\surf/G$};
  \end{scope}
  
  \end{tikzpicture}}
  \caption{Index between two graded arcs intersecting at an orbifold point on the orbifold surface $\surf/G$}\label{figure::index-orbifold}
  \end{figure}
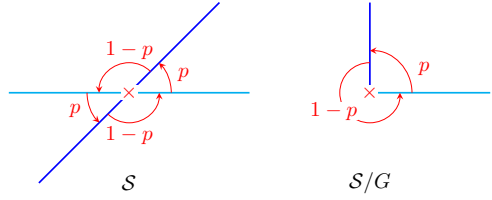

We describe here the indecomposable objects discussed in Section \ref{Section4} in terms of graded tagged arcs on the orbifold $\surf/G$. 

\begin{definition}\label{defi::arcs-to-objects}
 Let~$\gamma$ be a graded simple tagged arc~$\gamma$ on~$\surf/G$.  We exclude the case where both endpoints of~$\gamma$ are the same orbifold point.  We define the object~$\bX(\gamma)$ of~$H^0\Tw(\cA*G)_+$ as follows:
 
 \begin{itemize}
\item if~$\gamma$ is not connected to an orbifold point, then its preimage on~$\surf$ is a union of two (graded) arcs in the same~$G$-orbit.  These are associated to objects~$X$ and~$gX$ in~$\Tw\cA$.  We define~$\bX(\gamma) := \{X\}\oplus \{gX\}$.

\item if~$\gamma$ has one endpoint on the boundary and another on an orbifold point with tag~$\pm$, then its lift to~$\surf$ is a~$G$-invariant curve~$\bar\gamma$ with~$g\tilde{\gamma} = \tilde{\gamma}^{-1}$.  The graded arc~$\gamma$ is associated to a~$G$-invariant object~$X$.  Using the notation of Section~\ref{Section4}, we define~$\bX(\gamma) := X^{\pm}$.

\item if~$\gamma$ has both endpoints on orbifold points~$P$ and~$Q$ with tags~$\pm_P$ and~$\pm_Q$, then its lift is of the form $\alpha.g(\alpha^{-1})$ where $\alpha$ is a simple curve linking two fixed points $P$ and $Q$.  Then we define~$\bX := B^{\pm_P,\pm_Q}$, with the notations of Section~\ref{Section4}.

\end{itemize}
 
\end{definition}

\begin{theorem}\label{theo::local-conditions}
Let $C:=(C_1,\ldots,C_n)$ be a collection of pairwise distinct and non intersecting (except in the orbifold points) tagged simple graded arcs on $\surf/G$.  We assume that none of the~$C_i$ starts and ends at the same orbifold point.  The object~$\bigoplus_{i=1}^n\bX(C_i)$ of~$H^0(\Tw\cA*G)_+$ is rigid if and only if the following three conditions are satisfied    for the graded tagged arcs $C_i=(\gamma_i,\grad_i)$ and $C_j=(\gamma_j,\grad_j)$
\begin{itemize}
\item if for $i\neq j$, $\gamma_i$ is isotopic to $\gamma_j$, (then there exists $p\in\bZ$ with $\grad_i=\grad_j[p]$), then   

\begin{itemize}
\item  either $\gamma_i$ and $\gamma_j$ are arcs linking the boundary to an orbifold point, and the sign attached to $C_i$ is $(-1)^{p-1}$ times the sign attached to $C_j$ ;

\item or they are double tagged arcs linking orbifold points, and on one orbifold point, the sign attached to $C_i$ is $(-1)^p$ times the sign attached to $C_j$, and on the other orbifold point, the sign attached to $C_i$ is $(-1)^{p-1}$ times the sign attached to $C_j$;
\end{itemize}

\item if $\gamma_i$ and $\gamma_j$ are not isotopic and have an endpoint on the same boundary segment, then the corresponding $\bM$-segment has degree $0$ (here $i$ and $j$ may coincide);
\item if $C_i$ and $C_j$ are not isotopic and intersect at an orbifold point, then the signs at this point are the same, and the index at this orbifold point is $0$ and $1$ (cf Figure \ref{figure::index-orbifold});

\end{itemize}
If moreover we have $n= \bar{\mathbf{m}}+\bar{\mathbf{b}}+2\bar{\mathbf{g}}-2+2\mathbf{x}$, then $C$ correspond to a tilting object in the triangulated category $H^0(\Tw\cA*G)_+$. 
\end{theorem}

\begin{proof}
By Propositions \ref{prop:endoX+}, \ref{prop: morphism B epsilon B epsilon'}, each of these objects corresponds to a rigid object. The graded morphisms between $C_i$ and $C_j$ are computed in Propositions \ref{prop: morphism X epsilon Y epsilon}, \ref{prop: morphism B epsilon X epsilon} and \ref{prop: morphism B epsilon B' epsilon}.

Moreover, if $C:=(\gamma,\grad)$ is a $G$-invariant graded arc on $\surf$ corresponding to the objects $C^+$ and $C^-$, then $C^+\oplus C^-[p]$ is rigid for any $p\in \bZ$. The first subitem comes from the fact that $C^-[p]=(\gamma,\grad[p])^{(-1)^{p-1}}$ (cf Remark \ref{rema::shift}). 
The second subitem is analogous since $B^{+_P,+_Q}\oplus B^{-_P,+_Q}[p]$ is rigid for any  $p\in \bZ$.

For the second statement, we can combine the fact that the cardinality of a $G$-invariant dissection is $\mathbf{m}+\mathbf{b}+2\mathbf{g}-2$, and the fact that if $A$ is an algebra with a $G$-action then the rank of the algebra $A*G$ is $\frac{{\rm rk} A+3\mathbf{x}}{2}$ where $\mathbf{x}$ is the number of $G$-invariant projective modules. 
\end{proof}

\begin{remark}
If $X=(\gamma,\grad)$ is a~$G$-invariant arc on $\surf$, then for $p\neq 0$, the object $X^+\oplus X^-[p]$ is rigid, but cannot be completed into a collection of $n= \bar{\mathbf{m}}+\bar{\mathbf{b}}+2\bar{\mathbf{g}}-2+2\mathbf{x}$ arcs as in Theorem~\ref{theo::local-conditions}, and hence into a tilting object.
\end{remark}

\subsection{Tagged dissections on orbifold surfaces}

Using our description of tilting objects from Section~\ref{subs::tilting}, we can use the combinatorics of collections of tagged arcs on an orbifold surface to construct algebras derived-equivalent to graded skew-gentle algebras.

\begin{definition}\label{dissection-orbifold}
 Let~$\surf$ be a surface with an action of~$G$, and let~$\pi:\surf\to \surf/G$ be the projection.  A \emph{tagged dissection} of~$\surf/G$ is a collection~$\Delta$ of tagged arcs on~$\surf/G$ satisfying the following conditions.
 \begin{itemize}
  \item The arcs of~$\Delta$ are simple and are pairwise non-intersecting in the interior of~$\surf/G$.
  \item If~$C_i,C_j\in\Delta$ are distinct but isotopic, then they are either arcs linking the boundary to an orbifold point, and their signs at this point are distinct  or they are double tagged arcs with common signs at one orbifold point and opposite signs at the other.
  \item If~$C_i,C_j\in\Delta$ are not isotopic and meet at an orbifold point, then their signs at this point are equal.
  \item The arcs of~$\Delta$ cut~$\surf/G$ into pieces of three kinds:
   \begin{enumerate}
    \item\label{type1} discs whose sides are all arcs in~$\Delta$ except for one, which is a boundary segment with exactly one unmarked segment;
    \item\label{type2} discs containing exactly one unmarked boundary component and whose sides are all in~$\Delta$;
    \item\label{discsOrbifold} discs with three sides which are all arcs in~$\Delta$, at least one of which is connected to an orbifold point for which all the tagged arcs incident to it have the same sign.
   \end{enumerate}
  \item each orbifold point is the endpoint of at least two tagged arcs.
  \item there is a bijection between the set of orbifold points for which all the tagged arcs incident to it have the same sign and the set of pieces of type (\ref{discsOrbifold}) sending an orbifold point to a piece of which it is part of the boundary.  This bijection is part of the data of a tagged dissection.
 \end{itemize}
\end{definition}

\begin{figure}[!h]
 
  \scalebox{0.6}{
  \begin{tikzpicture}[scale=0.8,>=stealth]
  
  \draw (0,0) circle (6);
  
\node (A1) at (0,6) {$.$};
\node (A2) at (0,-6) {$.$};

\draw[blue, thick] (0,6)--(0,0);
\draw[blue, thick] (0,0)--(2,3);
\draw[blue, thick] (2,3)--(4,0);
\draw[blue, thick] (4,0)--(0,0);
\draw[blue, thick] (0.2,6)--(2,3);
\draw[blue, thick] (4,0)..controls (4.1,-1) and (4.1,-2)..(4,-3);
  \draw[blue, thick] (4,0)..controls (3.9,-1) and (3.9,-2)..(4,-3);

 \draw[blue,thick] (-0.2,-6).. controls (-3,-5) and (-3,5).. (-0.2,6);
 
 \draw[blue, thick] (-0.4,-6)..controls (-2,-6) and  (-4,-1)..(-4,0);
 
 \draw[blue,thick] (-0.1,6)..controls (-2,0) and  (0,-4)..(4,0);
 
 \draw[blue, thick] (-0.4,-6)..controls (-3,-6) and  (-4,-2)..(-4,0);

\draw[red, thick] (0.5,0) arc (0:60:0.5);
\draw[red,thick] (3.5,0) arc (180:225:0.5);
\draw[red, thick] (1.8,3.3) arc (115:220:0.5);

    \draw[dark-green, fill=dark-green] (-0.5,5.95) rectangle(0.5,6.05);
    \draw[dark-green, fill=dark-green] (-0.5,-5.95) rectangle(0.5,-6.05);
    
    \node at (0,-0.5) {$-$};
    \node at (2.5,3) {$+$};
    \node at (4.5,0) {$+$};
    \node at (-4.5,0) {$\pm$};
    \node at (3.5,-3) {$\pm$};
    
    \node[red, fill=white, inner sep=0pt] (A3) at (-4,0) {$\times$};
\node[red, fill=white, inner sep=0pt] (A4) at (0,0) {$\times$};
\node[red, fill=white, inner sep=0pt] (A5) at (4,0) {$\times$};
\node[red, fill=white, inner sep=0pt] (A6) at (2,3) {$\times$};
\node[red, fill=white, inner sep=0pt] (A7) at (4,-3) {$\times$}; 
    
    \end{tikzpicture}}
    \caption{Example of a dissection. The circle arcs around orbifold points illustrate the bijection with pieces of type (\ref{discsOrbifold}).}
    \end{figure}
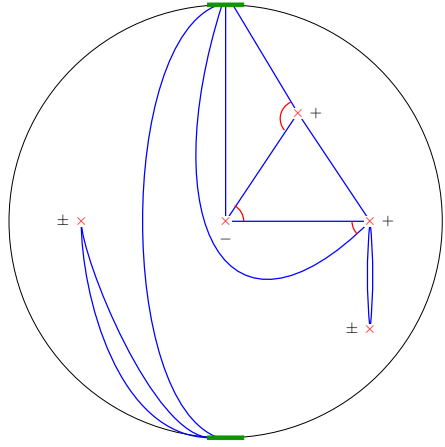

\begin{proposition}
 The number of tagged arcs in a tagged dissection is~$n=\bar{\mathbf{m}} + \bar{\mathbf{b}} + 2\bar{ \mathbf{g}} - 2 + 2\mathbf{x}$.
\end{proposition}
\begin{proof}

 Let~$\Delta$ be a tagged dissection.  If it contains isotopic arcs, we draw them parallel to each other so that they form a digon.  We note that if two isotopic arcs join two orbifold points, then one of these orbifold points has unequal signs, while the other has equal signs; if both had unequal signs, then no arcs could be added to either orbifold points, and the surface cannot be cut into discs.
 
 We turn~$\surf/G$ into a compact surface with marked points by turning all orbifold points into marked points, shrinking each marked segment on the boundary into a marked point, and filling each boundary component with a disc, keeping the boundary edges between marked points (unmarked boundary components thus disappear).  If~$\chi$ is the Euler characteristic of the surface, then~$\chi = 2-2\bar{\mathbf{g}}$.  Moreover, if the surface is cut into discs, then~$\chi$ is the number of faces, minus the number of edges, plus the number of vertices.  Now,~$\Delta$ provides a cutting of the surface into discs, with~$\bar{\mathbf{m}} + \mathbf{x}$ vertices, $n + \bar{\mathbf{m}}$ edges (one per arc in~$\Delta$ and one per boundary edge, which are in bijection with marked points on the boundary) and~$\bar{\mathbf{b}} + \bar{\mathbf{m}} + \mathbf{x}$ faces (one face per ``filled'' marked boundary component, one face containing each of the ``filled'' unmarked boundary components (pieces of type~(\ref{type2})), one face per unmarked boundary segment (pieces of type (\ref{type1}), one triangle per orbifold point with equal signs (pieces of type(\ref{discsOrbifold}), and one digon per orbifold point with unequal signs).  This proves the formula. 
\end{proof}

\begin{proposition}\label{prop::line-field}
 Let~$\Delta$ be a tagged dissection of~$\surf/G$.  There exists a line field~$\eta$ on~$\surf$ and a grading on each of the~$C_i$ such that the resulting graded tagged arcs~$C_1, \ldots, C_n$ of~$\Delta$ satisfy the conditions of Theorem~\ref{theo::local-conditions}.
\end{proposition}
\begin{proof}
 We construct~$\eta$ as follows.  First, let~$\overline{\Delta}$ be the collection of all lifts of arcs in~$\Delta$ via~$\pi:\surf\to\surf/G$ (forgetting the taggings).  If two arcs in~$\Delta$ are isotopic, we only keep the lifts of one of them in~$\overline{\Delta}$.   Then~$\Delta$ is a collection of arcs  and closed curves on~$S$ which may intersect at~$G$-fixed points.  It cuts~$S$ into polygons whose vertices are either on the boundary or are~$G$-fixed points, and whose arcs are portions of the curves in~$\overline{\Delta}$.

 We now build~$\eta$ by defining a line field on each of the polygons, in such a way that the foliation of the line field is transverse to the edges of the polygon.  Thus~$\eta$ is obtained by gluing these foliations together.  
 
 There are two types of polygons.
 
 \noindent\emph{Case 1: polygons containing exaclty one unmarked boundary segment.}.  This corresponds to pieces of types (\ref{type1}) and (\ref{type2}) in Definition~\ref{dissection-orbifold}.  In this case, the foliation is as follows:
 
 \begin{figure}
  \includegraphics[scale=.6]{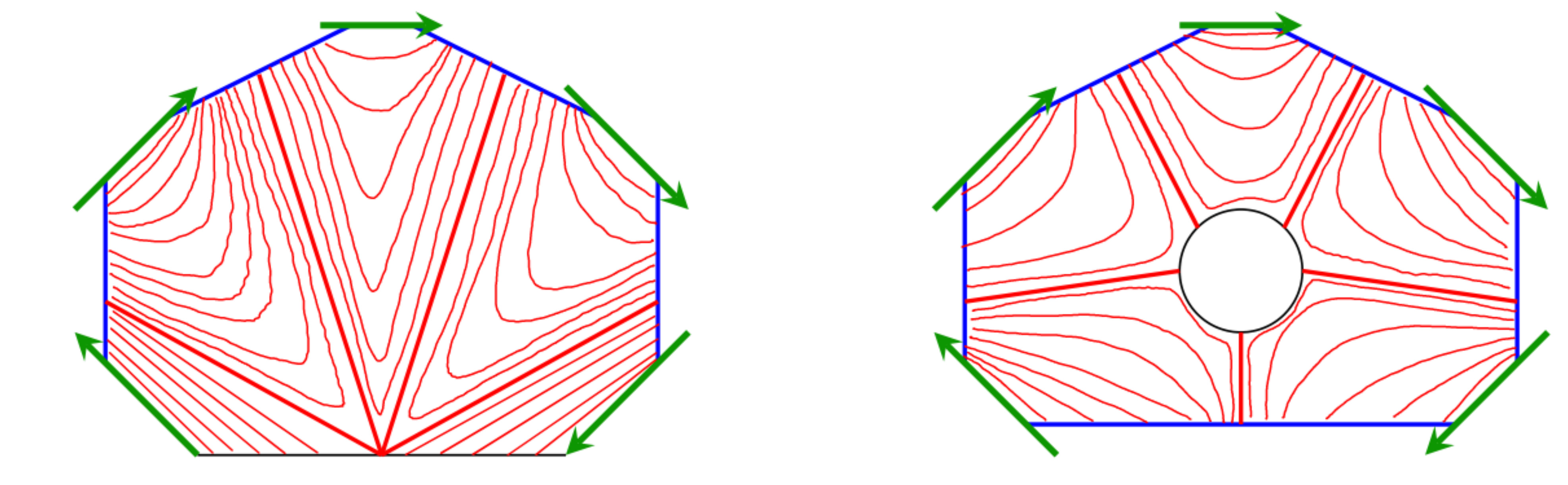}
  \caption{Foliations on polygons of types (\ref{type1}) and (\ref{type2})}
 \end{figure}

 We see that in each case, the~$\mathbb{M}$-path between two adjacent arcs is of degree 0.
 
 \smallskip
 \noindent\emph{Case 2: triangles with at least one orbifold point in their boundary.}  Those triangles are lifts of pieces of type (\ref{discsOrbifold}) in Definition~\ref{dissection-orbifold}.  Such a triangle is associated to one of the orbifold point in its boundary by the bijection in the definition of a dissection.  If we denote by~$\times$ this orbifold point, then the foliation in this triangle is given as follows:
 
 \begin{figure}[!h]
 
  \scalebox{0.8}{
  \begin{tikzpicture}[scale=0.8,>=stealth]
  
  \draw[blue, thick] (-1.5,0.5)--(0,2);
  \draw[blue, thick] (1.5,0.5)--(0,2);
  \draw[blue, thick] (-1.5,0)--(1.5,0);
  
  \draw[red,very thick] (0,0)--(0,2);
  
  \draw[red] (-0.25,1.75)--(-0.25,0);
  \draw[red] (0.25,1.75)--(0.25,0);
  \draw[red] (-0.5,1.5)--(-0.5,0);
  \draw[red] (0.5,1.5)--(0.5,0);
  \draw[red] (-0.75,1.25)--(-0.75,0);
  \draw[red] (0.75,1.25)--(0.75,0);
  \draw[red] (-1,1)--(-1,0);
  \draw[red] (1,1)--(1,0);
  \draw[red] (-1.25,0.75)--(-1.25,0);
  \draw[red] (1.25,0.75)--(1.25,0);
  
  \node[red, inner sep=0pt, fill=white] at (0,2) {$\times$};

  \end{tikzpicture}}
  \caption{Foliation on polygon of type (\ref{discsOrbifold})}
  \end{figure}
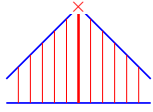

 All these foliations glue together and thus define a line field~$\eta$ on~$\surf$.  Moreover, by contruction and up to homotopy, we can assume that~$\eta$ is~$G$-invariant.  Then the map~$\bX$ of Definition~\ref{defi::arcs-to-objects} sends~$\Delta$ to an object which can be seen to be tilting using Theorem~\ref{theo::local-conditions}.

\end{proof}

We now define an algebra from a tagged dissection.

\begin{definition}\label{defi::algebra-of-dissection}
 Let~$\surf/G$ be an orbifold surface with a tagged dissection~$\Delta$.  We define an algebra~$A(\Delta)$ is several steps:
 \begin{enumerate}
  \item Let~$\eta$ be the line field of Proposition~\ref{prop::line-field}.
  \item Choose a full~$G$-invariant arc system~$\cA$ on~$\surf$.  
  \item By Proposition~\ref{prop::line-field} and Theorem~\ref{theo::local-conditions}, the arcs of~$\Delta$ correspond to the indecomposable summands of a tilting object~$T$ in $H^0\Tw\big((\cA*G)_+\big)$.
  \item Let~$A(\Delta)$ be the endomorphism algebra of~$T$ in~$H^0\Tw\big((\cA*G)_+\big)$.
 \end{enumerate}
\end{definition}

\begin{remark}
 The algebra~$A(\Delta)$ is generally not basic: indeed, any arc joining two boundary segments will correspond to two isomorphic objects~$\{X\}$ and~$\{gX\}$.
\end{remark}

\begin{corollary}\label{coro::derived-equivalence}
 Let~$\surf/G$ be an orbifold surface with a tagged dissection~$\Delta$.  Then~$A(\Delta)$ is derived equivalent to a graded skew-gentle algebra.
\end{corollary}
\begin{proof}
 By construction,~$A(\Delta)$ is the endomorphism algebra of a tilting object in~$H^0\Tw\big((\cA*G)_+\big)$, which is the perfect derived category of a skew-gentle algebra.
\end{proof}

\begin{remark}
 Corollay~\ref{coro::derived-equivalence} implies that the algebras~$A(\Delta)$ are in the derived-equivalence class of the class of graded skew-gentle algebra.  
\end{remark}

\begin{corollary}\label{coro::derived-equivalence2}
 Let~$\Delta$ and~$\Delta'$ be tagged dissections on~$\surf/G$.  If the corresponding line fields~$\eta$ and~$\eta'$ (see Proposition~\ref{prop::line-field}) are homotopic, then~$A(\Delta)$ and~$A(\Delta')$ are derived-equivalent.
\end{corollary}
\begin{proof}
 If~$\cA$ is a full~$G$-invariant arc system on~$\surf$, then~$A(\Delta)$ and~$A(\Delta')$ are the endormorphism algebras of tilging objects in~$H^0\Tw\big((\cA*G)_+\big)$.  Therefore they are derived-equivalent.
\end{proof}

\section{The $\mathbf{D}_n$-case: a complete classification}
\label{Section6}

In this section, we apply the results of Section \ref{Section5} to the case where the $G$-graded marked surface $(\surf, \bM, \eta, g)$ is a disc with $2(n-1)$ marked segment, and the automorphism $g$ is the central symmetry. Denote by $\bS$ a $G$-invariant full graded system on $(\surf, \bM, \eta, g)$, and by $\cA_{\bS}$ the corresponding Fukaya category. Then we have the following.

\begin{proposition}\label{prop:tilting-Dn}
Let $(\surf, \bM, \eta, g)$ and $\bS$ as above. Then we have a  one to one correspondence between graded tagged dissections on $\surf/G$ satisfying conditions of Theorem~\ref{theo::local-conditions} and basic tilting objects of $H^0(\Tw\cA*G)_+$. 
\end{proposition}

\begin{proof}
Any graded tagged dissection satisfying conditions of Theorem~\ref{theo::local-conditions} gives rise to a tilting object. Now let $T$ be a tilting object in $H^0(\Tw\cA*G)_+$. Since the surface $\surf$ is a disc, then the indecomposable objects of $H^0(\Tw\cA)$ are in bijection with the graded simple arcs on $(\surf,\bM,\eta)$. Therefore the indecomposable objects of $H^0(\Tw\cA)$ are in bijection with graded tagged arcs on $\surf/G$. Therefore the tilting object $T$ corresponds to a collection $(C_1,\ldots, C_\ell)$ of pairwise non interesecting (except in the orbifold point) simple graded tagged arcs. We conclude by Theorem \ref{theo::local-conditions}.

\end{proof}

\begin{corollary} Let $(\surf,\bM,g)$ be a disc with $2(n-1)$ marked points and $g$ be the central symmetry. Then for any tagged dissection $\Delta$ on $\surf/G$ the algebra $A(\Delta)$ is an algebra derived equivalent to the path algebra of a quiver $\bar{Q}$ of type $\mathbf{D}_n$. Moreover each algebra derived equivalent to $k\bar{Q}$ is Morita equivalent to an algebra of the form $A(\Delta)$. 
\end{corollary} 

\begin{proof}
By Definition 5.11, the algebra $A(\Delta)$ is the endomorphism algebra of a tilting object $T$ in $H^0(\cA*G)_+$ where $\cA$ is a $G$-invariant full system of graded arc defined on the graded surface $(\surf,\bM,\eta,g)$ where $\eta$ is a well-chosen line field on $\surf$. 

 Hence we have a triangle equivalence $H^0(\Tw\cA *G)_+\simeq \per A(\Delta)$. 

Now let $\mathbf{D}$ be the following $G$-invariant dissection on $\surf$. 

\begin{figure}[!h]
 \[
  \scalebox{1}{
  \begin{tikzpicture}[scale=2,>=stealth]
  
  \draw (0,0) circle (1);
  
  \draw[very thick, dark-green] (1,-0.1)--(1,0.1);
  
  \draw[very thick, dark-green] (0.8,0.6)--(0.6,0.8);
  
    \draw[very thick, dark-green] (-0.8,0.6)--(-0.6,0.8);
      \draw[very thick, dark-green] (0.8,-0.6)--(0.6,-0.8);
        \draw[very thick, dark-green] (-0.8,-0.6)--(-0.6,-0.8);
  
  \draw[very thick, dark-green] (-0.1,1)--(0.1,1);
  
  \draw[very thick, dark-green] (-1,-0.1)--(-1,0.1);
  
   \draw[very thick, dark-green] (-0.1,-1)--(0.1,-1);
  
  \draw[blue] (-1,0)--(1,0);
  \node[blue, inner sep=0pt, fill=white] at (0.5,0) {$X_1$};
  \draw[blue] (-1,0.03)--node[inner sep=0pt, fill=white]{$X_2$}(0.7,0.7);
  \draw[blue] (-1,0.06)--node[inner sep=0pt, fill=white]{$X_3$}(0,1);
  \draw[blue] (-1,0.09)--node[inner sep=0pt, fill=white]{$X_4$}(-0.7,0.7);
   
  \draw[blue] (1,-0.03)--node[inner sep=0pt, fill=white]{$gX_2$}(-0.7,-0.7);
  \draw[blue] (1,-0.06)--node[inner sep=0pt, fill=white]{$gX_3$}(0,-1);
  \draw[blue] (1,-0.09)--node[inner sep=0pt, fill=white]{$gX_4$}(0.7,-0.7);
  
  \node[red, fill=white, inner sep=0pt] at (0,0) {$\times$};

  \end{tikzpicture}}\]
  \caption{}
\end{figure}

Since the surface is a disc, there is only one line field up to homotopy, so it is possible to choose the grading on each arc in such a way that the graded algebra corresponding to the dissection $\mathbb D$ is concentrated in degree $0$. The endomorphism algebra of the object $\mathbb D$ in $H^0\Tw\cA$ is isomorphic to the path algebra of the quiver $Q$ of type $\mathbb A_{2n+1}$ with symmetric orientation.  

 \[
  \scalebox{1}{
  \begin{tikzpicture}[scale=1,>=stealth]
\node (1) at (0,0) {$X_1$};
\node (2) at (1,1) {$X_2$};
\node (3) at (2.5,1) {$X_3$};
\node (4) at (4,1) {$X_4$};
\node (g2) at (1,-1) {$gX_2$};
\node (g3) at (2.5,-1) {$gX_3$};
\node (g4) at (4,-1) {$gX_4$};

\draw[->] (1)--(2);
\draw[->] (2)--(3);
\draw[->] (3)--(4);
\draw[->] (1)--(g2);
\draw[->] (g2)--(g3);
\draw[->] (g3)--(g4);

  \end{tikzpicture}}\]

The corresponding formal generator $\{\mathbb D\}$ has its endomorphism algebra concentrated in degree $0$ and is Morita equivalent to the path algebra of a quiver $\bar{Q}$ of type $\mathbb D_n$.  

 \[
  \scalebox{1}{
  \begin{tikzpicture}[scale=1,>=stealth]
  
  \node (1+) at (0,1) {$X_1^+$};
   \node (1-) at (0,-1) {$X_1^-$};
   \node (2) at (1,0) {$\{X_2\}$};
   \node (3) at (2.5,0) {$\{X_3\}$};
   \node (4) at (4,0) {$\{X_4\}$};
   
   \draw[->] (1+)--(2);
    \draw[->] (1-)--(2);
     \draw[->] (2)--(3);
      \draw[->] (3)--(4);
      \end{tikzpicture}}\]

Hence we have a triangle equivalence $H^0(\Tw\cA*G)_+\simeq \per k\bar{Q}$. 

Conversely, if $A$ is an algebra derived equivalent to $k\bar{Q}$, then by \cite{Rickard} it is the endomorphism algebra of a tilting object $T$ in $H^0(\Tw\cA*G)_+$. We conclude by Proposition \ref{prop:tilting-Dn}. 

\end{proof}

\begin{example}
Here we give some examples of algebras which are derived equivalent to the path algebra of a quiver of type $\mathbf{D}_5$. In this case, $\surf$ is a disc with $8$ marked points, and $S/G$ is a disc with $4$ marked points and one orbifold point.

Let $\Delta$ be a tagged dissection of $\surf/G$. If $\Delta$ contains two isotopic tagged arcs with opposite signs, then it does not contain any other tagged arc with endpoint in the orbifold point. In that case, the endomorphism algebra $A(\Delta)$ is skew-gentle. 

We focus here on the case where $\Delta$ does not contain two isotopic arcs with opposite signs, since the skew-gentle algebras have been already heavily studied. In that case, since there is only one orbifold point, there is exactly one piece of type (3) in the dissection, and no choice for the bijection between pieces of type (3) and orbifold points. 

For example, consider the following tagged dissection of $\surf/G$. 

\begin{figure}[!h]
 \[
  \scalebox{1}{
  \begin{tikzpicture}[scale=1.5,>=stealth]
  
  \draw (0,0) circle (1);
  
  \draw[very thick, dark-green] (1,-0.1)--(1,0.1);
  
  \draw[very thick, dark-green] (-0.1,1)--(0.1,1);
  
  \draw[very thick, dark-green] (-1,-0.1)--(-1,0.1);
  
   \draw[very thick, dark-green] (-0.1,-1)--(0.1,-1);
   
   \draw[blue] (-1,0)--node[inner sep=0pt, fill=white]{$1$}(0,0);
   \draw[blue] (0,0)--node[inner sep=0pt, fill=white]{$2$}(0,1);
   \draw[blue] (-1,0.05)--node[inner sep=0pt, fill=white]{$3$}(-0.05,1);
   \draw[blue] (0.05,1).. controls (0.3,0.8) and (0.3,-0.8)..node[inner sep=0pt, fill=white]{$4$} (0.05,-1);
   \draw[blue] (0.09,-1)--node[inner sep=0pt, fill=white]{$5$}(1,0);
   
   \node[red, inner sep=0pt, fill=white] at (0,0) {$\times$};
   
    \node[inner sep=0pt, fill=white] at (-0.2,-0.1) {$+$};
     \node[inner sep=0pt, fill=white] at (0.1,0.2) {$+$};

  \end{tikzpicture}}\]
  \caption{}
\end{figure}
  
  It corresponds to the following collection of graded arcs in the covering of the orbifold (the indices between the objects $X_1$ and $X_2$ are indicated in red):
  
\begin{figure}[!h]
   \[
  \scalebox{0.8}{
  \begin{tikzpicture}[scale=2.2,>=stealth]
  
  \draw (0,0) circle (1);
  
  \draw[very thick, dark-green] (1,-0.1)--(1,0.1);
  
  \draw[very thick, dark-green] (0.8,0.6)--(0.6,0.8);
  
    \draw[very thick, dark-green] (-0.8,0.6)--(-0.6,0.8);
      \draw[very thick, dark-green] (0.8,-0.6)--(0.6,-0.8);
        \draw[very thick, dark-green] (-0.8,-0.6)--(-0.6,-0.8);
  
  \draw[very thick, dark-green] (-0.1,1)--(0.1,1);
  
  \draw[very thick, dark-green] (-1,-0.1)--(-1,0.1);
  
   \draw[very thick, dark-green] (-0.1,-1)--(0.1,-1);
  
  \draw[blue] (-1,0)--(1,0);
  \node[blue, inner sep=0pt, fill=white] at (0.5,0) {$X_1$};
  \draw[blue] (-0.7,0.7)--(0.7,-0.7);
  \node[blue, inner sep=0pt, fill=white] at (-0.5,0.5) {$X_2$};
  
  \draw[blue] (-1,0.05)-- node[inner sep=0pt, fill=white]{$X_3$} (-0.75,0.65);
  
  \draw[blue] (-0.65,0.75)--node[inner sep=0pt, fill=white]{$X_4$}(0.7,0.7);

  \draw[blue] (0,1)--node[inner sep=0pt, fill=white]{$X_5$}(0.65,0.75);
  
    \draw[blue] (1,-0.05)-- node[inner sep=0pt, fill=white]{$gX_3$} (0.75,-0.65);
  
  \draw[blue] (0.65,-0.75)--node[inner sep=0pt, fill=white]{$gX_4$}(-0.7,-0.7);

  \draw[blue] (0,-1)--node[inner sep=0pt, fill=white]{$gX_5$}(-0.65,-0.75);

  \node[red, fill=white, inner sep=0pt] at (0,0) {$\times$};
  
  \draw[red,->] (0.2,0) arc (0:135:0.2);
  \node[red, inner sep=0pt, fill=white] at (0,0.3) {$0$};
  \draw[red,<-] (-0.2,0) arc (180:135:0.2);
  \node[red, inner sep=0pt, fill=white] at (-0.3,0.1) {$1$};

  \end{tikzpicture}}\]
  \caption{}
\end{figure}
  
  One can compute the endomorphism algebra of the object $X_1^+\oplus X_2^+\oplus \{X_3\}\oplus \{X_4\}\oplus \{X_5\}$ in $H^0(\Tw\cA*G)_+$ and one obtains  the following quiver with relation (the dotted line indicates that the composition from $X_1^+$ to $X_4$ vanishes).

  \[
  \scalebox{0.8}{
  \begin{tikzpicture}[scale=1.5,>=stealth]  
  \node (1) at (0,0) {$X_1^+$};
  \node (2) at (1.5,0) {$X_3$};
  \node (3) at (3,0) {$X_2^+$};
  \node (4) at (4.5,0) {$X_4$};
  \node (5) at (6,0) {$X_5$};
  
  \draw[->] (1)--(2);
  \draw[->] (2)--(3);
  \draw[->] (3)--(4);
  \draw[->] (5)--(4);
  
  \draw[thick, loosely dotted] (1,0).. controls (1.5,0.5) and (3,0.5)..(3.5,0);
  \end{tikzpicture}}
  \]
Note that by Proposition \ref{prop: morphism X epsilon Y epsilon}, there is a non zero morphism from $X_1^+$ to $X_2^+$. On can easily check that this morphism factors through the object $\{X_3\}$.

In a similar way, the algebra $A(\Delta)$ corresponding to the following tagged dissection 

\begin{figure}[!h]
 \[
  \scalebox{1}{
  \begin{tikzpicture}[scale=1.5,>=stealth]
  
  \draw (0,0) circle (1);
  
  \draw[very thick, dark-green] (1,-0.1)--(1,0.1);
  
  \draw[very thick, dark-green] (-0.1,1)--(0.1,1);
  
  \draw[very thick, dark-green] (-1,-0.1)--(-1,0.1);
  
   \draw[very thick, dark-green] (-0.1,-1)--(0.1,-1);
   
   \draw[blue] (-1,0)--node[inner sep=0pt, fill=white]{$1$}(0,0);
   \draw[blue] (0,0)--node[inner sep=0pt, fill=white]{$2$}(0,1);
   \draw[blue] (-1,0.05)--node[inner sep=0pt, fill=white]{$3$}(-0.05,1);
   \draw[blue] (0,0)--node[inner sep=0pt, fill=white]{$4$} (1,0);
   \draw[blue] (0,0)--node[inner sep=0pt, fill=white]{$5$}(0,-1);
   
   \node[red, inner sep=0pt, fill=white] at (0,0) {$\times$};
   
    \node[inner sep=0pt, fill=white] at (-0.2,-0.1) {$+$};
    \node[inner sep=0pt, fill=white] at (0.1,-0.2) {$+$};
    \node[inner sep=0pt, fill=white] at (0.2,0.1) {$+$};
    \node[inner sep=0pt, fill=white] at (-0.1,0.2) {$+$};

  \end{tikzpicture}}\]
  \caption{}
\end{figure}
  
 gives rise to the following algebra (where the dotted line indicates a commutativity relation).
 
   \[
  \scalebox{0.8}{
  \begin{tikzpicture}[scale=1.5,>=stealth]  
  \node (1) at (0,0) {$X_1^+$};
  \node (2) at (3,0) {$X_2^+$};
  \node (3) at (1.5,-1) {$\{X_3\}$};
  \node (5) at (1,1) {$X_5^+$};
  \node (4) at (2,1) {$X_4^+$};
  
  \draw[->] (1)--(3);
  \draw[->] (3)--(2);
  \draw[->] (1)--(5);
  \draw[->] (5)--(4);
  \draw[->] (4)--(2);
  
  \draw[thick, loosely dotted] (1)--(2);
  \end{tikzpicture}}
  \]

This is (up to cyclic permutation, and change of tagging) the only tagged dissection containig four tagged arcs. In a similar way we can list all the algebras that are derived equivalent to the path algebra of a quiver of type $\mathbf{D}_5$ and which are not skew-gentle. There are 25.

There are $6$ corresponding to a dissection containing exactly $3$ tagged arcs.

\begin{figure}[!h]
 \[
  \scalebox{1}{
  \begin{tikzpicture}[scale=1.5,>=stealth]
  
  \draw (0,0) circle (1);
  
  \draw[very thick, dark-green] (1,-0.1)--(1,0.1);
  
  \draw[very thick, dark-green] (-0.1,1)--(0.1,1);
  
  \draw[very thick, dark-green] (-1,-0.1)--(-1,0.1);
  
   \draw[very thick, dark-green] (-0.1,-1)--(0.1,-1);
   
   \draw[blue] (-1,0)--(0,0);
   \draw[blue] (0,0)--(0,1);
  
   \draw[blue] (0,0)-- (1,0);

   \node[red, inner sep=0pt, fill=white] at (0,0) {$\times$};
   
    \node[inner sep=0pt, fill=white] at (-0.2,-0.1) {$+$};
    \node[inner sep=0pt, fill=white] at (-0.1,0.2) {$+$};
    \node[inner sep=0pt, fill=white] at (0.2,0.1) {$+$};
  
  \end{tikzpicture}}\]

 \[
  \scalebox{0.8}{
  \begin{tikzpicture}[scale=0.8,>=stealth]
\begin{scope}
\node (1) at (0,0) {$+$};
\node (2) at (1,1) {$.$};
\node (3) at (2,0) {$+$};
\node (4) at (1,-1) {$+$};
\node (5) at (2.5,1) {$.$};

\draw[->] (1)--(2);
\draw[->] (2)--(3);
\draw[->] (1)--(4);
\draw[->] (4)--(3);
\draw[thick, loosely dotted] (1)--(3);
\draw[->] (2)--(5);
  \end{scope}
  
  \begin{scope}[xshift=5cm]
\node (1) at (0,0) {$+$};
\node (2) at (1,1) {$.$};
\node (3) at (2,0) {$+$};
\node (4) at (1,-1) {$+$};
\node (5) at (-0.5,1) {$.$};

\draw[->] (1)--(2);
\draw[->] (2)--(3);
\draw[->] (1)--(4);
\draw[->] (4)--(3);
\draw[thick, loosely dotted] (1)--(3);
\draw[->] (5)--(2);
  \end{scope}
  
    \begin{scope}[xshift=10cm]
\node (1) at (0,0) {$+$};
\node (2) at (1,1) {$.$};
\node (3) at (2,0) {$+$};
\node (4) at (1,-1) {$+$};
\node (5) at (-0.5,-1) {$.$};

\draw[->] (1)--(2);
\draw[->] (2)--(3);
\draw[->] (1)--(4);
\draw[->] (4)--(3);
\draw[thick, loosely dotted] (1)--(3);
\draw[->] (5)--(4);

\draw[loosely dotted, thick] (0,-1)..controls (0.5,-1.5) and (1.5,-1.5)..(1.5,-0.5);
  \end{scope}
  
      \begin{scope}[yshift=-3cm]
\node (1) at (0,0) {$+$};
\node (2) at (1,1) {$.$};
\node (3) at (2,0) {$+$};
\node (4) at (1,-1) {$+$};
\node (5) at (2.5,-1) {$.$};

\draw[->] (1)--(2);
\draw[->] (2)--(3);
\draw[->] (1)--(4);
\draw[->] (4)--(3);
\draw[thick, loosely dotted] (1)--(3);
\draw[->] (4)--(5);

\draw[thick, loosely dotted] (0.5,-0.5).. controls (0.5,-1.5) and (1.5,-1.5).. (2,-1);
  \end{scope}

       \begin{scope}[yshift=-3cm, xshift=5cm]
\node (1) at (0,0) {$+$};
\node (2) at (1,1) {$.$};
\node (3) at (2,0) {$+$};
\node (4) at (1,-1) {$+$};
\node (5) at (-1.5,0) {$.$};

\draw[->] (1)--(2);
\draw[->] (2)--(3);
\draw[->] (1)--(4);
\draw[->] (4)--(3);
\draw[thick, loosely dotted] (1)--(3);
\draw[->] (5)--(1);

\draw[thick, loosely dotted] (-1,0).. controls (-0.5,-0.5) and (0,-0.5)..(0.5,-0.5);
  \end{scope}

      \begin{scope}[yshift=-3cm, xshift=10cm]
\node (1) at (0,0) {$+$};
\node (2) at (1,1) {$.$};
\node (3) at (2,0) {$+$};
\node (4) at (1,-1) {$+$};
\node (5) at (3.5,0) {$.$};

\draw[->] (1)--(2);
\draw[->] (2)--(3);
\draw[->] (1)--(4);
\draw[->] (4)--(3);
\draw[thick, loosely dotted] (1)--(3);
\draw[->] (3)--(5);

\draw[thick, loosely dotted] (3,0).. controls (2.5,-0.5) and (2,-0.5)..(1.5,-0.5);

  \end{scope}

  \end{tikzpicture}}\]
\caption{}
\end{figure}

There are exactly $4$ tagged dissections the following three arcs, and no other tagged arc.

\begin{figure}[!h]
\[
  \scalebox{1}{
  \begin{tikzpicture}[scale=1.5,>=stealth]
  
  \draw (0,0) circle (1);
  
  \draw[very thick, dark-green] (1,-0.1)--(1,0.1);
  
  \draw[very thick, dark-green] (-0.1,1)--(0.1,1);
  
  \draw[very thick, dark-green] (-1,-0.1)--(-1,0.1);
  
   \draw[very thick, dark-green] (-0.1,-1)--(0.1,-1);
   
   \draw[blue] (-1,0)--(0,0);
   \draw[blue] (0,0)--(1,0);
  
 \draw[blue] (-1,0).. controls (-0.5,0.5) and (0.5,0.5)..node[inner sep=0pt, fill=white]{$\circ$}(1,0);
   
   \node[red, inner sep=0pt, fill=white] at (0,0) {$\times$};
   
    \node[inner sep=0pt, fill=white] at (-0.2,-0.1) {$+$};

    \node[inner sep=0pt, fill=white] at (0.2,0.1) {$+$};
  
  \end{tikzpicture}}\]
  \caption{}
\end{figure}

  The corresponding $4$ algebras are the following
  
 \[
  \scalebox{0.8}{
  \begin{tikzpicture}[scale=0.8,>=stealth]
\begin{scope}
\node (1) at (0,0) {$.$};
\node (2) at (1.5,0) {$+$};
\node (3) at (3,0) {$\circ$};
\node (4) at (4.5,0) {$+$};
\node (5) at (3,-1) {$.$};

\draw[->] (1)--(2);
\draw[->] (2)--(3);
\draw[->] (3)--(4);
\draw[->] (3)--(5);

\draw[thick, loosely dotted] (0.5,0)..controls (1,0.5) and (3.5,0.5)..(4,0);
  
  \end{scope}
  
  \begin{scope}[xshift=6cm]
\node (1) at (0,0) {$.$};
\node (2) at (1.5,0) {$+$};
\node (3) at (3,0) {$\circ$};
\node (4) at (4.5,0) {$+$};
\node (5) at (3,-1) {$.$};

\draw[->] (1)--(2);
\draw[->] (2)--(3);
\draw[->] (3)--(4);
\draw[->] (5)--(3);

\draw[thick, loosely dotted] (0.5,0)..controls (1,0.5) and (3.5,0.5)..(4,0);
  
  \end{scope}
  
  \begin{scope}[yshift=-3cm]
\node (1) at (0,0) {$+$};
\node (2) at (1.5,0) {$\circ$};
\node (3) at (3,0) {$+$};
\node (4) at (4.5,0) {$.$};
\node (5) at (1.5,-1) {$.$};

\draw[->] (1)--(2);
\draw[->] (2)--(3);
\draw[->] (3)--(4);
\draw[->] (2)--(5);

\draw[thick, loosely dotted] (0.5,0)..controls (1,0.5) and (3.5,0.5)..(4,0);
  
  \end{scope}
  
  \begin{scope}[xshift=6cm,yshift=-3cm]
\node (1) at (0,0) {$+$};
\node (2) at (1.5,0) {$\circ$};
\node (3) at (3,0) {$+$};
\node (4) at (4.5,0) {$.$};
\node (5) at (1.5,-1) {$.$};

\draw[->] (1)--(2);
\draw[->] (2)--(3);
\draw[->] (3)--(4);
\draw[->] (5)--(2);

\draw[thick, loosely dotted] (0.5,0)..controls (1,0.5) and (3.5,0.5)..(4,0);
  
  \end{scope}

 \end{tikzpicture}}\] 
  
There are $7$ dissections containing the following three arcs and no other tagged arcs.

\begin{figure}[!h]
\[
  \scalebox{1}{
  \begin{tikzpicture}[scale=1.5,>=stealth]
  
  \draw (0,0) circle (1);
  
  \draw[very thick, dark-green] (1,-0.1)--(1,0.1);
  
  \draw[very thick, dark-green] (-0.1,1)--(0.1,1);
  
  \draw[very thick, dark-green] (-1,-0.1)--(-1,0.1);
  
   \draw[very thick, dark-green] (-0.1,-1)--(0.1,-1);
   
   \draw[blue] (0,1)--(0,0);
   \draw[blue] (0,0)--(1,0);
  
 \draw[blue] (1,0).. controls (0,-0.75) and (-0.75,0)..node[inner sep=0pt, fill=white]{$\circ$}(0,1);
   
   \node[red, inner sep=0pt, fill=white] at (0,0) {$\times$};

  \end{tikzpicture}}\]
  \caption{}
\end{figure}
  
 The $7$ corresponding algebras are as follows :
 
  \[
  \scalebox{0.8}{
  \begin{tikzpicture}[scale=0.8,>=stealth]
\begin{scope}
\node (1) at (0,0) {$+$};
\node (2) at (1.5,0) {$\circ$};
\node (3) at (3,0) {$.$};
\node (4) at (4.5,0) {$.$};
\node (5) at (1.5,-1) {$+$};

\draw[->] (1)--(2);
\draw[->] (2)--(3);
\draw[->] (3)--(4);
\draw[->] (2)--(5);

\draw[thick, loosely dotted] (2,0)..controls (2.5,0.5) and (3.5,0.5)..(4,0);
  
  \end{scope}

\begin{scope}[yshift=-3cm]
\node (1) at (0,0) {$+$};
\node (2) at (1.5,0) {$\circ$};
\node (3) at (3,0) {$.$};
\node (4) at (4.5,0) {$.$};
\node (5) at (1.5,-1) {$+$};

\draw[->] (1)--(2);
\draw[->] (2)--(3);
\draw[->] (3)--(4);
\draw[->] (2)--(5);
  
  \end{scope}  
  
  \begin{scope}[xshift=6cm]
\node (1) at (0,0) {$+$};
\node (2) at (1.5,0) {$\circ$};
\node (3) at (3,0) {$.$};
\node (4) at (4.5,0) {$.$};
\node (5) at (1.5,-1) {$+$};

\draw[->] (1)--(2);
\draw[->] (3)--(2);
\draw[->] (4)--(3);
\draw[->] (2)--(5);

\draw[thick, loosely dotted] (2,0)..controls (2.5,0.5) and (3.5,0.5)..(4,0);
  
  \end{scope}
  
    \begin{scope}[xshift=6cm,yshift=-3cm]
\node (1) at (0,0) {$+$};
\node (2) at (1.5,0) {$\circ$};
\node (3) at (3,0) {$.$};
\node (4) at (4.5,0) {$.$};
\node (5) at (1.5,-1) {$+$};

\draw[->] (1)--(2);
\draw[->] (3)--(2);
\draw[->] (4)--(3);
\draw[->] (2)--(5);
  
  \end{scope}
  
      \begin{scope}[xshift=12cm]
\node (1) at (0,0) {$+$};
\node (2) at (1.5,0) {$\circ$};
\node (3) at (3,0) {$.$};
\node (4) at (4.5,0) {$.$};
\node (5) at (1.5,-1) {$+$};

\draw[->] (1)--(2);
\draw[->] (2)--(3);
\draw[->] (4)--(3);
\draw[->] (2)--(5);
  
  \end{scope}
      \begin{scope}[xshift=12cm,yshift=-3cm]
\node (1) at (0,0) {$+$};
\node (2) at (1.5,0) {$\circ$};
\node (3) at (3,0) {$.$};
\node (4) at (4.5,0) {$.$};
\node (5) at (1.5,-1) {$+$};

\draw[->] (1)--(2);
\draw[->] (3)--(2);
\draw[->] (3)--(4);
\draw[->] (2)--(5);
  
  \end{scope}
  
      \begin{scope}[xshift=18cm,yshift=-1.5cm]
\node (1) at (0,0) {$+$};
\node (2) at (1.5,0) {$\circ$};
\node (3) at (1.5,1.5) {$.$};
\node (4) at (3,0) {$.$};
\node (5) at (1.5,-1.5) {$+$};

\draw[->] (1)--(2);
\draw[->] (2)--(3);
\draw[->] (4)--(2);
\draw[->] (2)--(5);

\draw[thick, loosely dotted] (1.5,1)..controls (2,1) and (1,1)..(2.5,0);
  
  \end{scope}

  \end{tikzpicture}}\]

  Finally there are 7 dissections containing the following three arcs and no other tagged arcs.
  
\begin{figure}[!h]
  \[
  \scalebox{1}{
  \begin{tikzpicture}[scale=1.5,>=stealth]
  
  \draw (0,0) circle (1);
  
  \draw[very thick, dark-green] (1,-0.1)--(1,0.1);
  
  \draw[very thick, dark-green] (-0.1,1)--(0.1,1);
  
  \draw[very thick, dark-green] (-1,-0.1)--(-1,0.1);
  
   \draw[very thick, dark-green] (-0.1,-1)--(0.1,-1);
   
   \draw[blue] (0,1)--(0,0);
   \draw[blue] (0,0)--(1,0);
  
 \draw[blue] (1,0)--node[inner sep=0pt, fill=white]{$\circ$}(0,1);
   
   \node[red, inner sep=0pt, fill=white] at (0,0) {$\times$};

  \end{tikzpicture}}\]
  \caption{}
\end{figure}
  
  The corresponding algebras are as follows 
  
  \[
  \scalebox{0.8}{
  \begin{tikzpicture}[scale=0.8,>=stealth]
\begin{scope}
\node (1) at (0,0) {$+$};
\node (2) at (1.5,0) {$\circ$};
\node (3) at (3,0) {$+$};
\node (4) at (4.5,0) {$.$};
\node (5) at (6,0) {$.$};

\draw[->] (1)--(2);
\draw[->] (2)--(3);
\draw[->] (3)--(4);
\draw[->] (4)--(5);

\draw[thick, loosely dotted] (0.5,0)..controls (1,0.5) and (3,0.5)..(4,0);

\draw[thick, loosely dotted] (4,0)..controls (4.5,0.5) and (5,0.5)..(5.5,0);

  \end{scope}
  
  \begin{scope}[xshift=8cm]
\node (1) at (0,0) {$+$};
\node (2) at (1.5,0) {$\circ$};
\node (3) at (3,0) {$+$};
\node (4) at (4.5,0) {$.$};
\node (5) at (6,0) {$.$};

\draw[->] (1)--(2);
\draw[->] (2)--(3);
\draw[->] (3)--(4);
\draw[->] (4)--(5);

\draw[thick, loosely dotted] (0.5,0)..controls (1,0.5) and (3,0.5)..(4,0);

  \end{scope}

  \begin{scope}[yshift=-3cm]
\node (1) at (0,0) {$+$};
\node (2) at (1.5,0) {$\circ$};
\node (3) at (3,0) {$+$};
\node (4) at (4.5,0) {$.$};
\node (5) at (6,0) {$.$};

\draw[<-] (1)--(2);
\draw[<-] (2)--(3);
\draw[<-] (3)--(4);
\draw[<-] (4)--(5);

\draw[thick, loosely dotted] (0.5,0)..controls (1,0.5) and (3,0.5)..(4,0);

\draw[thick, loosely dotted] (4,0)..controls (4.5,0.5) and (5,0.5)..(5.5,0);

  \end{scope}
  \begin{scope}[xshift=8cm,yshift=-3cm]
\node (1) at (0,0) {$+$};
\node (2) at (1.5,0) {$\circ$};
\node (3) at (3,0) {$+$};
\node (4) at (4.5,0) {$.$};
\node (5) at (6,0) {$.$};

\draw[<-] (1)--(2);
\draw[<-] (2)--(3);
\draw[<-] (3)--(4);
\draw[<-] (4)--(5);

\draw[thick, loosely dotted] (0.5,0)..controls (1,0.5) and (3,0.5)..(4,0);
 
  \end{scope}

  \begin{scope}[yshift=-6cm]
\node (1) at (0,0) {$+$};
\node (2) at (1.5,0) {$\circ$};
\node (3) at (3,0) {$+$};
\node (4) at (4.5,0) {$.$};
\node (5) at (6,0) {$.$};

\draw[->] (1)--(2);
\draw[->] (2)--(3);
\draw[->] (3)--(4);
\draw[->] (5)--(4);

\draw[thick, loosely dotted] (0.5,0)..controls (1,0.5) and (3,0.5)..(4,0);

  \end{scope}

    \begin{scope}[xshift=8cm,yshift=-6cm]
\node (1) at (0,0) {$+$};
\node (2) at (1.5,0) {$\circ$};
\node (3) at (3,0) {$+$};
\node (4) at (4.5,0) {$.$};
\node (5) at (6,0) {$.$};

\draw[<-] (1)--(2);
\draw[<-] (2)--(3);
\draw[<-] (3)--(4);
\draw[<-] (5)--(4);

\draw[thick, loosely dotted] (0.5,0)..controls (1,0.5) and (3,0.5)..(4,0);

  \end{scope}
  
    \begin{scope}[yshift=-9cm,xshift=4cm]
\node (1) at (0,0) {$.$};
\node (2) at (1.5,0) {$+$};
\node (3) at (3,0) {$\circ$};
\node (4) at (4.5,0) {$+$};
\node (5) at (6,0) {$.$};

\draw[->] (1)--(2);
\draw[->] (2)--(3);
\draw[->] (3)--(4);
\draw[->] (4)--(5);

\draw[thick, loosely dotted] (0.5,0)..controls (1,0.5) and (3,0.5)..(4,0);
\draw[thick, loosely dotted] (2,0)..controls (3.5,-0.5) and (4.5,-0.5)..(5.5,0);

  \end{scope}
  
  \end{tikzpicture}}\]  

\end{example}

\section*{Acknowledgements}
Much of this work was completed while we were visiting the Centre for Advanced Study in Oslo during the project \emph{Representation Theory: Combinatorial Aspects and Applications}.  We are grateful to the organizers of the project, Aslak Buan and Steffen Oppermann, for inviting us.  We acknowledge the support of the French ANR grant CHARMS (ANR-19-CE40-0017-02) and of the Institut Universitaire de France (IUF).  Some of the diagrams where produced using~\url{q.uiver.app}.  We thank an anonymous referee for their very helpful comments.

\bibliographystyle{alpha}
\bibliography{orbifold.bib}

\end{document}